\documentclass[12pt]{amsart}

\usepackage{amssymb}
\usepackage{amsthm}
\theoremstyle{plain}

\usepackage{overpic}

\usepackage{tikz}

\usetikzlibrary{arrows,backgrounds,decorations}
\usetikzlibrary{positioning,automata}

\usepackage{fancyhdr,enumerate}

\newtheorem{thm}{Theorem}[section]
\newtheorem{cor}[thm]{Corollary}
\newtheorem{lem}[thm]{Lemma}
\newtheorem{prop}[thm]{Proposition}

\theoremstyle{definition}
\newtheorem{dfn}[thm]{Definition}
\newtheorem*{SN*}{Standing Notation}

\newtheorem*{SA*}{Standing Assumption}

\newtheorem{ntn}[thm]{Notation}

\theoremstyle{remark}
\newtheorem{rmk}[thm]{Remark}

\newtheorem{example}[thm]{Example}
\newtheorem{examples}[thm]{Examples}

\numberwithin{equation}{section}

\newcommand{\C}{\mathbb{C}}

\newcommand{\N}{\mathbb{N}}
\newcommand{\Q}{\mathbb{Q}}
\newcommand{\T}{\mathbb{T}}
\newcommand{\Z}{\mathbb{Z}}

\newcommand{\GG}{\mathcal{G}}

\newcommand{\KK}{\mathcal{K}}

\newcommand{\UU}{\mathcal{U}}

\newcommand{\G}{\Gamma}

\def\id{\operatorname{id}}

\def\Aut{\operatorname{Aut}}
\def\Stab{\operatorname{Stab}}
\def\clsp{\operatorname{\overline{span}}}
\def\sp{\operatorname{span}}
\def\BS{\operatorname{BS}}
\def\SL{\operatorname{SL}}

\begin{document}

\title[$C^*$-algebras associated to graphs of groups]{$C^*$-algebras 
associated to graphs of groups}

\thanks{This research was supported by the Australian Research Council and University of Wollongong
	Research Support Scheme.}


\author[Brownlowe, Mundey, Pask]{Nathan Brownlowe, Alexander Mundey, David Pask}
\address[N. Brownlowe, A. Mundey, D. Pask]{School of Mathematics and Applied Statistics\\
University of Wollongong\\
NSW  2522\\
Australia}
\email{nathanb@uow.edu.au, alex.mundey@gmail.com, dpask@uow.edu.au}
\author[Spielberg]{Jack Spielberg}
\address[J. Spielberg]{School of Mathematical and Statistical Sciences, Arizona State University, Tempe AZ 85287-1804, USA}
\email{jack.spielberg@asu.edu}
\author[Thomas]{Anne Thomas}
\address[A. Thomas]{School of Mathematics and Statistics F07, University of Sydney, NSW 2006, Australia}
\email{anne.thomas@sydney.edu.au}

\begin{abstract}
To a large class of graphs of groups we associate a $C^*$-algebra universal for generators and relations. We show that this $C^*$-algebra is stably isomorphic to the crossed product induced from the action of the fundamental group of the graph of groups on the boundary of its Bass--Serre tree. We characterise when this action is minimal, and find a sufficient condition under which it is locally contractive. In the case of generalised Baumslag--Solitar graphs of groups (graphs of groups in which every group is infinite cyclic) we also characterise topological freeness of this action. We are then able to establish a dichotomy for simple $C^*$-algebras associated to generalised Baumslag--Solitar graphs of groups: they are either a Kirchberg algebra, or a stable Bunce--Deddens algebra.    
\end{abstract}

\maketitle

\setcounter{tocdepth}{1}

\tableofcontents

\section{Introduction}\label{sec: intro}


Actions of groups on trees, and the induced actions on tree boundaries, have given rise to many interesting $C^*$-algebras via the crossed product construction. Such examples include certain Cuntz--Krieger algebras considered in \cite{Sp} and \cite{Rob}, and the generalised Bunce--Deddens algebras considered in \cite{Orfanos} (see also Proposition~\ref{prop: classic odometer} of this paper).  

Actions of groups on trees of course play a fundamental role in many fields outside of $C^*$-algebras. A part of Serre's extensive contribution to the theory was the introduction of graphs of groups in \cite{Se}. The theory of graphs of groups was further developed by Bass in~\cite{Bass}, and is now known as Bass--Serre theory. Roughly speaking, a {\em graph of groups} consists of a graph $\G$ together with a group for each vertex and edge of $\G$, and monomorphisms from each edge group to the adjacent vertex groups.  Any group action on a tree (satisfying some mild hypotheses) induces a graph of groups, while any graph of groups has a canonical associated group, called the fundamental group, and a tree, called the Bass--Serre tree, such that the fundamental group acts on the Bass--Serre tree.  The so-called Fundamental Theorem of Bass--Serre Theory says that these processes are mutually inverse, so that graphs of groups ``encode" group actions on trees; see Theorem~\ref{thm: BS theorem} for a more precise statement.  

The original motivation for Bass--Serre theory was to study rank one reductive groups over nonarchimedean local fields, such as $\SL_2(\Q_p)$, by considering the action of such groups on their associated Bruhat--Tits tree (see for instance Chapter II of~\cite{Se}).  Graphs of groups are now fundamental tools in geometric group theory and low-dimensional topology.  The fundamental group of a graph of groups generalises two basic constructions in combinatorial group theory, namely free products with amalgamation and HNN extensions.  These constructions correspond in topology to taking a connected sum and adding a handle, respectively.  Classes of groups which are studied using Bass--Serre theory include lattices in automorphism groups of trees (see~\cite{BL} and its references), fundamental groups of $3$-manifolds (see, for instance,~\cite{dlHP}) and generalised Baumslag--Solitar groups (see Section~\ref{sec: GBS} below).

While the action of a group on the boundary of a tree is modelled $C^*$-algebraically via a crossed product, the data present in a graph of groups is more naturally modelled via a {\em combinatorial $C^*$-algebra}, which is a $C^*$-algebra universal for generators and relations encoding an underlying combinatorial object. Combinatorial $C^*$-algebras are a fruitful source of examples in $C^*$-algebra theory. The starting point was Cuntz and Krieger's work in \cite{CK}, and the theory has rapidly progressed in many directions, including a generalisation to directed graphs (see \cite{CBMS} for an overview). Much of the work on combinatorial $C^*$-algebras has been concerned with more general {\em oriented} situations.  Unoriented examples have received relatively little attention.  In \cite{CLM} the authors treat the case of a finite graph, giving a combinatorial analysis and relating the $K$-theory of the $C^*$-algebra to the first Betti number of the graph.  Iyudu \cite{I} and Ivankov-Iyudu \cite{II} have extended this to infinite graphs. Graphs of groups are a natural generalisation to a large class of new examples.  The first work to consciously apply Bass--Serre theory to the study of boundary crossed product $C^*$-algebras was done by Okayasu \cite{O}, who extended the results of \cite{Sp} and \cite{Rob} to finite graphs of finite groups, explicitly working with the Bass--Serre tree and the associated action on its boundary.

The main aim of this paper is to significantly extend the work of \cite{O} to build a Bass--Serre theory in the $C^*$-algebraic setting. We give a different construction  to that of \cite{O} that generalises to a much larger class of graphs of groups.  We work with two $C^*$-algebras: we first construct our graph of groups $C^*$-algebra, which is universal for generators and relations, and then we construct a boundary action crossed product $C^*$-algebra.  This is much like the approach of \cite{O}; however that approach relies heavily on the choice of a maximal tree for the graph of groups, which makes the universal $C^*$-algebra description somewhat complicated (see \cite[Definition 3.1]{O}).  In the special case of a (finite) graph, the paper \cite{CLM} gives a very simple and natural presentation, which inspired that of this paper (see Definition \ref{def: a G, Sigma family}).  Our methods apply to graphs of {\em countable} groups, in the sense that all vertex groups are countable. We also assume that the underlying graph is locally finite and that edge groups have finite index in their adjacent vertex groups; these assumptions ensure that the Bass--Serre tree is locally finite. We note that we do allow graphs of groups where the underlying graph is infinite and where the vertex groups are infinite groups. We also assume an extra condition (which we refer to as nonsingularity) on our graph of groups which ensures that the Bass--Serre tree has no finite ends.

In general, the graph of groups algebra contains as a distinguished subalgebra a certain directed graph $C^*$-algebra (see Theorem~\ref{thm: associated directed graph}).  When the examples of \cite{Sp} and \cite{Rob} are realised as graph of groups algebras, the subalgebra equals the whole algebra.  We give examples to show that in general the containment is proper.  One of the advantages of the graph of groups presentation is that it permits an easy description of a gauge action, analogous to the situation for directed graph algebras. The gauge action is needed for the existence of the directed graph subalgebra mentioned above.

Our main result, Theorem~\ref{thm: main theorem}, shows that the $C^*$-algebras obtained by our two constructions are related by stable isomorphism.  (Thus we find that in the case of a finite graph of finite groups, Okayasu's algebra from \cite[Theorem 4.4]{O} is isomorphic to a certain corner inside our graph of groups algebra.)  For the proof of this theorem we realise the graph of groups algebra as the $C^*$-algebra of a certain \'etale groupoid obtained from the fundamental groupoid of Higgins \cite{H}.  This groupoid could be taken as an alternate point of departure for the results of the paper.  We chose to emphasize the generators and relations as a natural generalisation of the theory of directed graph $C^*$-algebras, and the boundary crossed product as a direct generalisation of previous work.

After proving our main theorem, we consider structural properties of these $C^*$-algebras.  In line with the philosophy of the theory of directed graph $C^*$-algebras, we characterise various $C^*$-properties by means of combinatorial properties of the underlying graph of groups.  It is proved in \cite{AS} (see also \cite{KT}) that for an action of a discrete group $\Lambda$ on a compact Hausdorff space $Y$, the reduced crossed product $C(Y) \times_r \Lambda$ is simple if and only if the action of $\Lambda$ on $Y$ is minimal and topologically free. In this case, the reduced crossed product is purely infinite if the action is locally contractive (\cite{LS,A2}). In addition, the full  and the reduced crossed products coincide if either one is nuclear, and this occurs if and only if the action is amenable (\cite{A}).  Thus we investigate how minimality, topological freeness, and local contractivity of the boundary action are reflected in the graph of groups (amenability follows generally from known results).  Since the associated $C^*$-properties are invariant under stable isomorphism, our main theorem establishes these for the graph of groups algebra.

We are able to completely characterise minimality and local contractivity for locally finite nonsingular graphs of countable groups.  However we are able to characterise topological freeness only in two special cases: in the case of trivial groups, and generalised Baumslag--Solitar graphs of groups (graphs of groups where all vertex and edge groups are infinite cyclic). Our work on generalised Baumslag--Solitar graphs of groups culminates in a dichotomy which has a striking similarity to one present in the theory of directed graph $C^*$-algebras: a simple generalised Baumslag--Solitar graph of groups $C^*$-algebra is either a Kirchberg algebra, or a stable Bunce--Deddens algebra.

In general, topological freeness seems to be a subtle property to characterise, and it has been considered by other authors.  Topological freeness appears as a crucial hypothesis in the paper \cite{dlHP} of de la Harpe and Pr\'eaux, where it is referred to as {\em slenderness}.  In conjunction with some properties of the action on the tree, they use it to deduce minimality of the boundary action.  This contrasts with our methods, where we deduce minimality from properties of the underlying graph of groups, and do not assume topological freeness.  We note that our Theorem \ref{thm: gbs top free} generalises their Lemma 20(iii) on topological freeness for Baumslag--Solitar groups.  Topological freeness is also relevant in the paper \cite{BAP} of Broise--Alamichel and Paulin.  Their Remarque 4.2 gives a proof of topological freeness under several hypotheses.  These include finiteness of all vertex groups, as well as existence of a Patterson-Sullivan measure of positive finite dimension and a restriction on the elliptic elements of the group.  Our results on this topic intersect with theirs only for the case of graphs of trivial groups where the underlying graph $\G$ has finite Betti number greater than one (see Remark~\ref{rem:BroiseStuff} for more details).


\subsection*{Acknowledgements} The first author would like to thank the University of Glasgow and Arizona State University for their hospitality during research visits while this paper was written. The third author would like to thank Arizona State University for their hospitality during his research visit while this paper was written. The fourth and fifth author would like to thank the University of Wollongong for their warm hospitality during research visits while this paper was written. The fifth author would also like to thank Stuart White for helpful conversations.

\section{Background}\label{sec: background}

This section recalls background material and establishes our notation.  
In Section~\ref{subsec: Graphs} we recall the necessary concepts from graph theory and then in Section~\ref{subsec: Bass--Serre Theory} we recall the main definitions and results of Bass--Serre Theory.  In Section~\ref{subsec: boundaries of trees} we discuss boundaries of trees.  Section~\ref{subsec: groupoid} then discusses the groupoid approach to graphs of groups and the action on the relevant boundary of the fundamental groupoid and the fundamental group.  We finish by giving some background on $C^*$-algebras in Section~\ref{subsec: cstar background}.

\subsection{Graphs and trees}\label{subsec: Graphs}

The notion of graph that we use comes from Serre \cite{Se}.  

\begin{dfn}\label{dfn: graph}
A {\em graph} $\Gamma$ is a quadruple $(\Gamma^0,\Gamma^1,r,s)$, consisting of 
countable sets of \emph{vertices} $\Gamma^0$ and \emph{edges} $\Gamma^1$, together with \emph{range} and \emph{source maps} $r,s:\Gamma^1\to \Gamma^0$, and an ``edge-reversing" map $e\mapsto \overline{e}$ from $\Gamma^1$ to $\Gamma^1$ 
so that for all $e \in \G^1$,
\[
\overline{e}\not= e,\quad\overline{\overline{e}}=e\quad\text{and}\quad s(e)=r(\overline{e}).
\]
\end{dfn}

\noindent Such a graph $\G$ can be viewed as an undirected graph in which each geometric edge is replaced by a pair of edges $e$ and $\overline e$.
For a graph $\Gamma$ and an edge $e\in\Gamma^1$ we may draw the edge 
$\overline{e}$ with a dashed line, as in the top row of the following figure, or we may omit the edge $\overline{e}$, as in the bottom row:
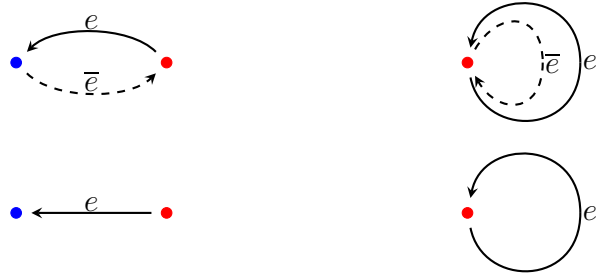
\[
\begin{tikzpicture}[scale=2]

\node (0_0) at (0,0) [circle] {};
\node (1_0) at (1,0) [circle] {};
\node (3_0) at (3,0) [circle] {};
\node (0_1) at (0,1) [circle] {};
\node (1_1) at (1,1) [circle] {};
\node (3_1) at (3,1) [circle] {};

\foreach \y in {0,1} \filldraw [blue] (0_\y) circle (1pt);
\foreach \y in {0,1} \filldraw [red] (1_\y) circle (1pt);
\foreach \y in {0,1} \filldraw [red] (3_\y) circle (1pt);

\draw[-stealth,thick] (1_1) .. controls (.75,1.25) and (.25,1.25) .. (0_1) node[pos=0.5, inner sep=0.5pt, anchor=south] {$e$};
\draw[-stealth,thick,dashed] (0_1) .. controls (.25,.75) and (.75,.75) .. (1_1) node[pos=0.5, inner sep=0.5pt, anchor=south] {$\overline{e}$};

\draw[-stealth,thick] (1_0)-- (0_0) node[pos=0.5, inner sep=0.5pt, anchor=south] {$e$};

\draw[thick,dashed] (3_1) .. controls (3.2,1.35) and (3.5,1.35) .. (3.5,1) node[pos=1, inner sep=0.5pt, anchor=west] {$\overline{e}$};
\draw[-stealth,thick,dashed] (3.5,1) .. controls (3.5,.65) and (3.2,.65) .. (3_1);

\draw[-stealth,thick] (3.75,1) .. controls (3.75,1.5) and (3.1,1.5) .. (3_1) node[pos=0, inner sep=0.5pt, anchor=west] {$e$};
\draw[thick] (3.75,1) .. controls (3.75,.5) and (3.1,.5) .. (3_1);

\draw[-stealth,thick] (3.75,0) .. controls (3.75,.5) and (3.1,.5) .. (3_0) node[pos=0, inner sep=0.5pt, anchor=west] {$e$};
\draw[thick] (3.75,0) .. controls (3.75,-.5) and (3.1,-.5) .. (3_0);

\end{tikzpicture}
\]

For $x \in \G^0$ the \emph{valence} of $x$ is the cardinality of the set $r^{-1}(x)$ (equivalently, the cardinality of the set $s^{-1}(x)$).  A graph $\G$ is \emph{locally finite} if $|r^{-1}(x)| < \infty$ for all $x \in \G^0$, that is, each vertex has finite valence.  We only consider locally finite graphs in this paper.

A \emph{path (of length $n$)} in $\Gamma$ is either a vertex $x \in \G^0$ (when $n = 0$) or, if $n > 0$, a sequence of edges $e_1 e_2 \dots e_n$ with $s(e_i ) = 
r (e_{i+1} )$ for $1 \le i \le n-1$.  For example:

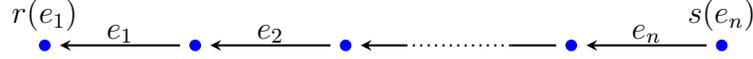
\[
\begin{tikzpicture}[scale=2]

\node (0_0) at (0,0) [circle] {};
\node (1_0) at (1,0) [circle] {};
\node (2_0) at (2,0) [circle] {};
\node (3h_0) at (3.5,0) [circle] {};
\node (4h_0) at (4.5,0) [circle] {};

\foreach \x in {0,1,2,3.5,4.5} \filldraw [blue] (\x,0) circle (1pt);

\draw[-stealth,thick] (1_0) -- (0_0) node[pos=.5, inner sep=.5pt, anchor=south] {$e_1$};
\draw[-stealth,thick] (2_0) -- (1_0) node[pos=.5, inner sep=.5pt, anchor=south] {$e_2$};
\draw[-stealth,thick] (4h_0) -- (3h_0) node[pos=.5, inner sep=.5pt, anchor=south] {$e_n$};
\draw[-stealth,thick] (2.4,0) -- (2_0);
\draw[thick,dotted] (3.1,0) -- (2.4,0);
\draw[thick] (3h_0) -- (3.1,0);

\draw (0,.2) node {$r(e_1)$};
\draw (4.5,.2) node {$s(e_n)$};

\end{tikzpicture}
\]

\noindent Note that we are using the ``Australian" convention for paths in a graph, which is more functorial and sits well with the operator-algebraic methods we will use in this paper;  the same reasoning applied in \cite{CBMS}. For a path $\mu = e_1 e_2 \cdots e_n$ we define the {\em reversal} of $\mu$ to be $\overline{\mu} = \overline{e_n} \cdots \overline{e_2} \, \overline{e_1}$.

A path $x$ of length $0$ has \emph{range} and \emph{source} equal to the vertex $x$, and a path $e_1 e_2 \dots e_n$ with $n \geq 1$ has \emph{range} $r(e_1)$ and \emph{source} $s(e_n)$.  A \emph{cycle} is a path with range equal to its source.   A cycle $e_1 e_2 \dots e_n$ is \emph{minimal} if $n=0$, or the vertices $r(e_i)$ are pairwise distinct.   A path $e_1 e_2 \dots e_n$ is \emph{reduced} if either $n = 0$, or $e_{i+1} \neq \overline{e}_i$ for $1 \leq i \leq n-1$, that is, there is no immediate back-tracking in the path.  We write $\G^*$ (respectively, $x\G^*$, $\G^*y$ and $x \G^* y$) for the set of reduced paths in $\G$ (respectively, the reduced paths in $\G$ with range $x$, with source $y$, and with range $x$ and source $y$). A vertex of a graph is \emph{singular} if it has valence one; that is, it is the range (equivalently, source) of a unique edge.  We say that a graph is \emph{nonsingular} if it has no singular vertices. A graph $\Gamma$ is {\em connected} if 
for every $x \neq y \in \Gamma^0$ there is a path $e_1 e_2 \dots e_n$ with range $x$ and source $y$.  We only consider 
connected graphs in this paper. 

 A graph $\Gamma$ is a \emph{tree} if it is connected and for every vertex $x \in \G^0$, the only reduced path which starts and ends at $x$ is of length $0$.  Equivalently, $\G$ is a tree if for every $x \neq y \in \Gamma^0$, there is a unique reduced path $e_1 e_2 \dots e_n$ with range $x$ and source $y$.  If $\G$ is a tree, $x \in \G^0$ and $e \in \G^1$, then we say that \emph{$e$ points towards $x$} if the unique reduced path with range $x$ and source $r(e)$ does not contain the edge $e$.  In other words, $x$ is closer in $\G$ to $r(e)$ than to $s(e)$.

Every graph $\G$ has at least one maximal subtree $T$, for which  $T^0=\Gamma^0$ and 
$T^1\subseteq \Gamma^1$.  For such a $T$, a \emph{path in $T$} is a path in $\Gamma$ in which every edge belongs to $T^1$. Note that from Section~\ref{sec: C stars associated to gog} onwards we only consider nonsingular trees in this paper.

\subsection{Bass--Serre Theory}\label{subsec: Bass--Serre Theory}  

In this section we recall the theory of graphs of groups, also known as Bass--Serre Theory.  We generally follow Bass' work~\cite{Bass}, although our notation and terminology differ in places since we use directed graph $C^*$-algebra conventions where possible.  We start with the definition of a graph of groups~$\GG$.  We then discuss the fundamental group~$\pi_1(\GG,v)$ and the Bass--Serre tree $X_{\GG,v}$ of a graph of groups $\GG$, and describe the action of $\pi_1(\GG,v)$ on $X_{\GG,v}$.  

\begin{dfn}\label{def: graph of groups}
A {\em graph of groups} $\GG=(\Gamma,G)$ consists of a connected graph 
$\Gamma$ together with:
\begin{enumerate}
\item a \emph{vertex group} $G_x$ for each $x\in\Gamma^0$;
\item an \emph{edge group}  $G_e$ for each $e\in \Gamma^1$, such that $G_e=G_{\overline{e}}$ for all $e \in \G^1$; and 
\item a monomorphism $\alpha_e:G_e\to G_{r(e)}$, for each $e \in \G^1$.
\end{enumerate}
\end{dfn}
\noindent   For each $x\in\Gamma^0$ we denote by $1_x$ the identity element of the vertex group $G_x$, and we write $1$ for $1_x$ if the vertex $x$ is clear.

\begin{examples}\label{egs: graphs of groups}  The two basic examples of graphs of groups are edges of groups and loops of groups, which we now define.  We will follow these examples throughout this section.
\begin{enumerate}
\item[(E1)] Let $\Gamma$ be the graph 
\[
\begin{tikzpicture}[scale=2]

\node (0_0) at (0,0) [circle] {};
\node (1_0) at (1,0) [circle] {};

\filldraw [blue] (0,0) circle (1pt);
\filldraw [red] (1,0) circle (1pt);

\draw[-stealth,thick] (1_0) -- (0_0) node[pos=.5, inner sep=.5pt, anchor=south] {$e$};

\draw (0,.15) node {$x$};
\draw (1,.15) node {$y$};

\end{tikzpicture}
\]
\noindent An \emph{edge of groups} is a graph of groups $\GG= (\Gamma,G)$, and can be depicted as
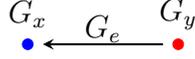
\[
\begin{tikzpicture}[scale=2]

\node (0_0) at (0,0) [circle] {};
\node (1_0) at (1,0) [circle] {};

\filldraw [blue] (0,0) circle (1pt);
\filldraw [red] (1,0) circle (1pt);

\draw[-stealth,thick] (1_0) -- (0_0) node[pos=.5, inner sep=.5pt, anchor=south] {$G_e$};

\draw (0,.2) node {$G_x$};
\draw (1,.2) node {$G_y$};

\end{tikzpicture}
\]
\noindent The monomorphisms are $\alpha_e:G_e \to G_x$ and $\alpha_{\overline{e}}:G_{\overline{e}} = G_e \to G_y$. 

\item[(E2)] Let $\Gamma$ be the graph 
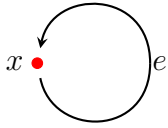
\[
\begin{tikzpicture}[scale=2]

\node (0_0) at (0,0) [circle] {};

\filldraw [red] (0,0) circle (1pt);

\draw[-stealth,thick] (.75,0) .. controls (.75,.5) and (.1,.5) .. (0_0) node[pos=0, inner sep=0.5pt, anchor=west] {$e$};
\draw[thick] (.75,0) .. controls (.75,-.5) and (.1,-.5) .. (0_0);

\draw (-.15,0) node {$x$};

\end{tikzpicture}
\]

\noindent A \emph{loop of groups} is a graph of groups $\GG= (\Gamma,G)$, and can be depicted as
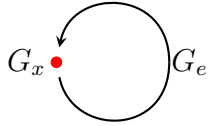
\[
\begin{tikzpicture}[scale=2]

\node (0_0) at (0,0) [circle] {};

\filldraw [red] (0,0) circle (1pt);

\draw[-stealth,thick] (.75,0) .. controls (.75,.5) and (.1,.5) .. (0_0) node[pos=0, inner sep=0.5pt, anchor=west] {$G_e$};
\draw[thick] (.75,0) .. controls (.75,-.5) and (.1,-.5) .. (0_0);

\draw (-.2,0) node {$G_x$};

\end{tikzpicture}
\]

\noindent The monomorphisms are $\alpha_e:G_e \to G_x$ and $\alpha_{\overline{e}}:G_{\overline{e}} = G_e \to G_x$.

\end{enumerate}
\end{examples}

We say that a graph of groups $\GG = (\G, G)$ is {\em locally finite} if 
\begin{enumerate}[(a)]
\item the underlying graph $\G$ is locally finite; and 
\item $[G_{r(e)}:\alpha_e(G_e)]<\infty$ for all $e\in\Gamma^1$.
\end{enumerate}
Condition (b) here is saying that each edge group has finite index image in the adjacent vertex groups (recall that $\alpha_{\overline e}$ is a monomorphism from $G_{\overline e} = G_e$ to $G_{r(\overline e)} = G_{s(e)}$).
A graph of groups $\GG = (\G,G)$ is \emph{nonsingular} if for all $e\in\Gamma^1$ such that $r^{-1}(r(e))=\{e\}$, we have 
$[G_{r(e)}:\alpha_e(G_e)]>1$.  That is, if $e$ is the unique edge with range $r(e)$, then $\alpha_e(G_e)$ must be a proper subgroup of $G_{r(e)}$.
We only consider locally finite nonsingular graphs of groups in this paper.

Throughout this work, we will need various kinds of words and paths associated to a graph of groups $\GG = (\G, G)$. 

\begin{dfn}\label{def: words and paths}  For each $e \in \Gamma^1$, fix a transversal $\Sigma_e$ for $G_{r (e)} / 
\alpha_e (G_e )$, with $1_{r (e)} \in  \Sigma_e$.
\begin{enumerate}[(i)]
\item A {\em $\GG$-word (of length $n$)} is a sequence of the form
\[
g_1 \quad\text{ or }\quad g_1e_1g_2e_2\dots g_ne_n\quad\text{ or }\quad g_1e_1g_2e_2\dots g_ne_ng_{n+1},
\]
where $s(e_i ) = r (e_{i+1} )$ for $1 
\le i \le n-1$, $g_j\in G_{r(e_j)}$ for $1\le j\le n$, and $g_{n+1}\in 
G_{s(e_n)}$.  (In the case $n = 0$,  $g_1$ is just required to be an element of some vertex group $G_x$.)
For example, a $\GG$-word of the form $g_1 e_1 g_2 e_2\dots g_n e_n g_{n+1}$ can be pictured as follows:
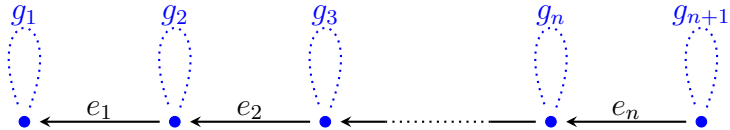
\[
\begin{tikzpicture}[scale=2]

\node (0_0) at (0,0) [circle] {};
\node (1_0) at (1,0) [circle] {};
\node (2_0) at (2,0) [circle] {};
\node (3h_0) at (3.5,0) [circle] {};
\node (4h_0) at (4.5,0) [circle] {};

\foreach \x in {0,1,2,3.5,4.5} \filldraw [blue] (\x,0) circle (1pt);

\draw[-stealth,thick] (1_0) -- (0_0) node[pos=.5, inner sep=.5pt, anchor=south] {$e_1$};
\draw[-stealth,thick] (2_0) -- (1_0) node[pos=.5, inner sep=.5pt, anchor=south] {$e_2$};
\draw[-stealth,thick] (4h_0) -- (3h_0) node[pos=.5, inner sep=.5pt, anchor=south] {$e_n$};
\draw[-stealth,thick] (2.4,0) -- (2_0);
\draw[thick,dotted] (3.1,0) -- (2.4,0);
\draw[thick] (3h_0) -- (3.1,0);

\draw[thick,dotted,blue] (0_0) .. controls (.3,.8) and (-.3,.8) .. (0_0) node[pos=.5, inner sep=0.5pt, anchor=south] {$g_1$};

\draw[thick,dotted,blue] (1_0) .. controls (1.3,.8) and (.7,.8) .. (1_0) node[pos=.5, inner sep=0.5pt, anchor=south] {$g_2$};

\draw[thick,dotted,blue] (2_0) .. controls (2.3,.8) and (1.7,.8) .. (2_0) node[pos=.5, inner sep=0.5pt, anchor=south] {$g_3$};

\draw[thick,dotted,blue] (3h_0) .. controls (3.8,.8) and (3.2,.8) .. (3h_0) node[pos=.5, inner sep=0.5pt, anchor=south] {$g_n$};

\draw[thick,dotted,blue] (4h_0) .. controls (4.8,.8) and (4.2,.8) .. (4h_0) node[pos=.5, inner sep=0.5pt, anchor=south] {$g_{n+1}$};

\end{tikzpicture}
\]

\noindent If $g_1 \in G_x$ is a $\GG$-word of length $0$ then we define this $\GG$-word to have \emph{range} and \emph{source} the vertex $x$.  For $n > 0$  the $\GG$-words $g_1e_1g_2e_2\dots g_ne_n$ and $g_1e_1g_2e_2\dots g_ne_ng_{n+1}$ are both defined to have \emph{range} $r(e_1)$ and \emph{source} $s(e_n)$.  We denote by $|\mu|$ the length of a $\GG$-word $\mu$.

\item A {\em reduced $\GG$-word (of length $n$)} is a $\GG$-word of the form
\[
g_1 \quad\text{ or }\quad g_1e_1g_2e_2\dots g_ne_n \quad\text{ or }\quad g_1e_1g_2e_2\dots g_ne_ng_{n+1},
\]
where if $n > 0$ we have $g_j\in \Sigma_{e_j}$ for $1\le j\le n$ and 
$e_i=\overline{e_{i+1}}\Longrightarrow g_{i+1}\not= 1_{r(e_{i+1})}$, and $g_{n+1}$ is free to be any element of 
$G_{s(e_n)}$.   (In the case $n = 0$,  $g_1$ is just required to be an element of some vertex group $G_x$; in contrast to \cite[Section~1.7]{Bass}, we refer to the trivial $\GG$-words $1_x$ as reduced.)

\item A {\em $\GG$-loop based at $x\in\Gamma^0$} is a $\GG$-word of the 
form $g_1 \in G_x$ or $g_1e_1g_2e_2\dots g_ne_ng_{n+1}$ with $r(e_1)=s(e_n)=x$. A $\GG$-loop 
is called \emph{reduced} if it is also a reduced $\GG$-word.

\item A {\em $\GG$-path} is a reduced $\GG$-word of the form $g_1 = 1$ or 
$g_1e_1g_2e_2\dots g_ne_n$; so we insist that $\GG$-paths of length $n > 0$ must end in edges. We 
denote the collection of $\GG$-paths of length $n$ by $\GG^n$, and $\GG^*:=\cup_{n=0}^\infty \GG^n$. For $x\in\Gamma^0$ we denote 
by $x\GG^n$ the collection of $\GG$-paths of length $n$ with range~$x$, and by $x\GG^*$ the collection of all $\GG$-paths with range $x$.

\end{enumerate}
\end{dfn}

In order to define the fundamental group of a graph of groups $\GG = (\G, G)$, we will need the following auxiliary group, which is defined via a presentation.

\begin{dfn}\label{def: path group}
Let $\GG = (\G,G)$ be a graph of groups.  The \emph{path group}, denoted $\pi(\GG)$, has generating set  
\begin{equation}\label{E:PathGenerators}
\G^1 \sqcup \left(\bigsqcup_{x \in \G^0} G_x \right),
\end{equation}
that is, the edge set of the graph $\G$ together with the elements of the vertex groups of $\GG$, and defining relations the relations in the vertex groups, together with: 
\begin{enumerate}
\item[(R1)] $\overline{e} = e^{-1}$ for all $e \in \G^1$; and  
\item[(R2)] $e \alpha_{\overline{e}}(g) \overline{e} = \alpha_e(g)$ for all $e \in \G^1$ and all $g \in G_e = G_{\overline e}$.
\end{enumerate}  
\end{dfn}

The relation (R2) in the definition of the path group can be thought of as identifying the ``loop'' $\alpha_e(g)$ with the ``loop'' obtained (reading from right to left, as with composition of morphisms) by following the edge $\overline{e}$ (from $r(e)$ to $s(e)$), then $\alpha_{\overline{e}}(g)$, and then ``returning'' along $e$:
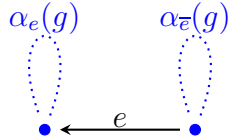
\[
\begin{tikzpicture}[scale=2]

\node (0_0) at (0,0) [circle] {};
\node (1_0) at (1,0) [circle] {};

\foreach \x in {0,1} \filldraw [blue] (\x,0) circle (1pt);

\draw[-stealth,thick] (1_0) -- (0_0) node[pos=.5, inner sep=.5pt, anchor=south] {$e$};

\draw[thick,dotted,blue] (0_0) .. controls (.3,.8) and (-.3,.8) .. (0_0) node[pos=.5, inner sep=0.5pt, anchor=south] {$\alpha_e(g)$};

\draw[thick,dotted,blue] (1_0) .. controls (1.3,.8) and (.7,.8) .. (1_0) node[pos=.5, inner sep=0.5pt, anchor=south] {$\alpha_{\overline{e}}(g)$};

\end{tikzpicture}
\]

We now observe that (reduced) $\GG$-words, and consequently (reduced) $\GG$-loops and $\GG$-paths, naturally map to elements of the path group $\pi(\GG)$.  More precisely, if the sequence
\[
g_1 \quad\text{ or }\quad g_1e_1g_2e_2\dots g_ne_n\quad\text{ or }\quad g_1e_1g_2e_2\dots g_ne_ng_{n+1}
\]
is a $\GG$-word, then by abuse of notation we associate to this $\GG$-word the element $g_1$ or $g_1e_1g_2e_2\dots g_ne_n$ or $g_1e_1g_2e_2\dots g_ne_ng_{n+1}$, respectively, of the path group $\pi(\GG)$.  

\begin{dfn}\label{def: path subset}
For $x, y \in \G^0$, define $\pi[x,y] \subseteq \pi(\GG)$ to be the set of images in the path group $\pi(\GG)$ of the $\GG$-words which have range $x$ and source $y$.  
\end{dfn}

Notice that, due to the relations (R1) and (R2), two different $\GG$-words can map to the same element of $\pi[x,y]$.  For instance the $\GG$-words $g_1 e_1 g_2 $ and $g_1 e_1 g_2 e_2 1 \overline{e_2}$ have the same image $g_1 e_1 g_2$ in $\pi[x,y]$, where $g_1 \in G_x$ and $g_2 \in G_y$.   By Theorem~1.8 of~\cite{Bass}, for any choice of the transversals $\Sigma_e$, the image of a nontrivial reduced $\GG$-word in $\pi(\GG)$ is nontrivial.  Thus by definition the image of a nontrivial $\GG$-path in $\pi(\GG)$ is nontrivial.  By Corollary~1.13 of~\cite{Bass}, for all $x$, $y \in \G^0$, every element of $\pi[x,y]$ is represented by a unique reduced $\GG$-word with range $x$ and source $y$; again, this result holds for any choice of transversals $\Sigma_e$.

We now choose a base vertex $v \in \G^0$ and consider the set $\pi[v,v]$, which by definition is the set of images of $\GG$-loops based at $v$.  By Theorem~1.8 and Corollary~1.13 of~\cite{Bass}, the image in $\pi[v,v]$ of a nontrivial reduced $\GG$-loop based at $v$ is nontrivial, and every nontrivial element of $\pi[v,v]$ is represented by a unique reduced $\GG$-loop based at $v$.  Now let $\gamma, \gamma' \in \pi[v,v]$.  Then $\gamma\gamma' \in \pi[v,v]$.  Also if 
\[
\gamma = g_1 \quad\text{ or }\quad \gamma = g_1e_1g_2e_2\dots g_ne_ng_{n+1},
\] 
then $\gamma^{-1}$ in the group $\pi(\GG)$ is given by
\[
\gamma^{-1} = g_1^{-1}  \quad\text{ or }\quad \gamma^{-1} = g_{n+1}^{-1}\overline{e}_n g_{n-1}^{-1} \overline{e}_{n-1} \dots g_2^{-1}\overline{e}_1 g_{1}^{-1},
\]
 respectively, and so $\gamma^{-1}$ is also in the set $\pi[v,v]$.  This allows us to make the following definition.

\begin{dfn}\label{def: fundamental group}  Let $\GG = (\G, G)$ be a graph of groups, and choose a base vertex $v \in \G^0$.  The \emph{fundamental group of $\GG$ based at $v$}, denoted $\pi_1(\GG,v)$, is the subgroup $\pi[v,v]$ of the path group $\pi(\GG)$.
\end{dfn}

An alternative way of defining the fundamental group of a graph of groups is via a presentation, as follows. 

\begin{dfn}\label{def: fund group tree}
Let $\GG = (\G, G)$ be a graph of groups and choose a maximal subtree $T$ of the graph $\G$.  The \emph{fundamental group of $\GG$ relative to $T$}, denoted $\pi_1(\GG,T)$, has the same generating set \eqref{E:PathGenerators} as the path group $\pi(\GG)$, and defining relations the relations in the vertex groups of $\GG$, the relations (R1) and (R2) above, and the additional relation:
\begin{enumerate}
\item[(R3)] $e = 1$ for all $e \in T^1$.
\end{enumerate}
\end{dfn}

\noindent That is, the fundamental group $\pi_1(\GG,T)$ is the quotient of the path group $\pi (\GG)$ obtained by killing all of the edges in the chosen maximal subtree $T \subset \G$.  

For each choice of base vertex $v$ and maximal subtree $T \subset \G$, the induced projection $\pi(\GG) \to \pi_1(\GG,T)$ restricts to an isomorphism of fundamental groups $\pi_1(\GG,v) \to \pi_1(\GG,T)$ (see~\cite[Proposition 1.20]{Bass}).  Thus in particular, up to isomorphism the fundamental group of a graph of groups is independent of the choice of base vertex or of maximal subtree.  When the base vertex $v$ is clear, we may drop it from the notation and write $\pi_1(\GG)$ for the fundamental group of $\GG$ based at $v$.  Notice that if all the vertex and edge groups in the graph of groups $\GG$ are trivial, then $\pi_1 ( \GG , T ) \cong \pi_1(\GG,v)$ is the usual fundamental group of the (geometric realisation of the) graph $\Gamma$.

\begin{ntn} \label{ntn: general notation}

We use $\varepsilon$ to denote the inverse isomorphism from $\pi_1(\GG,T)$ to $\pi_1(\GG,v)$, for $T \subseteq \G$ any maximal subtree and any $v \in \G^0$. In order to define this map $\varepsilon$, for each $x \neq y\in \Gamma^0$ we first define 
\begin{equation}\label{E: [x,y]_T}
[x,y]_T:=e_1e_2\dots e_n
\end{equation}
to be the unique reduced path in $T$ such that $r(e_1)=x$ and $s(e_n)=y$.  Note that the path $[x,y]_T = e_1 e_2 \dots e_n$ has the same image in the path group $\pi(\GG)$ as the reduced $\GG$-path 
\begin{equation}\label{E: [x,y]}
[x,y]:= 1e_1 1 e_2 \dots 1 e_n.
\end{equation}
Thus in particular, the image of $[x,y]_T$ in the path group is an element of $\pi[x,y]$. We also define $[x,x]$ to be the trivial element of $G_x$, and so $[x,x]$ has trivial image in $\pi_1(\GG,x)$. 

Now the group $\pi_1(\GG,T)$ is generated by the edge set $\Gamma^1$ together with the vertex groups of $\GG$, and so it suffices to specify $\varepsilon$ on these generators.
For each $e\in \Gamma^1$ we define 
\begin{equation}\label{E: ep(e)}
\varepsilon(e):=[v,r(e)]e[s(e),v]\in\pi_1(\GG,v).
\end{equation}
Note that $\varepsilon(e)$ is the identity element of $\pi_1(\GG,v)$ whenever 
$e\in T^1$.  Also, for each $x\in\Gamma^0$ and $g\in G_x$, we define 
\begin{equation}\label{E: ep(g)}
\varepsilon(g):=[v,x]g[x,v]\in\pi_1(\GG,v).
\end{equation}
If we need to specify the vertex $x$ we will write $\varepsilon(x,g)$ for $\varepsilon(g)$.  Since $\varepsilon$ is an isomorphism of groups, the elements $\varepsilon(e)$ and $\varepsilon(g)$ generate $\pi_1(\GG,v)$. Note that as a straightforward consequence of (R1) and (R2) we have 
\begin{equation}\label{eq: formerly appx (B)}
\varepsilon(\alpha_e(g))\varepsilon(e)=\varepsilon(e)\varepsilon(\alpha_{\overline{e}}(g)),\quad\text{for each $e\in \Gamma^1$, $g\in G_e$.}
\end{equation}
\end{ntn}

\begin{rmk}\label{rem:EpsilonNotation}
We reiterate that even though $\varepsilon(e)=[v,r(e)]e[s(e),v]$ is not a $\GG$-word (because there is no group element between the last edge of $[v,r(e)]$ and $e$), it is a fundamental group element. Indeed, inside the path group, the image of the $\GG$-loop based at $v$ given by $[v,r(e)] 1e [s(e),v]$ is the fundamental group element represented by all of the products
\[
[v,r(e)] 1e [s(e),v] = [v,r(e)] e [s(e),v] = [v,r(e)]_T e [s(e),v]_T = [v,r(e)]_T 1e [s(e),v]_T.
\]	
\end{rmk}	

\begin{ntn}\label{ntn:UnderlineArrow}
For each $e\in \Gamma^1$ with $r(e) \not= v$ we denote by $\underleftarrow{e}$ the 
source-most edge in $[v,r(e)]_T$.  That is, if $[v,r(e)]_T = e_1 e_2 \dots e_n$ then $\underleftarrow{e} = e_n$.  By definition, $\underleftarrow{e}$ will always be an edge of the tree $T$, and so if $e \not \in T^1$ then $\underleftarrow{e} \neq e,\overline{e}$.  In the following figure, we assume all edges are in the tree $T$.  In the case shown on top, we have $\underleftarrow{e}$ different from both $e$ and $\overline{e}$, while on the bottom we have $\underleftarrow{e} = \overline{e}$.  The case $\underleftarrow{e} = e$ never occurs.
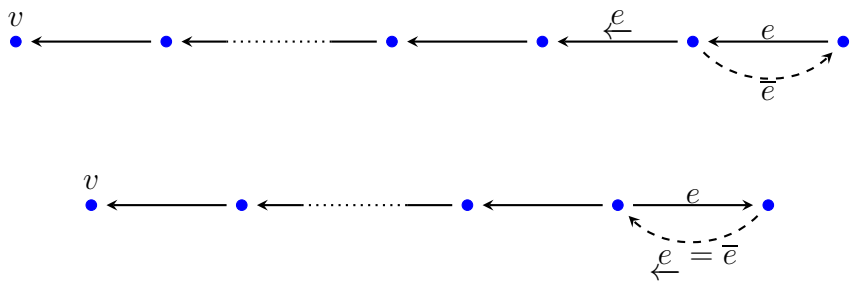
\[
\begin{tikzpicture}[scale=2]

\node (0_0) at (0,0) [circle] {};
\node (1_0) at (1,0) [circle] {};
\node (2h_0) at (2.5,0) [circle] {};
\node (3h_0) at (3.5,0) [circle] {};
\node (4h_0) at (4.5,0) [circle] {};
\node (5h_0) at (5.5,0) [circle] {};

\foreach \x in {0,1,2.5,3.5,4.5,5.5} \filldraw [blue] (\x,0) circle (1pt);

\draw[-stealth,thick] (1_0) -- (0_0);
\draw[-stealth,thick] (1.4,0) -- (1_0);
\draw[thick,dotted] (2.1,0) -- (1.4,0);
\draw[thick] (2h_0) -- (2.1,0);
\draw[-stealth,thick] (3h_0) -- (2h_0);
\draw[-stealth,thick] (4h_0) -- (3h_0) node[pos=.5, inner sep=.5pt, anchor=south] {$\underleftarrow{e}$};
\draw[-stealth,thick] (5h_0) -- (4h_0) node[pos=.5, inner sep=.5pt, anchor=south] {$e$};
\draw[-stealth,thick,dashed] (4h_0) .. controls (4.8,-.3) and (5.2,-.3) .. (5h_0) node[pos=.5, inner sep=.5pt, anchor=north] {$\overline{e}$};

\draw (0,.15) node {$v$};

\end{tikzpicture}
\]

\[
\begin{tikzpicture}[scale=2]

\node (0_0) at (0,0) [circle] {};
\node (1_0) at (1,0) [circle] {};
\node (2h_0) at (2.5,0) [circle] {};
\node (3h_0) at (3.5,0) [circle] {};
\node (4h_0) at (4.5,0) [circle] {};

\foreach \x in {0,1,2.5,3.5,4.5} \filldraw [blue] (\x,0) circle (1pt);

\draw[-stealth,thick] (1_0) -- (0_0);
\draw[-stealth,thick] (1.4,0) -- (1_0);
\draw[thick,dotted] (2.1,0) -- (1.4,0);
\draw[thick] (2h_0) -- (2.1,0);
\draw[-stealth,thick] (3h_0) -- (2h_0);
\draw[-stealth,thick] (3h_0) -- (4h_0) node[pos=.5, inner sep=.5pt, anchor=south] {$e$};
\draw[-stealth,thick,dashed] (4h_0) .. controls (4.2,-.3) and (3.8,-.3) .. (3h_0) node[pos=.5, inner sep=.5pt, anchor=north] {$\underleftarrow{e} = \overline{e}$};

\draw (0,.15) node {$v$};

\end{tikzpicture}
\]

\end{ntn} 

\begin{examples}\label{egs: fund gps}  We describe the fundamental groups of the graphs of groups from Examples~\ref{egs: graphs of groups}, using the version $\pi_1(\GG,T)$, since this is given by a presentation.  In both cases, the fundamental group obtained is a standard construction in combinatorial group theory.
\begin{enumerate}
\item[(E1)] The only maximal subtree $T$ of the graph $\G$ is $T = \G$, and so in $\pi_1(\GG,T)$ we have $e = \overline{e} = 1$.  Thus $\pi_1(\GG,T)$ is generated by $G_x \sqcup G_y$, and has defining relations the relations in $G_x$ and in $G_y$, together with $\alpha_{\overline{e}}(g) = \alpha_e(g)$ for all $g \in G_e$.  Hence the fundamental group of this edge of groups is isomorphic to the free product of $G_x$ and $G_y$ amalgamated over $G_e$.  That is,
\[\pi_1(\GG) \cong G_x *_{G_e} G_y. \]
\item[(E2)] The only maximal subtree $T$ of the graph $\G$ is $T = \{ x \}$, and so this time $e$ is nontrivial in $\pi_1(\GG,T)$.  The group $\pi_1(\GG,T)$ is generated by $\{e,\overline{e} \} \sqcup G_x$, and has defining relations the relations in $G_x$, together with $\overline{e} = e^{-1}$ and $e\alpha_{\overline{e}}(g)e^{-1} = \alpha_e(g)$ for all $g \in G_e$.  Hence the fundamental group of this loop of groups is isomorphic to the corresponding HNN extension of $G_x$.  That is,
\[\pi_1(\GG) \cong G_x *_{G_e}. \]
\end{enumerate}
\end{examples}

In the next definition, we will give the vertex and edge sets of a certain graph.  By Theorem~1.17 of~\cite{Bass}, this graph is actually a tree, so the following terminology is justified.  

\begin{dfn}\label{def: BS Tree}  Let $\GG = (\G,G)$ be a graph of groups, and choose a base vertex $v \in \G^0$.  The \emph{Bass--Serre tree}   $X_{\GG,v}$ of the graph of groups $\GG$, also known as its \emph{universal cover}, has vertex set
\[
X_{\GG,v}^0 = \bigsqcup_{x \in \G^0} \pi[v,x]/G_x = \{ \gamma G_x \mid \gamma \in \pi[v,x], x \in \G^0 \} .
\]
For $\gamma \in \pi[v,x]$ and $\gamma' \in \pi[v,x']$, there is an edge $f \in X_{\GG,v}^1$ with $r(f) = \gamma G_x$ and $s(f) = \gamma' G_{x'}$ if and only if $\gamma^{-1} \gamma' \in G_x  e G_{x'}$, where $e \in \G^1$ with $r(e) = x$ and $s(e) = x'$.
\end{dfn}

As discussed in Remark~1.18 of \cite{Bass}, the Bass--Serre tree $X_{\GG,v}$ is the tree associated to the inverse system
\begin{equation}\label{E: inverse}
v\GG^0 \overset{q_1}{\longleftarrow} v\GG^1 \overset{q_2}{\longleftarrow} \dots \overset{q_{n-1}}{\longleftarrow} v\GG^{n-1} \overset{q_n}{\longleftarrow} v\GG^n \overset{q_{n+1}}{\longleftarrow} \dots,
\end{equation}
where for each $n > 0$ the function $q_n:v\GG^n \to v\GG^{n-1}$ is given by $q_n (g_1e_1g_2e_2\dots g_ne_n) = g_1e_1g_2e_2\dots g_{n-1}e_{n-1}$. In the identification of this inverse system with the tree $X_{\GG,v}$, the set $v\GG^0$ corresponds to the base vertex $1G_v$ of $X_{\GG,v}$, and for $n > 0$ the set $v\GG^n$ corresponds to the vertices of $X_{\GG,v}$ at distance $n$ from $1G_v$. Thus each vertex at distance $n > 0$ from $1G_v$ has a unique representative $\GG$-path of the form $g_1e_1g_2e_2\dots g_ne_n$ where $r(e_1) = v$. With this description, it is easy to verify that the graph of groups $\GG$ is locally finite (respectively, nonsingular) if and only if the Bass--Serre tree $X_{\GG,v}$ is locally finite (respectively, nonsingular).

\begin{examples}  We describe the Bass--Serre trees of the graphs of groups from Examples~\ref{egs: graphs of groups}, using the inverse system~\eqref{E: inverse}.  
\begin{enumerate}
\item[(E1)] Choose $v = x$.  Let $n_x = [G_x: \alpha_e(G_e)]$ and $n_y = [G_y:\alpha_{\overline{e}}(G_e)]$, so that the transversals $\Sigma_e$ and $\Sigma_{\overline e}$ have respectively $n_x$ and $n_y$ elements.  Since $\GG$ is locally finite and nonsingular, we have $1 < n_x, n_y < \infty$. We have $|v\GG^1|=|\{ge:g\in \Sigma_e\}|=n_x$, and so the base vertex $1G_v$ of $X_{\GG,v}$ has valence $n_x$. For $geg_2e_2\in x\GG^2$ we must have $e_2=\overline e$ and $g_2 \neq 1_{y}$, and so there are $n_y - 1$ choices for $g_2$.  Thus each vertex in $v\GG^1$ is the range of one edge with source $1G_v$, and $n_y - 1$ other edges, and so each vertex in $v\GG^1$ has valence $n_y$.  Continuing in this way, we see that the vertices in $v\GG^k$ have valence $n_x$ or $n_y$ as $k$ is odd or even, respectively, and so the vertices in the tree $X_{\GG,v}$ alternate between having valence $n_x$ and valence $n_y$.  That is, $X_{\GG,v}$ is the $(n_x,n_y)$-biregular tree.

\item[(E2)] The only possible base vertex is $v = x$.  Let $m_e = [G_x: \alpha_e(G_e)]$ and $m_{\overline e} = [G_x: \alpha_{\overline e}(G_e)]$, so that the transversals $\Sigma_e$ and $\Sigma_{\overline e}$ have respectively $m_e$ and $m_{\overline{e}}$ elements. Since $\GG$ is locally finite, we have $1 \le m_e, m_{\overline{e}} < \infty$. We have $|v\GG^1|=|\{ge:g\in \Sigma_e\}\cup\{g\overline{e}:g\in\Sigma_{\overline{e}}\}|=m_e+m_{\overline{e}}$, and so the base vertex $1G_v$ of $X_{\GG,v}$ has valence $m:=m_e+m_{\overline{e}}$. For $geg_2e_2\in v\GG^2$, there are $m_e$ choices for $g_2$ when $e_2=e$, and $m_{\overline{e}}$ choices for $g_2$ when $e_2=\overline{e}$. So each vertex of the form $ge\in v\GG^1$ is the range of one edge with source $1G_v$, and $m_e + (m_{\overline{e}}-1)$ other edges; that is, it has valence $m$. A similar argument shows that every vertex of the form $g\overline{e}\in v\GG^2$ also has valence $m$. Continuing with the same reasoning, we see that every vertex in the tree $X_{\GG,v}$ has valence $m$, and so $X_{\GG,v}$ is the $m$-regular tree.
\end{enumerate}
\end{examples}

The fundamental group $\pi_1(\GG,v)$ acts on the Bass--Serre tree $X_{\GG,v}$ as follows.  Let $\gamma \in \pi_1(\GG,v) = \pi[v,v]$.  Then for all $\gamma' \in \pi[v,x]$, we have $\gamma\gamma' \in \pi[v,x]$.  Hence there is a natural left action of $\pi_1(\GG,v)$ on the vertex set of $X_{\GG,v}$.  Explicitly, $\gamma \cdot \gamma'G_x = \gamma\gamma' G_x$.  It can be verified that this action extends to an action on the edges of the tree $X_{\GG,v}$.  

\begin{rmk}\label{rmk: action}
In terms of the inverse system~\eqref{E: inverse}, the action of $\pi_1(\GG,v)$ on $X_{\GG,v}$ is more difficult to describe explicitly.  Let $\gamma \in \pi_1(\GG,v)$ be nontrivial.  Then $\gamma$ is represented by a unique reduced $\GG$-loop $g_1$ or $g_1e_1g_2e_2\dots g_me_m g_{m+1}$ based at $v$.  

We first consider the action on $v\GG^0$, which we identify with the base vertex $1G_v$ of $X_{\GG,v}$.  If $\gamma$ is represented by $g_1 \in G_v$ then $\gamma \cdot 1G_v = g_1G_v = 1G_v$ so $\gamma$ fixes $1G_v$.  Otherwise, 
\[
\gamma \cdot 1 G_v = g_1e_1g_2e_2\dots g_me_m g_{m+1} G_v = g_1e_1g_2e_2\dots g_me_m G_v
\]
since $g_{m+1} \in G_v$, and so $\gamma$ takes the base vertex $1G_v$ to the vertex $g_1e_1g_2e_2\dots g_me_m\in v\GG^m$.  

For $n > 0$, let $g_1'e_1'g_2'e_2'\dots g_n'e_n' \in v\GG^n$.  The concatenation 
\[
g_1g_1'e_1'g_2'e_2'\dots g_n'e_n' \quad\text{ or }\quad g_1e_1g_2e_2\dots g_me_m g_{m+1}g_1'e_1'g_2'e_2'\dots g_n'e_n'
\]
will not in general be a $\GG$-path, but by repeated application of relations in the vertex groups of $\GG$ together with relations (R1) and (R2), it can be transformed into a reduced $\GG$-word of the form 
\[
g_1''\quad\text{or}\quad g_1'' e_1'' \dots g_k'' e_k''\quad\text{or}\quad g_1'' e_1'' \dots g_k'' e_k'' g_{k+1}'',
\]
with range $v$. For the case of a single group element $g_1''$, we take the $\GG$-path $1_v\in v\GG^0$. A word of the form $g_1'' e_1'' \dots g_k'' e_k''$ is already a $\GG$-path and so is in $v\GG^k$.  In the remaining case, we remove the final group element $g_{k+1}''$ to obtain an element of $v\GG^k$. Hence in all cases the image can be determined.
\end{rmk}



If we change base vertex from $v \in \G^0$ to $v' \in \G^0$, then there is a $(\pi_1(\GG,v),\pi_1(\GG,v'))$-equivariant isomorphism of Bass--Serre  trees $X_{\GG,v} \to X_{\GG,v'}$ (see 1.22 of \cite{Bass}).  When the base vertex $v$ is clear, we may just write $X_\GG$ for the Bass--Serre tree.

We conclude this section with a theorem sometimes known as the Fundamental Theorem of Bass--Serre Theory.  Roughly speaking, this states that graphs of groups encode group actions on trees.  This theorem is important for the later sections of this paper when we wish to compare dynamical properties of the boundary action with combinatorial properties of the corresponding graph of groups. We do not define all of the terms in this result, and refer the reader to~\cite{Bass}. We do however explain how a group acting on a tree induces a graph of groups.  The action of a group $G$ on a tree $X$ is said to be \emph{without inversions} if $g e \neq \overline e$, for all $g \in G$ and $e \in X^1$. (Note that this is a mild restriction, because an action always induces an action without inversion on a tree obtained by subdividing edges of the original tree.) If $G$ acts on $X$ without inversions, then there is a well-defined quotient graph $\G := G\backslash X$, and the $G$-action induces a graph of groups $\GG = (G,\G)$ as follows.  Let $p: X \to G\backslash X$ be the natural projection.  For each $x \in \G^0$ choose an element $x' \in p^{-1}(x)$, and for each $e \in \G^1$ choose an element $e' \in p^{-1}(e)$.  Then define the vertex group $G_x$ to be $\Stab_G(x')$, the stabiliser in $G$ of $x'$, and the edge group $G_e$ to be $\Stab_G(e')$.  For $e \in \G^1$, the monomorphism $\alpha_e$ is defined as follows.  Suppose $r(e) = x$.  Then by definition of the quotient map, there is $g = g_{e',x'} \in G$ so that $g \cdot r(e') = x'$.  Now if a group element fixes the edge $e'$ it fixes the vertex $r(e')$, so
$$ g G_e g^{-1} = g\Stab_G(e')g^{-1} \leq g \Stab_G(r(e')) g^{-1} = \Stab_G(x') = G_x = G_{r(e)},$$
and we may define $\alpha_e:G_e \to G_{r(e)}$ by $h \mapsto ghg^{-1}$.

\begin{thm}\label{thm: BS theorem}    Let $\GG = (\G,G)$ be a graph of groups, and choose a base vertex $v \in \G^0$.  Then the action of the fundamental group $\pi_1(\GG,v)$ on the Bass--Serre tree $X_{\GG,v}$ induces a graph of groups isomorphic to $\GG$.  Conversely, if $G$ is a group acting without inversions on a tree $X$ with quotient graph $\G = G\backslash X$, and $\GG = (\G,G)$ is a graph of groups induced by this action, then for all $v \in \G^0$ there is an isomorphism of groups $\pi_1(\GG,v) \cong G$ and an equivariant isomorphism of trees $X_{\GG,v} \cong X$.
\end{thm}

\begin{rmk}\label{rem:LatticeStuff}
As mentioned in the introduction, one key application of graphs of groups is to the study of lattices in automorphism groups of trees (a reference for this theory is~\cite{BL}).  Recall that if $G$ is a locally compact group, a \emph{lattice} in $G$ is a discrete subgroup $\Lambda < G$ so that $G/\Lambda$ admits a $G$-invariant measure of finite volume.  A lattice $\Lambda < G$ is \emph{uniform} if $G/\Lambda$ is compact, and otherwise is \emph{nonuniform}.

Now let $T$ be a locally finite tree. Then its automorphism group $\Aut(T)$ is naturally a locally compact group, when equipped with the compact-open topology.  In this topology, if $\Lambda$ is a subgroup of $\Aut(T)$ then $\Lambda$ is discrete if and only if $\Lambda$ acts on $T$ with finite vertex stabilisers, $\Lambda$ is a uniform lattice in $\Aut(T)$ if and only if $\Lambda$ is discrete and the quotient graph $\Lambda \backslash T$ is finite, and $\Lambda$ is a nonuniform lattice in $\Aut(T)$ if and only if $\Lambda$ is discrete, the quotient graph $\Lambda \backslash T$ is infinite, and the series $\Sigma \frac{1}{|\Lambda_v|}$ converges, where this sum is taken over a set of representatives of the $\Lambda$-orbits on $T^0$ (see Chapter 1 of~\cite{BL}).  Converting these to statements about a graph of groups, say $\GG_\Lambda = (\G,G)$, induced by the $\Lambda$-action, we see that $\Lambda$ is discrete exactly when $\GG_\Lambda$ is a graph of finite groups, $\Lambda$ is a uniform lattice exactly when $\GG_\Lambda$ is a finite graph of finite groups, and $\Lambda$ is a nonuniform lattice exactly when $\GG_\Lambda$ is an infinite graph of finite groups such that the series $\Sigma \frac{1}{|G_x|}$ converges, where $x$ now runs over all vertices of $\G = \Lambda \backslash T$.   The tree $T$ is said to be a \emph{uniform tree} if $\Aut(T)$ admits at least one uniform lattice, which is equivalent to $T$ being the Bass--Serre tree for some finite graph of finite groups.
\end{rmk}

\subsection{Boundaries of trees}\label{subsec: boundaries of trees}

We first give general definitions concerning boundaries of trees, then discuss in Section~\ref{subsubsec: BS tree} the situation where the tree $X$ is the Bass--Serre tree for a graph of groups $\GG$.

Let $X$ be a locally finite, nonsingular tree, and choose a base vertex $x \in X^0$.    We now allow infinite paths, and so a \textit{path} in $X$ is either a finite path of length $n \geq 0$ as defined in Section~\ref{subsec: Graphs}, or an infinite sequence of edges $e_1 e_2 \dots$ such that $s(e_i) = r(e_{i+1})$ for $i \geq 1$; this path is \textit{reduced} if $e_{i+1} \not= \overline{e_i}$ for all $i \geq 1$, and it has \emph{range} $r(e_1)$.  We let $|\alpha|$ denote the length of a finite reduced path $\alpha$, and for $n \geq 0$ write $xX^n$ for the set of reduced paths of length $n$ with range $x$ (thus $xX^0 = \{ x\}$).   Since $X$ is locally finite, the set $xX^n$ is finite for each $n$.   We denote by $x X^*$ the set of all finite reduced paths in $X$ with range~$x$, that is, $xX^* = \cup_{n=0}^\infty \, xX^n$.   

\begin{dfn}\label{def: boundary}  The \textit{boundary (from $x$)} of $X$ is the set of infinite reduced paths with range $x$, and is denoted $x\partial X$.  
\end{dfn}

For a finite reduced path $\mu \in xX^*$, we define the \emph{cylinder set} $Z(\mu)$ to be the elements of $x\partial X$ that extend $\mu$.   Since $X$ is nonsingular, the set $Z(\mu)$ is nonempty for all such $\mu$.  The collection $\{ Z(\mu) : \mu \in xX^* \}$ is a base for a totally disconnected compact Hausdorff topology on $x\partial X$, coinciding with the cone topology as described below.  

\begin{rmk} \label{rmk: visualboundary}  As we now explain, Definition~\ref{def: boundary} of the boundary of the tree $X$ is consistent with the definition of the visual boundary of a tree in geometric group theory.  We follow Chapter II.8 of the reference~\cite{BH}, which considers the more general setting of CAT(0) spaces.  Equip the tree $X$ with its usual metric $d$, in which each edge has length one.  A \emph{geodesic ray} is then a map $c:[0,\infty) \to X$ such that for all $0 < t, t' <\infty$, we have $|t - t'| = d(c(t), c(t'))$.  Two geodesic rays $c, c'$ are said to be \emph{equivalent} if the function $t \mapsto d(c(t), c'(t))$ is bounded.  Note that since $X$ is a tree, geodesic rays $c$ and $c'$ are equivalent if and only if their images eventually coincide.  

The \emph{visual boundary} of $X$, often denoted $\partial_\infty X$ or just $\partial X$, is the set of equivalence classes of geodesic rays.   The visual boundary is sometimes also called the \emph{ideal boundary} or \emph{Gromov boundary}, and it coincides with the set of \emph{ends} of the tree.  If $c$ is a geodesic ray we write $c(\infty)$ for the boundary point that it represents. 
  It is a basic result that given any $\xi \in \partial_\infty X$, and any base vertex $x \in X^0$, there is a unique geodesic ray $c$ such that $c(0) = x$ and $c(\infty) = \xi$.  That is, for all $x \in X^0$, every point on the boundary is represented by a unique geodesic ray from $x$.  
  
  The visual boundary $\partial_\infty X$ can be equipped with the \emph{cone topology}.  In this topology, if $\xi  \in \partial_\infty X$ is represented by the geodesic ray $c$ with $c(0) = x$, a basic neighbourhood of $\xi$ has the form:
\[
U(\xi,n) := \{ \xi' \in \partial_\infty X \mid \mbox{ $\xi' = c'(\infty)$ for a geodesic ray $c'$ such that $c|_{[0,n]} = c'|_{[0,n]}$} \},
\]
where $n \in \mathbb{N}$.  
\end{rmk}

\subsubsection{The boundary of the Bass--Serre tree}\label{subsubsec: BS tree}
We now consider the special case that the tree $X$ is the Bass--Serre tree $X_{\GG,v}$ for a locally finite, nonsingular graph of groups $\GG = (\G, G)$.  For each $n$, the set of reduced paths $(1G_v)X_{\GG,v}^n$ of length $n$ which have range its base vertex $1G_v$ can be identified with the set of $\GG$-paths $v\GG^n$.  We can then further identify the boundary $(1G_v)\partial X_{\GG,v}$ with the set of \emph{infinite reduced $\GG$-words with range $v$}, which is the set of all infinite sequences $g_1e_1g_2e_2\dots$ such that each initial finite subsequence of the form $g_1 e_1 \dots g_n e_n$ is an element of $v\GG^n$. (More generally, an \emph{infinite reduced $\GG$-word} is an infinite sequence $g_1 e_1 g_2 e_2 \dots$ such that each initial finite subsequence $g_1 e_1 \dots g_n e_n$ is in $\GG^*$. The range map $r$ extends to infinite reduced words in the obvious way: $r(g_1e_1g_2e_2\dots):=r(e_1)$.)

The action of the fundamental group $\pi_1(\GG,v)$ on the tree $X_{\GG,v}$ can be described via the sets $v\GG^n$ (see Remark~\ref{rmk: action}). Note that this action does not in general fix the base vertex $vX^0$, since the stabiliser of $1G_v$ is the vertex group $G_v$ of the graph of groups $\GG$, and $G_v$ is in general a proper subgroup of $\pi_1(\GG,v)$.  Hence this action does not in general take an element of $v\GG^n$ to an element of $v\GG^n$.  However, as we now describe, there is an induced action of $\pi_1(\GG,v)$ on the boundary $(1G_v)\partial X_{\GG,v}$.  

Let $\gamma$ be a nontrivial element of $\pi_1(\GG,v)$ and let $\xi \in (1G_v)\partial X_{\GG,v}$.  Then $\gamma$ is represented by a unique reduced $\GG$-loop based at $v$, and $\xi$ corresponds to a unique infinite reduced $\GG$-path with range $v$. The infinite $\GG$-word consisting of the concatenation of the reduced $\GG$-loop $\gamma$ with the reduced $\GG$-path $\xi$ will still have range $v$, but will not in general be reduced.  However by possibly infinitely many applications of the relations in the vertex groups of $\GG$ and relations (R1) and (R2), this concatenation can be transformed into an infinite reduced $\GG$-path with range $v$, say $\xi'$.  It can be verified that this procedure of concatenation and then reduction taking $\xi$ to $\xi'$ does indeed define an action of the group $\pi_1(\GG,v)$ on the set $(1G_v)\partial X_{\GG,v}$.  Moreover, the image of a cylinder set $Z(\mu)$ under this action is a union of cylinder sets, and so the fundamental group $\pi_1(\GG,v)$ acts on the boundary $(1G_v)\partial X_{\GG,v}$ by homeomorphisms.

\begin{SN*}
From here on we denote the boundary of the Bass--Serre tree discussed above by $v\partial X_\GG$. In most cases we do not use any special notation to denote the action of fundamental group elements on boundary points or cylinder sets; for instance, the action of $\gamma\in \pi_1(\GG,v)$ on $\xi\in v\partial X_\GG$ is denoted $\gamma\xi$, and the action of $\gamma$ on a cylinder set $Z(\mu)$ is denoted $\gamma Z(\mu)$.
\end{SN*}

\subsection{The fundamental groupoid}\label{subsec: groupoid}

It will be very useful to extend the notions described above by considering the \textit{fundamental groupoid} of $\GG$, which we denote $F(\GG)$. The groupoid approach will simplify calculations involving the action of $\pi_1(\GG)$ on $v \partial X_\GG$. (See \cite{H} for the fundamental groupoid of a graph of groups, and also \cite{Ren} for the general theory of groupoids).

The fundamental groupoid $F(\GG)$ is given by generators and relations as follows.  The generating set is the same as for the path group $\pi(\GG)$ of Definition \ref{def: path group}.  The relations are the same as for $\pi(\GG)$ except that relation  (R1) is omitted. The objects are identified with $\Gamma^0$ as $F(\GG)^0 = \{ 1_x : x \in \Gamma^0 \}$.  For $x \in \Gamma^0$ the range and source maps on $G_x$ are given by $r(G_x) = x = s(G_x)$.  It follows from (R2) that $e \overline{e} = 1_{r(e)}$, so that $e$ and $\overline{e}$ are inverse elements in the sense of groupoids.  The theorem of \cite{H} implies that $F(\GG)$ may be identified with the set of all reduced $\GG$-words; in fact a reduced $\GG$-word is precisely what is termed the \textit{normal form} of an element of the fundamental groupoid in \cite{H}. The isotropy at $x \in \Gamma^0$ is $\pi_1(\GG,x)$, that is, the group of all reduced $\GG$-words with source and range equal to $x$.
 
For $x, y \in \G^0$, the collection of cosets $x F(\GG) y / G_y$ may be identified with the set of all $\GG$-paths having range $x$ and source $y$.  We define the set $W_\GG$ and its subset $x W_\GG$ as follows: $$W_\GG := \bigsqcup_{x,y \in \Gamma^0} x F(\GG) y / G_y \quad \mbox{and} \quad x W_\GG = \bigsqcup_{y \in \Gamma^0} x F(\GG) y / G_y.$$  Then $x W_\GG$ may be identified with the set of all $\GG$-paths with range $x$, and so $x W_\GG$ may be identified with the Bass--Serre tree based at $x$, as given in Definition \ref{def: BS Tree}.  The set $W_\GG$ is then a bundle of trees, i.e. a forest, fibred over $\Gamma^0$. 

The boundary of the tree $x W_\GG$ will be denoted $x \partial W_\GG$, and it may be identified with the set of all infinite reduced $\GG$-paths having range $x$. We let $\partial W_\GG$ denote the disjoint union of the boundaries of the trees in the forest $W_\GG$, that is, $\partial W_\GG = \bigsqcup_{x \in \Gamma^0} x \partial W_\GG$.  Then $\partial W_\GG$  is a locally compact Hausdorff space, and is compact if and only if $\Gamma^0$ is finite.  

Now the fundamental groupoid $F(\GG)$ acts on the bundle of trees $W_\GG$, and hence also acts on $\partial W_\GG$ (cf. \cite{KMRW}, p. 912).  The action is written the same as just before Remark~\ref{rmk: action}, that is, $\gamma \cdot \gamma' G_x = \gamma\gamma' G_x$, but $\gamma$ and $\gamma'$ are now any reduced $\GG$-words such that $s(\gamma) = r(\gamma')$.  Since the range map $r : \partial W_\GG \to F(\GG)^0$ is continuous and open, the fibred product groupoid $F(\GG) * \partial W_\GG$ (see \cite{MRW} for details) is again a Hausdorff \'etale groupoid, with unit space $\partial W_\GG$.  

We use the following lemma concerning the action of $F(\GG)$ on $\partial W_\GG$.

\begin{lem} \label{lem: actionmoves1} Let $e,f \in \G^1$ with $s(f) = r(e)$ and let $g \in \Sigma_{r(e)}$ and $h \in \Sigma_{r(f)}$.  Let $\mu$ be any $\GG$-path with range $s(e)$.   
\begin{enumerate}
\item  \label{lem: actionmoves1.1} If $g e \not= 1 \overline{f}$, then $h f Z(g e \mu) = Z(h f g e \mu)$.
\item  \label{lem: actionmoves1.2} If $|\mu| \ge 1$, then $1 \overline{e} Z(1 e \mu) = Z(\mu)$.
\item  \label{lem: actionmoves1.3} $1 e Z(1 \overline{e}) = r(e) \partial W_\GG \setminus Z(1 e).$ 
\item  \label{lem: actionmoves1.4} If $f \not= \overline{e}$, then $
 h f 1 e Z(1 \overline{e})  = r(f) \partial W_\GG \setminus Z(h f 1 e)$.
\end{enumerate}
\end{lem}

\begin{proof}
Parts \eqref{lem: actionmoves1.1} and \eqref{lem: actionmoves1.2} are clear.  For part \eqref{lem: actionmoves1.3}, note that an infinite $\GG$-path in $Z(1 \overline{e})$ has the form $\xi = 1 \overline{e} g' e' \dots$ with the only restriction being $g' e' \not= 1 e$.  Thus removing the initial $1 \overline{e}$ from such $\xi$ leaves all infinite $\GG$-paths in $s(\overline{e}) \partial W_\GG$ not beginning with $1 e$, i.e. the paths in $r(e) \partial W_\GG \setminus Z(1 e)$.  Part \eqref{lem: actionmoves1.4} follows from \eqref{lem: actionmoves1.3} and the fact that pre-appending $h f$ to reduced infinite $\GG$-paths in $s(f) \partial W_\GG$ is a bijection.
\end{proof}

We now prove a lemma making explicit how the generators of $\pi_1(\GG,v)$ act on certain cylinder sets in $v\partial X_\GG$. Note that $v \partial X_\GG$ and $v \partial W_\GG$ denote the same space. The generators $\varepsilon(x,g)$ and $\varepsilon(e)$ in the following result are as defined in Notation~\ref{ntn: general notation} above, using a (fixed) choice of a maximal tree $T$ in $\G$.

\begin{lem} \label{lem: actiondetails}  Let $v,x \in \G^0$, $g \in G_x$ and $e \in \G^1$. 
\begin{enumerate}
\item \label{lem: actiondetails.1} If $\mu$ is a $\GG$-path with $r(\mu) = x$ and $|\mu| \ge 1$, then $$
\varepsilon(x,g) Z([v,x]\mu)
= Z([v,x] g \mu).$$  
\item \label{lem: actiondetails.2} If $x \not= v$ let $f$ be the rangemost edge of $[x,v]$.  (Thus $f$ is the unique edge in $\Gamma^1$ such that $r(f) = x$ and $\underleftarrow{f} = \overline{f}$.)  Then
$$\varepsilon(x,g) Z([v,x])
= \begin{cases}
 Z([v,x] g f)^c &\text{ if } x \not= v \text{ and } g \notin \alpha_f(G_f) \\
 Z([v,x]) &\text{ if } x \not= v \text{ and } g \in\alpha_f(G_f) \\
 v \partial W_\GG, &\text{ if } x = v.
 \end{cases}$$
\item \label{lem: actiondetails.3} If $\mu$ is a $\GG$-path with $r(\mu) = s(e)$, $|\mu| \geq 1$ and $\mu \neq 1\overline{e}$, then $$\varepsilon(e) Z([v,s(e)] \mu) = Z([v,s(e)] 1 e \mu).$$
\item \label{lem: actiondetails.4}  If $e \in T^1$ then $$\varepsilon(e) Z([v,s(e)] 1 \overline{e}) = 
 Z([v,s(e)] 1 \overline{e}) = Z([v,r(e)].$$  If $e \notin T^1$ then $$\varepsilon(e) Z([v,s(e)] 1 \overline{e}) = Z([v,r(e)] 1 e)^c.$$
\item \label{lem: actiondetails.5}  If $e \in T^1$ or $s(e) = v$ then 
$$\varepsilon(e) Z([v,s(e)])
= Z([v,s(e)]).$$  If $e \notin T^1$ and $s(e) \neq v$ then $$\varepsilon(e) Z([v,s(e)]) = 
 Z([v,r(e)] 1 e 1 f)^c,$$   where $f \in \G^1$ is defined by $\overline{f} = \underleftarrow{(\overline{e})}$.
\end{enumerate}

\end{lem}

\begin{proof}
Part \eqref{lem: actiondetails.1} follows from Lemma \ref{lem: actionmoves1}\eqref{lem: actionmoves1.1}.  For part \eqref{lem: actiondetails.2}, first suppose that $x \not= v$, and let $f$ be as in the statement.  In the fundamental groupoid we compute
\begin{align*}
[v,x] g [x,v] Z([v,x])
&=  [v,x] g f [s(f),v] Z([v,s(f)] 1 \overline{f}) \\
&=  [v,x] g f Z(1 \overline{f}), \mbox{ by Lemma \ref{lem: actionmoves1}\eqref{lem: actionmoves1.1},} \\
&=  [v,x] g \left(r(f) \partial W_\GG \setminus Z(1 f) \right), \mbox{ by Lemma \ref{lem: actionmoves1}\eqref{lem: actionmoves1.3}.} \end{align*}
If $g \not\in \alpha_f(G_f)$, then this equals $Z([v,x] g f)^c$, by Lemma \ref{lem: actionmoves1}\eqref{lem: actionmoves1.4}, while if $g \in \alpha_f(G_f)$ we have
\begin{align*}
[v,x] (r(f) \partial W_\GG \setminus Z(1f)) 
&= [v,s(f)] 1 \overline{f} (r(f) \partial W_\GG \setminus Z(1 f)) \\
&= [v,s(f)] Z(1 \overline{f}), \mbox{ by Lemma \ref{lem: actionmoves1}\eqref{lem: actionmoves1.3},} \\
&= Z([v,s(f)] 1 \overline{f}) \\
&= Z([v,r(f)]).
\end{align*}
If $x = v$ then $Z([v,x]) = Z([v,v]) = v \partial W_\GG$.  Since $\varepsilon(v,g)$ is a homeomorphism of $v \partial W_\GG$, we have that $\varepsilon(v,g) Z([v,v]) = v \partial W_\GG$.
For part \eqref{lem: actiondetails.3}, we have
\[
[v,r(e)] 1 e [s(e),v] Z([v,s(e)] \mu)
= [v,r(e)] 1 e Z(\mu)
= Z([v,r(e)] 1 e \mu)
\]
by Lemma \ref{lem: actionmoves1}\eqref{lem: actionmoves1.2} then  \ref{lem: actionmoves1}\eqref{lem: actionmoves1.1}.  

For part \eqref{lem: actiondetails.4}, if $e \in T^1$ then $\varepsilon(e)$ is trivial, and
\[
\varepsilon(e) Z([v,s(e)] 1 \overline{e})
= Z([v,s(e)] 1 \overline{e}) = Z([v,r(e)]) 
\]
as required.  If $e \not\in T^1$, then 
\begin{align*}
[v,r(e)]1e[s(e),v]Z([v,s(e)]1\overline{e})&=  [v,r(e)] 1 e Z(1 \overline{e}) \\
&=  [v,r(e)] (r(e) \partial W_GG \setminus Z(1 e)), \text{ by Lemma \ref{lem: actionmoves1}\eqref{lem: actionmoves1.3},} \\
&= Z([v,r(e)] 1 e)^c, \text{ by Lemma \ref{lem: actionmoves1}\eqref{lem: actionmoves1.1}.}
\end{align*}

For part \eqref{lem: actiondetails.5}, the result is immediate if $e \in T^1$.  Now assume that $e \not\in T^1$.  If $s(e) = v$, then $[v,s(e)] = v$ and so $Z([v,s(e)]) = v \partial W_\GG$.  Since $\varepsilon(e) \in \pi_1(\GG,v)$ acts homeomorphically on this space, we get
\[
[v,r(e)] 1 e [s(e),v] Z([v,s(e)])= v \partial W_\GG = Z([v,s(e)]). \]
If $s(e) \not= v$, then defining $f$ by $\overline{f} = \underleftarrow{(\overline{e})}$, we have $[v,s(e)] = [v,s(f)] 1 \overline{f}$.  Then
\begin{align*}
[v,r(e)] 1 e [s(e),v] Z([v,s(e)])
&= [v,r(e)] 1 e 1 f [s(f),v] Z([v,s(f)] 1 \overline{f}) \\
&= [v,r(e)] 1 e (s(e) \partial W_\GG \setminus Z(1 f) ) \\
&=  Z([v,r(e)] 1 e 1 f)^c. \qedhere
\end{align*}
\end{proof}

\subsection{$C^*$-algebra background} \label{subsec: cstar background}

We present some facts from $C^*$-algebra theory that are essential for the results in the paper.  We mention references for more details.


\subsubsection{Bounded linear operators on Hilbert space}\label{subsubsec:BLO}
$C^*$-algebras are the abstract characterisation of norm-closed self-adjoint subalgebras of $B(H)$, the algebra of all bounded linear operators on a complex Hilbert space $H$.  Here ``self-adjoint'' means ``closed under the operator adjoint $T \mapsto T^*$''.  Thus a $C^*$-algebra is a complete normed algebra equipped with an involution, whose norm satisfies $\| x^* x \| = \| x \|^2$.  The beginning of the subject is the proof that this identity does in fact characterise closed self-adjoint subalgebras of $B(H)$.  A homomorphism in the category of $C^*$-algebras is an algebra homomorphism that preserves the involution.  It is a fact that a nonzero homomorphism is a bounded map of norm one.

Certain types of operators in $B(H)$ can be characterised generally in $C^*$-algebras.  A {\em unitary} element satisfies $u^* u = u u^* = 1$ and an {\em isometry} satisfies $s^* s = 1$ (only possible in unital $C^*$-algebras).  A {\em projection} satisfies $p^2 = p = p^*$.  A {\em partial isometry} is an element $s$ such that $s^* s$ is a projection.  If $s$ is a partial isometry then so is $s^*$.  The projections $s^*s$ and $ss^*$ are called the {\em initial} and {\em final} projections of $s$.  A {\em partial unitary} is a partial isometry whose initial and final projections are equal; thus a partial unitary $u$ in a $C^*$-algebra $A$ is a unitary in the {\em corner algebra} $pAp$, where $p := u^* u$.  Details about such elements can be read in \cite{CBMS}.

\subsubsection{Fundamental examples}\label{subsubsec:FundamentalExamples}
Two fundamental examples of $C^*$-algebras are the algebra $\KK(H)$ of compact operators on a Hilbert space $H$, and the algebra $C_0(X)$ of continuous complex-valued functions vanishing at infinity on a locally compact Hausdorff space $X$.  We briefly discuss these examples.  

The algebra of compact operators may be defined as the norm closure of the algebra of finite-rank operators on $H$.  It is a {\em simple} $C^*$-algebra, in that it has no nontrivial closed 2-sided ideals.  The algebra $\KK(H)$ is determined up to isomorphism solely by the dimension of $H$, i.e. the cardinality of an orthonormal basis.  In this paper we will consider only {\em (topologically) separable} $C^*$-algebras. The algebra $\KK(H)$ is separable if and only if $H$ is separable, i.e. if and only if $H$ has a countable orthonormal basis.  We will write $\KK$ for the algebra of compact operators on a separable infinite dimensional Hilbert space.  

It is useful to realise $\KK$ as the inductive limit of finite-dimensional matrix algebras $M_n(\C)$.  For any $C^*$-algebra $A$, the tensor product $A \otimes \KK$ is then the inductive limit of matrix algebras  $M_n(A)$ over $A$.  The algebra $A \otimes \KK$ is called the {\em stabilisation} of $A$.  Two $C^*$-algebras are called {\em stably isomorphic} if their stabilisations are isomorphic.  Stably isomorphic $C^*$-algebras have many properties in common (some to be described below), such as simplicity, nuclearity, pure infiniteness, and isomorphic $K$-theory.  We will use these facts freely throughout the paper.

The algebra $C_0(X)$ is a $C^*$-algebra when equipped with the supremum norm, and with involution given by pointwise complex conjugation.  These are precisely the {\em commutative} $C^*$-algebras.  The algebra $C_0(X)$ is a unital algebra if and only if $X$ is compact.

\subsubsection{Crossed products}\label{subsubsec:CrossedProducts}
Let $\Lambda$ be a discrete group, $A$ a unital $C^*$-algebra, and $\alpha : \Lambda \to \Aut(A)$ a group homomorphism; that is, an {\em action} of $\Lambda$ on $A$. The triple $(A,\Lambda,\alpha)$ is referred to as a {\em $C^*$-dynamical system}.  A {\em covariant representation} of a $C^*$-dynamical system $(A,\Lambda,\alpha)$ in a unital $C^*$-algebra $B$ consists of a unital $*$-homomorphism $j_A : A \to B$ and a group homomorphism $j_\Lambda : \Lambda \to \UU(B)$ such that
\[
j_\Lambda(t) j_A(x) j_\Lambda(t^{-1}) = j_A(\alpha_t (x)),\quad\text{ for all $t \in \Lambda$ and $x\in A$.}
\]
(Here, $\UU(B)$ is the group of unitary elements of $B$.)  A covariant representation $(i_A, i_\Lambda)$ in a $C^*$-algebra $C$ is {\em universal} if for every covariant representation $(j_A,j_\Lambda)$ in a $C^*$-algebra $B$, there is a unique homomorphism $\pi : C \to B$ such that $\pi \circ i_B = j_B$ and $\pi \circ i_\Lambda = j_\Lambda$.  It is a fundamental result that a universal covariant representation exists, and is unique. We denote this unique covariant representation by $(i_A,i_\Lambda)$. The $C^*$-algebra $C$ is generated as a $C^*$-algebra by $i_A(A) \cup i_\Lambda(\Lambda)$, is denoted $A \rtimes_\alpha \Lambda$, and is called the {\em (full) crossed product of $A$ by (the action $\alpha$ of) $\Lambda$}.  

There is a related construction called the {\em reduced crossed product}, which is denoted $A \rtimes_{\alpha,r} \Lambda$. Let $\pi_0 : A \to B(H_0)$ be an injective unital homomorphism, for some Hilbert space $H_0$.  Let $H = \ell^2(\Lambda) \otimes H_0$, and define a covariant pair $j_A$, $j_\Lambda$ in $B(H)$ by 
\[
j_A(x) (\delta_t \otimes \eta) = \delta_t \otimes \pi_0(\alpha_{t^{-1}}(x)) \eta\quad\text{and}\quad j_\Lambda(s)(\delta_t \otimes \eta) = \delta_{st} \otimes \eta, 
\]
where $\{ \delta_t : t \in \Lambda \}$ is the standard orthonormal basis for $\ell^2(\Lambda)$.  It is easily checked that this is a covariant pair.  The reduced crossed product is defined to be the $C^*$-subalgebra of $B(H)$ generated by $j_A(A) \cup j_\Lambda(\Lambda)$; it is independent of the choice of $\pi_0$.  By the universal property, there is a surjective homomorphism $\pi : A \rtimes_\alpha \Lambda \to A \rtimes_{\alpha,r} \Lambda$.  It is known that if $\Lambda$ is an amenable group then $\pi$ is an isomorphism.  There is a notion of amenability for actions of groups on $C^*$-algebras; amenable groups always act amenably, but nonamenable groups may have some amenable actions.  It is a theorem that if the action is amenable then $\pi$ is an isomorphism (see \cite{A}).

Examples of crossed products come from actions on compact Hausdorff spaces $X$. Actions on $C(X)$ correspond to actions on $X$:  if $\phi : \Lambda \to \text{Homeo}(X)$ is a homomorphism, then there is an action $\alpha:\Lambda\to \Aut(A)$ given by $\alpha_t(f) = f \circ \phi_{t^{-1}}$. It is known that if $X$ is second countable and $\Lambda$ is countable, then the crossed product $C(X)\rtimes_\alpha\Lambda$ is separable.

\subsubsection{Other important properties}\label{subsubsec:OtherStuff}
There are several other properties that a $C^*$-algebra might have that are important for our results, but are more subtle.  We refer to \cite{BO}, \cite{Bla2}, and \cite{Ror} for details.  {\em Nuclearity} is the analogue of amenability for $C^*$-algebras.  We will not define it here, but we mention that the class of nuclear $C^*$-algebras contains all commutative $C^*$-algebras and finite dimensional $C^*$-algebras, and is closed under the formation of ideals, quotients, extensions, tensor products, inductive limits, and crossed products by amenable actions.  {\em Pure infiniteness} is an example of the high degree of infiniteness possible in $C^*$-algebras.  We give one of many equivalent formulations for the case of simple algebras.  A simple $C^*$-algebra is {\em purely infinite} if for any nonzero elements $a$, $b$ there are elements $x$, $y$ such that $b = xay$.  A more esoteric property is referred to as {\em UCT}, for Universal Coefficient Theorem.  It is an open question whether all nuclear $C^*$-algebras satisfy the UCT.  It is known that if a countable group acts amenably on a commutative $C^*$-algebra then the crossed product algebra satisfies the UCT (see \cite{T}).

The combination of these properties has a striking consequence, the Kirchberg--Phillips classification theorem.  A simple separable nuclear purely infinite $C^*$-algebra is called a {\em Kirchberg algebra}.  The theorem states that Kirchberg algebras satisfying the UCT are classifed up to stable isomorphism by their $K$-theory (see \cite{Ror}).

\section{The graph of groups $C^*$-algebra}\label{sec: C stars associated to gog}

We now introduce our main object of study, the graph of groups $C^*$-algebra $C^*(\GG)$.  In Section~\ref{subsec: gog algebra}, we define a family of partial isometries and partial unitaries called a $\GG$-family, and use this family to define $C^*(\GG)$.  Then in Section~\ref{subsec: directed graph} we identify a natural directed graph whose associated directed graph $C^*$-algebra sits faithfully inside $C^*(\GG)$.  We show that when the edge groups are trivial these algebras are equal, and use this result to see that $C^*(\GG)$ recovers some known classes of $C^*$-algebras.  We conclude with examples where $C^*(\GG)$ is strictly larger than the associated directed graph algebra. 

\subsection{$\GG$-families and the graph of groups $C^*$-algebra}\label{subsec: gog algebra}

We begin with the definition of a $\GG$-family, then define the graph of groups $C^*$-algebra $C^*(\GG)$.  In Remark~\ref{rem: G-family} we build a concrete $\GG$-family.

\begin{SA*}
Throughout the rest of this paper we will consider locally finite nonsingular graphs of groups $\GG=(\Gamma,G)$ in which $G_x$ is countable for each $x\in \Gamma^0$. We refer to a graph of groups with countable vertex groups as a {\em graph of countable groups}.
\end{SA*}

\begin{dfn}\label{def: a G, Sigma family}
For each $e \in \G^1$, choose a transversal $\Sigma_e$ for $G_{r(e)}/\alpha_e(G_e)$ so that $1 \in \Sigma_e$. A {\em $(\GG,\Sigma)$-family} is a collection of partial isometries 
$S_e$ for each $e\in\Gamma^1$ 
and representations $g\mapsto U_{x,g}$ of $G_x$ by partial unitaries for each 
$x\in\Gamma^0$ satisfying the relations:
\begin{itemize}
\item[(G1)] $U_{x,1}U_{y,1}=0$ for each $x,y \in\Gamma^0$ with 
$x \not= y $;
\item[(G2)] $U_{r(e),\alpha_e(g)}S_e=S_e U_{s(e),\alpha_{\overline{e}}(g)}$ 
for each $e\in\Gamma^1$ and $g\in G_e$;
\item[(G3)] 
$U_{s(e),1}=S_e^*S_e+S_{\overline{e}}S_{\overline{e}}^*$ for each 
$e\in\Gamma^1$;  and
\item[(G4)] $$S_e^*S_e=\sum_{\substack{ r(f)=s(e),\, h\in \Sigma_f \\ hf\not= 
1 \overline{e}}} U_{s(e),h}S_fS_f^*U_{s(e),h}^*$$ for each 
$e\in\Gamma^1$. 
\end{itemize}
\end{dfn}

Relation (G1) ensures that the representations of the vertex groups are mutually orthogonal. Relation (G2) is an analogue of the path group relation (R2). Relations (G3) and (G4) ensure that each collection $\{U_{x,h}S_f:hf\in x\GG^1\}$ consists of partial isometries with mutually orthogonal range projections. We also emphasise that (G3) implies that $S_eS_{\overline{e}}=0$ for all edges $e\in\Gamma^1$. This comes despite $s(e)=r(\overline{e})$, which will feel foreign to those used to Cuntz--Krieger families of directed graphs.  

\begin{rmk}\label{rem: G-family}
Relation (G4) is independent of the choice of transversals. Given a second choice of transversals $\{\Sigma_e':e\in\Gamma^1\}$, edges $e,f\in\Gamma^1$ with $r(f)=s(e)$, and $h'\in\Sigma_f'$, we write $h'=h\alpha_f(g)$ for some $h\in\Sigma_f$ and $g\in G_f$. Then (G2) gives 
\begin{align*}
U_{s(e),h'}S_fS_f^*U_{s(e),h'}^*&=U_{s(e),h}U_{s(e),\alpha_f(g)}S_fS_f^*U_{s(e),\alpha_f(g)}^*U_{s(e),h}^*\\
&=
U_{s(e),h}S_fU_{s(e),\alpha_{\overline{f}}(g)}U_{s(e),\alpha_{\overline{f}}(g)}^*S_f^*U_{s(e),h}^*\\
&=
U_{s(e),h}S_fS_f^*U_{s(e),h}^*.
\end{align*}
We see that (G4) is satisfied for one choice of transversals exactly when it is satisfied for all choices.\end{rmk}

Given Remark~\ref{rem: G-family}, from now on we call a family of partial isometries and partial unitaries as in Definition~\ref{def: a G, Sigma family} a {\em $\GG$-family}.  We can now define the graph of groups algebra.

\begin{dfn}\label{def: Gog C*}
Let $\GG=(G,\Gamma)$ be a locally finite nonsingular graph of countable groups. The {\em graph of 
groups algebra $C^*(\GG)$} is 
the universal $C^*$-algebra generated by a $\GG$-family, in the sense that $C^*(\GG)$ is generated by a $\GG$-family $\{u_x, s_e : x \in \Gamma^0, e \in \Gamma^1 \}$ such that if $B$ is a $C^*$-algebra, and if $\{U_x, S_e : x \in \Gamma^0, e \in \Gamma^1 \}$ is a $\GG$-family in $B$, there is a unique $*$-homomorphism from $C^*(\GG)$ to $B$ such that $u_x \mapsto U_x$ and $s_e \mapsto S_e$. (We refer to \cite{Bla} for the existence and uniqueness of $C^*(\GG)$.)
\end{dfn}

\begin{rmk}\label{rmk: concrete G family}
We can use the regular representations of the fundamental transformation groupoid $F(\GG) * \partial W_\GG$ as discussed in Section~\ref{subsec: groupoid} to build a concrete $\GG$-family. Let $v \in \Gamma^0$ and let $\xi \in v \partial W_\GG$.  Put $H_\xi = \ell^2(F(\GG) \xi)$ and for $\gamma \in F(\GG)$ let $\delta_{\gamma\xi}$ be the point-mass function.  We define a $\GG$-family in $B(H_\xi)$ by, for each $\gamma \in F(\GG)$, $e \in \G^1$, $x \in \G^0$ and $g \in G_x$, defining 
\begin{align*}
	S_e \delta_{\gamma \xi} &=
	\begin{cases}
		\delta_{1 e \gamma \xi} &\text{ if } \gamma \xi \in s(e) \partial W_\GG \setminus Z(1 \overline{e}) \\
		0 &\text{ otherwise,}
	\end{cases} \\
	U_{x,g}\delta_{\gamma \xi} &=
	\begin{cases}
		\delta_{g \gamma \xi} &\text{ if } r(\gamma) = x \\
		0 &\text{ otherwise.}
	\end{cases}
\end{align*}

We verify that this is indeed a $\GG$-family.  First, it follows immediately from the above formula for $U_{x,g}$ that (G1) holds.  Next note that
\[
S_e^* \delta_{\gamma \xi} =
\begin{cases}
\delta_{\xi'} &\text{ if } \gamma \xi = 1 e \xi' \text{ for some } \xi' \in s(e) \partial W_\GG \setminus Z(1 \overline{e}) \\
0 &\text{ if } r(\gamma) \not= r(e), \text{ or } \gamma \xi \in Z(1 \overline{e}).
\end{cases}
\]
It follows easily that $S_e S_e^*$ is the projection onto $\overline{\text{span}} \{\delta_{\gamma \xi} : \gamma \xi \in Z(1 e) \}$.  Observe that $U_{x,1}$ is the projection onto $\overline{\text{span}} \{\delta_{\gamma \xi} : r(\gamma) = x \}$, and hence $S_e^* S_e = U_{x,1} - S_{\overline{e}} S_{\overline{e}}^*$.  Therefore (G3) holds.  

Next we note that $U_{r(e),\alpha_e(g)} S_e \delta_{\gamma \xi} \not= 0$ if and only if $\gamma \xi \in s(e) \partial W_\GG \setminus Z(1 \overline{e})$, and in this case, 
\[
U_{r(e),\alpha_e(g)} S_e \delta_{\gamma \xi} = \delta_{\alpha_e(g) e \gamma \xi} = \delta_{1 e \alpha_{\overline{e}}(g) \gamma \xi}.
\]
We also note that $U_{s(e),\alpha_{\overline{e}}(g)} \delta_{\gamma \xi} = \delta_{\alpha_{\overline{e}}(g) \gamma \xi}$ if and only if $r(\gamma) = s(e)$.  Since $\alpha_{\overline{e}}(g) \gamma \xi \in Z(1 \overline{e})$ if and only if $\gamma \xi \in Z(1 \overline{e})$, it follows that $S_e U_{s(e),\alpha_{\overline{e}}(g)} \delta_{\gamma \xi} = 0$ if and only if $U_{r(e),\alpha_e(g)} S_e \delta_{\gamma \xi} = 0$, and if nonzero, they are equal.  Therefore (G2) holds.  Finally, we see from the above that $(U_{r(e),h} S_f)(U_{r(e),h} S_f)^*$ is the projection onto $\overline{\text{span}} \{\delta_{\gamma \xi} : \gamma \xi \in Z(h f) \}$.  Now (G4) follows from this observation, together with (G3).
\end{rmk}

\subsection{Relationship with directed graph $C^*$-algebras}\label{subsec: directed graph}

The main result of this section is Theorem~\ref{thm: associated directed graph} below, which identifies a natural directed graph built from $\GG$-paths, and whose associated directed graph $C^*$-algebra sits faithfully inside $C^*(\GG)$.   We also provide a class of graphs of groups for which this directed graph $C^*$-algebra is all of $C^*(\GG)$, use this result to see that $C^*(\GG)$ recovers some previously-studied algebras, and discuss two examples where $C^*(\GG)$ is strictly larger than the associated directed graph algebra.  We refer the reader to \cite[Page~6]{CBMS} for the Cuntz--Krieger relations (CK1) and (CK2) for a directed graph, and to \cite[Proposition~1.21]{CBMS} for the notion of a directed graph $C^*$-algebra. 

We start with some notation. 

\begin{ntn}\label{ntn: S mu}  
Let $\{U_{x},S_e: x \in\Gamma^0,e\in\Gamma^1\}$ be a $\GG$-family. For each $\GG$-path $\mu=g_1e_1\dots g_n e_n$ we define 
\[
S_\mu := U_{r(e_1),g_1}S_{e_1}\dots U_{r(e_n),g_n}S_{e_n}. 
\]
\end{ntn}

\noindent Each $S_\mu$ is a partial isometry, because for each $1 < i \le n$, the final projection of $U_{r(e_i),g_i} S_{e_i}$ is a subprojection of the initial projection of $U_{r(e_{i-1}),g_{i-1}} S_{e_{i-1}}$, by (G4). 

\begin{thm}\label{thm: associated directed graph}
	Let $\GG=(\Gamma,G)$ be a locally finite nonsingular graph of countable groups. Then $$E_\GG=(E_\GG^0:=\GG^1,E_\GG^1:=\GG^2,r_E,s_E)$$ where $s_E:\GG^2\to\GG^1$ is given by $s_E(g_1e_1g_2e_2)=g_2e_2$ and $r_E:\GG^2\to\GG^1$ is given by $r_E(g_1e_1g_2e_2)=g_1e_1$, is a row-finite directed graph with no sources. 
	
	Now let $\{p_\nu:\nu\in E_\GG^0\}$ and $\{t_\mu:\mu\in E_\GG^1\}$ be the Cuntz--Krieger $E_\GG$-family generating $C^*(E_\GG)$. There is an embedding $$\phi:C^*(E_\GG)\hookrightarrow C^*(\GG)$$ satisfying 
	\[
	\phi(p_\nu)=s_\nu s_\nu^*\quad\text{and}\quad \phi(t_\mu)=s_\mu s_{s_E(\mu)}^*,
	\]  	
	for all $\nu\in E_\GG^0$, $\mu\in E_\GG^1$. Moreover, if $G_e=\{1\}$ for each $e\in\Gamma^1$, then $\phi$ maps onto $C^*(\GG)$.
\end{thm}

To prove the injectivity of $\phi$ we will use the gauge-invariant uniqueness theorem for directed graph $C^*$-algebras \cite[Proposition 2.1]{CBMS}. To do this we need a gauge action on $C^*(\GG)$, by which we mean a strongly continuous action by automorphisms of $C^*(\GG)$ of the one-dimensional torus $\mathbb{T} = \{ z \in \mathbb{C} : |z| = 1\}$, regarded as a multiplicative locally compact group. That there is such an action is a direct consequence of the universal definition of $C^* ( \GG )$. We state this result without proof as Proposition~\ref{prop:gauge action}, as  it is by now a standard argument (see the proof of \cite[Proposition 2.1]{CBMS}, for instance).

\begin{prop} \label{prop:gauge action}
Let $\GG=(G,\Gamma)$ be a locally finite nonsingular graph of countable groups and $\{ u_x , s_e : x \in \Gamma^0 , e \in \Gamma^1 \}$ be the universal $\GG$-family generating $C^* ( \GG )$. There is a strongly continuous action $\gamma : \mathbb{T} \to \operatorname{Aut} C^* ( \GG )$
such that
\[
\gamma_z ( u_x ) = u_x \quad\mbox{and}\quad \gamma_z ( s_e ) = z s_e  \quad\text{ for all } x \in \Gamma^0 , e \in \Gamma^1, z \in \mathbb{T}.
\]
\end{prop}

\begin{rmk}\label{rem:CondExp}
	Given a strongly continuous action $\sigma_z$ of $\T$ by automorphisms of a $C^*$-algebra $A$ there is an associated conditional expectation $\phi : A \to A^0$, where $A^0$ is the algebra of fixed-points of the action, defined by $\phi(a) = \int_\T \sigma_z(a) \, dz$. The key properties of $\phi$ are that it is a positive norm-one linear idempotent with range $A^0$, and that it is {\em faithful}, in that if $\phi(a^* a) = 0$ then $a = 0$.  See \cite{BO} for details.
\end{rmk}

\begin{proof}[Proof of Theorem~\ref{thm: associated directed graph}]
	The directed graph $E_\GG$ is row-finite because $\GG$ is locally finite, and has no sources because $\GG$ is nonsingular. 
	
	We now claim that the elements
	\[
	P_\nu:=s_\nu s_\nu^*\quad \text{and}\quad T_\mu:=s_\mu s_{s_E(\mu)}^*,
	\]
	for $\nu\in E_\GG^0$, $\mu\in E_\GG^1$, satisfy relations (CK1) and (CK2) from \cite[Page~6]{CBMS}; that is, they form a Cuntz--Krieger $E_\GG$-family in $C^*(\GG)$. For each $\mu=g_1e_1g_2e_2\in E_\GG^1$ an application of (G4) shows that $s_\mu^*s_\mu=s_{e_2}^*s_{e_2}$. It follows that
	\[
	T_\mu^* T_\mu = s_{s_E(\mu)}s_\mu^* s_\mu s_{s_E(\mu)}^* = s_{s_E(\mu)}s_{e_2}^*s_{e_2}s_{s_E(\mu)}^*=s_{s_E(\mu)}s_{s_E(\mu)}^*=P_{s_E(\mu)},
	\]
	and so (CK1) holds. For each $\nu=g_1e_1\in\GG^1$ we use (G4) to get
	\begin{align*}
		\sum_{r_E(\mu)=\nu}T_\mu T_\mu^* = \sum_{r_E(\mu)=\nu}s_\mu s_{s_E(\mu)}^*s_{s_E(\mu)}s_\mu^* = \sum_{r_E(\mu)=\nu} s_\mu s_\mu^*
		&= s_\nu \left(\sum_{\substack{\nu'\in s(e_1)\GG^1 \\ \nu'\not= 1\overline{e_1}}}s_{\nu'} s_{\nu'}^* \right)s_\nu^*\\
		&= s_\nu s_{e_1}^* s_{e_1} s_\nu\\
		& = s_\nu s_\nu ^*\\
		& = P_\nu,
	\end{align*}
	and so (CK2) holds. By the universal property of $C^*(E_\GG)$, the claim holds and hence we get a homomorphism $\phi:C^*(E_\GG)\to C^*(\GG)$ satisfying $\phi(p_\nu)=s_\nu s_\nu^*$ for each $\nu\in E_\GG^0$, and $\phi(t_\mu)=s_\mu s_{s_E(\mu)}^*$ for each $\mu\in E_\GG^1$. 
	
	To show that $\phi$ is an embedding, we use the gauge-invariant uniqueness theorem for directed graph $C^*$-algebras (see \cite[Theorem~2.2]{CBMS}). We know from Proposition~\ref{prop:gauge action} that there is a gauge action $\gamma$ of $C^*(\GG)$, and it satisfies $\gamma_z(P_\nu)=P_\nu$ for each $\nu\in E_\GG^0$, and $\gamma_z(T_\mu)=zT_\mu$ for each $\mu\in E_\GG^1$. We have discussed a concrete $\GG$-family in Remark~\ref{rmk: concrete G family}; since in this family each projection $S_\nu S_\nu^*$ is nonzero, it follows that each $P_\nu=s_\nu s_\nu^*$ is nonzero. So we can apply the gauge-invariant uniqueness theorem to see that $\phi$ is an embedding.  
	
	For the last assertion we assume each edge group $G_e$ is trivial, and we claim that each $s_e$ and each $u_{x,g}$ is in the image of $\phi$. Let $e\in\Gamma^1$ and $g\in G_{r(e)}$, and we observe that for all edges $f\in\Gamma^1$ the transversal $\Sigma_f$ is just the vertex group $G_{r(f)}$. We then use (G4) to get
	\[
	s_{ge}=s_{ge} s_e^* s_e = \sum_{\substack{ r(f)=s(e),\, h\in G_{s(e)} \\ hf\not= 
			1 \overline{e}}} s_{ge} s_{hf}s_{hf}^* = \sum_{\substack{ r(f)=s(e),\, h\in G_{s(e)} \\ hf\not= 
			1 \overline{e}}} T_{gehf} = \phi\left(\sum_{\substack{ r(f)=s(e),\, h\in G_{s(e)} \\ hf\not= 
			1 \overline{e}}} t_{gehf}\right). 
	\]
	Taking $g=1$ shows that each $s_e$ is in the image of $\phi$. For each $x\in \Gamma^0$ and $g\in G_x$ we use (G3) and (G4) to get
	\[
	u_{x,g}=u_{x,g}\left(\sum_{\substack{r(f)=x \\ h\in G_x}}s_{hf}s_{hf}^*\right)= \sum_{\substack{r(f)=x \\ h\in G_x}}s_{(gh)f}s_{hf}^*.
	\]
	Since each $s_{(gh)f}, s_{hf}$ is in the image of $\phi$, it follows that $u_{x,g}$ is in the image of $\phi$. The claim holds, and hence $\phi$ is an isomorphism onto $C^*(\GG)$.
\end{proof}

\begin{rmk}\label{rem:identifyingCstarEsubGG}
From this point we identify $C^*(E_\GG)$ with the $C^*$-subalgebra of $C^*(\GG)$ generated by $\{s_\nu s_\nu^*:\nu\in\GG^1\}\cup \{s_\mu s_{s_E(\mu)}^*:\mu\in E_\GG^1\}$.
\end{rmk}

\begin{rmk}\label{rmk: known constructions}

	We can use Theorem~\ref{thm: associated directed graph} to see that our graph of groups $C^*$-algebras recover some known classes of $C^*$-algebras. 

	\begin{enumerate}

		\item[(1)] When all the vertex groups are trivial, and $\Gamma$ is a finite graph, the $C^*$-algebra $C^*(E_\GG)$, and hence by Theorem~\ref{thm: associated directed graph} the $C^*$-algebra $C^*(\GG)$, is the Cuntz--Krieger algebra associated to $\Gamma$ by Cornelissen, Lorscheid and Marcolli in \cite{CLM}.
		
		\item[(2)] Suppose that a group $G$ is the free product of finitely many finite groups $G_1,\dots,G_n$.  Then $G$ is naturally the fundamental group of a locally finite nonsingular graph of groups with trivial edge groups, as follows.  Let $\G$ be the star graph with $n+1$ vertices $x,x_1,\ldots,x_n$, and edge set $\{e_1,\dots,e_n,\overline{e_1},\dots,\overline{e_n}\}$ such that $r(e_i) = x_i$ and $s(e_i) = x$ for $1 \leq i \leq n$.  Let $G_x$ and all edge groups be trivial and let $G_{x_i} = G_i$.  Then $\pi_1(\GG,x) \cong G$.  This construction is a special case of the graphs of groups considered in, for example,~\cite{AMS}.   The $C^*$-algebra construction is a special case of results of \cite{Rob2,Sp,O}.

\end{enumerate}

\end{rmk}

In order to describe further examples, we will need the following lemma.

\begin{lem} \label{lem: generators of directed graph subalgebra}
Let $\GG$ be a nonsingular locally finite graph of countable groups, and let $E_\GG$ be the associated directed graph defined in the statement of Theorem \ref{thm: associated directed graph}. We have
\[
C^*(E_\GG)=\clsp\{s_\mu s_\nu^* : \mu, \nu \in \GG^*,\ s(\mu) = s(\nu) \}
\]
and
\[
C^*(\GG)= \clsp\{ s_\mu u_{s(\mu),g} s_\nu^* : \mu, \nu \in \GG^*,\ s(\mu) = s(\nu),\ g \in G_{s(\mu)} \}.
\]
\end{lem}

\begin{proof}  To prove the first assertion, first let $\mu = ge \in \GG^1$. Then
	\[
	s_\mu = u_{r(e),g} s_e = \sum_{\substack{r(f) = s(e),\, h \in \Sigma_f \\ hf \not= 1 \overline{e}}} u_{r(e),g} s_e u_{r(f),h} s_f s_f^* u_{r(f),h}^*
	= \sum_{\substack{\nu \in E_\GG^1 \\ s_E(\nu) = ge}} s_\nu s_{s_E(\nu)}^*\in C^*(E_\GG).
	\]
Now let $g e$, $h f \in \GG^1$ with $r(e) = r(f)$.  By (G4) we have that $s_{ge} s_{ge}^* s_{hf} s_{hf}^* = 0$ if $g e \not= hf$, and hence also $s_{ge}^* s_{hf} = 0$ if $ge \not= hf$.  Then by (G4) again we have $s_{ge}^* s_{ge} s_{hf} = s_{hf}$.  Iterating these identities proves that for $\mu$, $\nu \in \GG^*$, 
	\begin{equation}\label{eq:M}
s_\mu^* s_\nu  = 
 \begin{cases}
  s_{\nu'} &\text{ if } \nu = \mu \nu'\text{ for some }\nu'\in\GG^* \\
  s_{\mu'}^* &\text{ if } \mu = \nu \mu'\text{ for some }\mu'\in\GG^* \\
  s_e^* s_e &\text{ if } \mu = \nu \text{ and } \mu = \mu_1 s_e \\
  0 &\text{ otherwise.}
 \end{cases}
\end{equation}
It follows that $\sp\{s_\mu s_\nu^* : \mu, \nu \in \GG^*,\ s(\mu) = s(\nu) \}$ is a $*$-subalgebra of $C^*(E_\GG)$. Since the generators of $C^*(E_\GG)$ are in $\{s_\mu s_\nu^* : \mu, \nu \in \GG^*,\ s(\mu) = s(\nu) \}$, the first assertion holds. 

For the second assertion, we first use Equation~\eqref{eq:M} to see that for $\mu_1,\mu_2,\nu_1,\nu_2\in\GG^*$, $g_1\in G_{s(\mu_1)}$, $g_2\in G_{s(\mu_2)}$ we have
\[
s_{\mu_1} u_{s(\mu_1),g_1} s_{\nu_1}^* s_{\mu_2} u_{s(\mu_2),g_2} s_{\nu_2}^*
= \begin{cases}
  s_{\mu_1} u_{s(\mu_1),g_1} s_{\mu_2'} u_{s(\mu_2),g_2} s_{\nu_2}^* &\text{ if } \mu_2 = \nu_1 \mu_2'\text{ for some }\mu_2'\in\GG^* \\
  s_{\mu_1} u_{s(\mu_1),g_1} s_{\nu_1'} u_{s(\mu_2),g_2} s_{\nu_2}^* &\text{ if } \nu_1 = \mu_2 \nu_1'\text{ for some }\nu_1'\in\GG^* \\
  0 &\text{ otherwise.}
 \end{cases}
\]
By (G2) we have $u_{s(\mu_1),g_1} s_{\mu_2'} = s_{\mu_2''} u_{s(\mu_2'),g_1'}$ for some $\mu_2'' \in \GG^*$ and $g_1' \in G_{s(\mu_2')}$.  Then 
\begin{align*}
s_{\mu_1} u_{s(\mu_1),g_1} s_{\mu_2'} u_{s(\mu_2),g_2} s_{\nu_2}^* &= s_{\mu_1 \mu_2''} u_{s(\mu_2'),g_1' g_2} s_{\nu_2}^*\\
& \in \{ s_\mu u_{s(\mu),g} s_\nu^* : \mu, \nu \in \GG^*,\ s(\mu) = s(\nu),\ g \in G_{s(\mu)} \}.  
\end{align*}
A similar calculation shows that $s_{\mu_1} u_{s(\mu_1),g_1} s_{\nu_1'} u_{s(\mu_2),g_2} s_{\nu_2}^*$ is in this set. It follows that $\sp\{ s_\mu u_{s(\mu),g} s_\nu^* : \mu, \nu \in \GG^*,\ s(\mu) = s(\nu),\ g \in G_{s(\mu)} \}$ is a $*$-subalgebra of $C^*(\GG)$. Since the generators of $C^*(\GG)$ are in $\{ s_\mu u_{s(\mu),g} s_\nu^* : \mu, \nu \in \GG^*,\ s(\mu) = s(\nu),\ g \in G_{s(\mu)} \}$, the second assertion holds.
\end{proof}

In the remainder of this section we describe two examples of graphs of groups $\GG$ so that $C^*(\GG)$ is strictly larger than the associated directed graph algebra $C^*(E_\GG)$.  Example~\ref{subsubsec: an HNN example} is a loop of finite groups, and Example~\ref{subsubsec: BS graphs of groups} is a loop of groups whose fundamental group is a Baumslag--Solitar group.  See Examples~\ref{egs: graphs of groups}(2) above for the definition of a loop of groups; we continue notation from this example.

\begin{example} \label{subsubsec: an HNN example} Let $n \geq 2$ be an integer, and $\GG$ the loop of groups
	\[
	\begin{tikzpicture}[scale=2]
	
	\node (0_0) at (0,0) [circle] {};
	
	\filldraw [red] (0,0) circle (1pt);
	
	\draw[-stealth,thick] (.75,0) .. controls (.75,.5) and (.1,.5) .. (0_0) node[pos=0, inner sep=0.5pt, anchor=west] {$G_e:=\Z / n\Z$};
	\draw[thick] (.75,0) .. controls (.75,-.5) and (.1,-.5) .. (0_0);
	
	\draw (-1,0) node {$\Z / n\Z \oplus \Z / n\Z=:G_x$};
	
	\end{tikzpicture}
	\]	
 where $\alpha_e(1) = (1,0)$ and $\alpha_{\overline{e}}(1) = (0,1)$.  (Thus in $\pi(\GG)$ we have $(1,0) e = e (0,1)$.)  We let $\Sigma_e = \{ (0,i) : 0 \le i < n \}$ and $\Sigma_{\overline{e}} = \{ (i,0) : 0 \le i < n \}$.  

Define $u := u_{x,(1,0)}\in C^*(\GG)$ and $v := u_{x,(0,1)}\in C^*(\GG)$.  Then by Lemma \ref{lem: generators of directed graph subalgebra}, $C^*(E_\GG)$ is generated by the partial isometries $\{ u^i s_{\overline{e}} : 0 \le i < n \} \cup \{ v^i s_e : 0 \le i < n \}$.  We claim that $C^*(E_\GG) \not= C^*(\GG)$.  We prove this by showing that dist$(v, C^*(E_\GG)) = 1$ (and a similar proof works for $u$).  For this we show that $\| v - \sum c_i s_{\mu_i} s_{\nu_i}^* \| \ge 1$ for every finite sum $\sum c_is_{\mu_i}s_{\nu_i}^*\in C^*(E_\GG)$, $c_i\in\C$.  

Note that since $v$ is fixed by the gauge action of $C^*(\GG)$, it is fixed by the associated conditional expectation on $C^*(\GG)$ (see Remark~\ref{rem:CondExp}).  Since $s_\mu s_\nu^*$ is in the kernel of the expectation whenever $|\mu| \not= |\nu|$, we may assume that $|\mu_i| = |\nu_i|$ for all $i$.  Now fix a finite sum as above with $|\mu_i| = |\nu_i|$ for all $i$.  Let $m$ be an even integer with $m \ge |\mu_i|$, $|\nu_i|$ for all $i$. We will use the fact that if $\beta$ is a $\GG$-path with $|\beta| < m$, then $s_\beta = \sum \{s_\beta s_\gamma s_\gamma^* : \gamma \in \GG^{m - |\beta|},\ \beta \gamma \text{ is reduced} \}$.  To see this, note that (G4) implies that $s_\beta = \sum \{ s_\beta u_h s_f s_f^* u_h^* : hf \not= 1 \overline{e} \}$, where $e$ is the sourcemost edge of $\beta$.  This proves the fact if $m - |\beta| = 1$, and the general case follows by induction.  

Now if $|\mu| = |\nu| < m$ and $s(\mu) = s(\nu)$, we have
	\[
	s_\mu s_\nu^*
	= \sum_{\beta_1,\beta_2 \in s(\mu)\GG^{m-|\nu|}} s_\mu s_{\beta_1} s_{\beta_1}^* s_{\beta_2} s_{\beta_2}^* s_{\nu}^*
	= \sum_{\beta \in s(\mu)\GG^{m-|\nu|}} s_{\mu \beta} s_{\nu \beta}^*,
	\]
	since $s_{\beta_1}^* s_{\beta_2} = 0$ if $\beta_1 \not= \beta_2$.  (This follows since (G4) also implies that $(u_h s_f)^*(u_{h'} s_{f'}) = 0$ if $(h,f) \not= (h',f')$.)  Applying this to all $\mu_i$ and $\nu_i$, we may assume that $|\mu_i| = |\nu_i| = m$ for all $i$.  
	
	Let $\xi \in Z(( (1,0) \overline{e} (0,1) e)^{m/2} (0,0) e)$, $H_\xi = \ell^2(F(\GG)\xi)$, and $U$, $V$, $S_e$, $S_{\overline{e}}$ be the $\GG$-family in $B(H_\xi)$ defined in Remark \ref{rmk: concrete G family}.  Since $\| v - \sum c_i s_{\mu_i} s_{\nu_i}^* \| \ge \| V - \sum c_i S_{\mu_i} S_{\nu_i}^* \|$, it suffices to estimate the difference with this $\GG$-family.  Note that
	\begin{align*}
	(0,1) \cdot (1,0) \overline{e} (0,1) e
	&= (1,0) (0,1) \overline{e} (0,1) e\\
	&= (1,0) \overline{e} (1,0) (0,1) e\\
	&= (1,0) \overline{e} (0,1) (1,0) e\\
	&= (1,0) \overline{e} (0,1) e (0,1).
	\end{align*}
	Therefore $(0,1) \xi \in Z( ((1,0) \overline{e} (0,1) e)^{m/2} (0,1) e)$. We also have
	\[
	S_{\nu_i}^* \delta_\xi =
	\begin{cases}
	\delta_{\xi'} &\text{ if } \nu_i = ((1,0) \overline{e} (0,1) e)^{m/2} \\
	0 & \text{ if } \nu_i \not= ((1,0) \overline{e} (0,1) e)^{m/2},
	\end{cases}
	\]
	for some $\xi' \in Z( (0,0) e)$. Therefore
	\[
	S_{\mu_i} S_{\nu_i}^* \delta_\xi =
	\begin{cases}
	\delta_{\xi''} &\text{ if } \nu_i = ((1,0) \overline{e} (0,1) e)^{m/2} \text{ and the sourcemost edge of } \mu_i \text{ is } \overline{e} \\
	0 &\text{ otherwise,}
	\end{cases}
	\]
	for some $\xi'' \in Z(\beta 1 e)$, where $\beta \in \GG^n$ with sourcemost edge not equal to $\overline{e}$.  In all cases we have that $\langle V \delta_\xi, S_{\mu_i} S_{\nu_i}^* \delta_\xi \rangle = 0$ for all $i$, and hence
	\[
	\| V - \sum_i c_i S_{\mu_i} S_{\nu_i}^* \| \ge \| (V - \sum_i c_i S_{\mu_i} S_{\nu_i}^*) \delta_\xi \| \ge \| V \delta_\xi \| = 1.
	\]

\end{example}

\begin{example} \label{subsubsec: BS graphs of groups}
	
Let $m$ and $n$ be positive integers, and $\GG$ the loop of groups
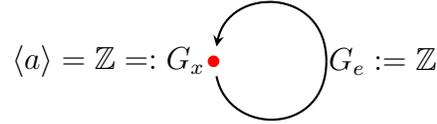
\[
	\begin{tikzpicture}[scale=2]
	
	\node (0_0) at (0,0) [circle] {};
	
	\filldraw [red] (0,0) circle (1pt);
	
	\draw[-stealth,thick] (.75,0) .. controls (.75,.5) and (.1,.5) .. (0_0) node[pos=0, inner sep=0.5pt, anchor=west] {$G_e:=\Z$};
	\draw[thick] (.75,0) .. controls (.75,-.5) and (.1,-.5) .. (0_0);
	
	\draw (-0.7,0) node {$\langle a\rangle =\Z=:G_x$};
	
	\end{tikzpicture}
	\] 
where $\alpha_e$ and $\alpha_{\overline{e}}$ send the generator of $G_e$ to $a^n$ and $a^m$, respectively.  The fundamental group of $\GG$ is the \emph{Baumslag--Solitar group} $\BS(m,n) = \langle a, e \mid ea^m e^{-1} = a^n \rangle$.  We choose $\Sigma_e = \{ a^i : 0 \le i < n \}$ and $\Sigma_{\overline{e}} = \{ a^i : 0 \le i < m \}$.  

In $C^*(\GG)$, write $u$ for $u_{x,a}$.  The relation (G2) then becomes $u^n s_e = s_e u^m$.  The directed graph $C^*$-algebra of Theorem \ref{thm: associated directed graph} is generated by $\{u^i s_e : 0 \le i < n \} \cup \{ u^i s_{\overline{e}} : 0 \le i < m \}$, by Lemma \ref{lem: generators of directed graph subalgebra}.  We show that $C^*(E_\GG) \not= C^*(\GG)$.

First, consider the case that $m \not= n$, say for definiteness $n>m$.   Let $\xi = a^{n-1} e (a^{n-m} e)^\infty$ be in $ \partial X_\GG$.  Then $a \xi = (1 e)^\infty$.  We use the regular representation $\pi_\xi$ of $C^*(\GG)$ on $H_\xi = \ell^2(F(\GG)\xi)$, described in Remark \ref{rmk: concrete G family}.  Thus $\pi_\xi(s_e) \delta_{\gamma\xi} = \delta_{1 e \gamma \xi}$ if $\gamma \xi \not \in Z(1 \overline{e})$, and is zero otherwise, and $\pi_\xi(u) \delta_{\gamma\xi} = \delta_{u \gamma \xi}$.  We obtain a vector functional $f \in C^*(\GG)^*$ by $f(b) = \langle \pi_\xi(b) \delta_\xi, \pi_\xi(u) \delta_\xi \rangle$.  Then $f(u) = 1$.  We claim that $C^*(E_\GG)$ is contained in the kernel of $f$, which will show that $u \not\in C^*(E_\GG)$.  It follows from Lemma \ref{lem: generators of directed graph subalgebra} that the set $\{ s_\mu s_\nu^* : \mu,\ \nu \in \GG^*,\ |\mu|, |\nu| \ge 1 \}$ spans a dense subset of $C^*(E_\GG)$.  Thus it suffices to show that $f(s_\mu s_\nu^*) = 0$ for all $\GG$-paths $\mu$ and $\nu$.
	
	We have that $f(s_\mu s_\nu^*) = \langle \pi_\xi(s_\nu^*) \delta_\xi, \pi_\xi(s_\mu^*) \pi_\xi(u) \delta_\xi \rangle$.  We may describe $s_\mu$ and $s_\nu$ as follows. Let $w_1 = s_e$ and $w_{-1} = s_{\overline{e}}$.  Then we may let $s_\mu = u^{i_1} w_{\beta_1} \dots u^{i_p} w_{\beta_p}$ and $s_\nu = u^{j_1} w_{\varepsilon_1} \dots u^{j_q} w_{\varepsilon_q}$, where $\beta_i$, $\varepsilon_j \in \{ +1, -1 \}$, and $\beta_\ell = 1$ forces $0 \le i_\ell < m$, $\beta_\ell = -1$ forces $0 \le i_\ell < n$ (and similarly for $\nu$).  Now 
	\begin{align*}
	\pi_\xi(s_\mu^* u) \delta_\xi = \pi_\xi(u^{-i_p}) \pi_\xi(w_{\beta_p}^*) \cdots \pi_\xi(u^{-i_1}) \pi_\xi(w_{\beta_1}^*) \delta_{(1 e)^\infty}
	&=
	\begin{cases}
	\pi_\xi(u) \delta_\xi & \text{if $\mu=e^p$}\\
	0 & \text{otherwise.}
	\end{cases}\\
	&=
	\begin{cases}
		\delta_{(1e)^\infty} & \text{if $\mu=e^p$}\\
		0 & \text{otherwise.}
	\end{cases}
	\end{align*}
	Similarly, 
	\[
	\pi_\xi(s_\nu^*) \delta_\xi =
	\begin{cases}
	\delta_{(a^{n-m} e)^\infty} & \text{if $\nu = a^{n-1} e (a^{n-m} e)^{q-1}$}\\
	0 & \text{otherwise.}
	\end{cases}
	\]
Since $n>m\implies \langle \delta_{(1e)^\infty}, \delta_{(a^{n-m} e)^\infty} \rangle=0$, we have $f(s_\mu s_\nu^*) = 0$ in all cases. The claim follows.
	
	If $m = n > 1$, then $u^m$ is a central element of $C^*(\GG)$.  It is easily seen that the directed graph $E_\GG$ is strongly connected, and hence $C^*(E_\GG)$ has trivial centre.  However it follows from considering the regular representations of Remark \ref{rmk: concrete G family} that $u^m$ is not a scalar multiple of the identity.  Thus again we have $u \not\in C^*(E_\GG)$.
	
\end{example}

\section{A $C^*$-algebraic Bass--Serre Theorem}\label{thm: main thm}

The action of the fundamental group $\pi_1(\GG,v)$ of a graph of groups $\GG$ on the boundary $v\partial X_\GG$ of the Bass--Serre tree $X_{\GG,v}$ induces a full crossed product $C^*$-algebra, in the sense of Section~\ref{subsubsec:CrossedProducts}. We denote this action by $\tau$. In this section we prove our main theorem, which says that the graph of groups $C^*$-algebra $C^*(\GG)$ is stably isomorphic to the crossed product $C(v\partial X_\GG)\rtimes_\tau\pi_1(\GG,v)$. Before stating the theorem, we remind the reader 
that much of the notation appearing below is defined in Notation~\ref{ntn: general notation}~and~\ref{ntn:UnderlineArrow}. We do introduce some more notation here: for every $x,y\in\Gamma^0$ we denote 
by 
$\theta_{x,y}$ the rank-one operator on $\ell^2(\Gamma^0)$ given by 
$\theta_{x,y}(f)=\langle f , \delta_y \rangle\delta_x$, where $\delta_x,\delta_y$ are 
point-mass functions. We denote the compact operators on $\ell^2(\Gamma^0)$ by 
$\KK(\ell^2(\Gamma^0))$, and we note that 
$\KK(\ell^2(\Gamma^0))=\clsp\{\theta_{x,y}:x,y\in\Gamma^0\}$. We denote the universal covariant representation of $(C(v\partial X_\GG),\pi_1(\GG,v),\tau)$ by $(i_A,i_\pi)$. For a $\GG$-path $\mu\in \GG^*$ we denote by $\chi_{Z(\mu)}$ the function that is $1$ on $Z(\mu)$ and is 0 on the complement $Z(\mu)^c$.

\begin{thm}\label{thm: main theorem}
Let $\GG=(\Gamma,G)$ be a locally finite nonsingular graph of countable groups. There is an 
isomorphism 
\[
\Phi:C^*(\GG)\to \KK(\ell^2(\Gamma^0))\otimes 
\big(C(v\partial X_\GG)\rtimes_\tau\pi_1(\GG,v)\big),
\]
satisfying $\Phi(u_{x,g})=\theta_{x,x}\otimes i_{\pi}(\varepsilon(g))$ for 
each $x\in 
\Gamma^0$, $g\in G_x$, and 
\[
\Phi(s_e)=
\begin{cases}
\theta_{r(e),s(e)}\otimes i_A(\chi_{Z([v,r(e)]1e)})i_{\pi}(\varepsilon(e))
& \text{if $\underleftarrow{e}\not=\overline{e}$}\\
\theta_{r(e),s(e)}\otimes i_A(\chi_{Z([v,r(e)])^c})
& \text{if $\underleftarrow{e}=\overline{e}$}
\end{cases}
\]
for each $e\in\Gamma^1$.
\end{thm}

The proof follows a standard approach---we use universal properties to construct mutually inverse homomorphisms between $C^*(\GG)$ and $\KK(\ell^2(\Gamma^0))\otimes 
\big(C(v\partial X_\GG) \rtimes_\tau \pi_1(\GG,v)\big)$---although it is long and heavy on calculations. We start by proving the existence 
of $\Phi$ by building a $\GG$-family in $\KK(\ell^2(\Gamma^0))\otimes 
\big(C(v\partial X_\GG) \rtimes_\tau \pi_1(\GG,v)\big)$, in Proposition~\ref{prop: G family in tensor}.  Next, in Proposition~\ref{prop: the Psi map}, we construct a homomorphism $\Psi: \KK(\ell^2(\Gamma^0))\otimes 
\big(C(v\partial X_\GG) \rtimes_\tau \pi_1(\GG,v)\big) \to C^*(\GG)$.  We show in Lemma~\ref{lem: Psi surjective} that $\Psi$ is surjective, and then complete the proof of Theorem~\ref{thm: main theorem} by proving that $\Phi \circ \Psi = \id$.

\begin{prop}\label{prop: G family in tensor}
Let $\GG=(\Gamma,G)$ be a locally finite nonsingular graph of countable groups. For 
each $x\in 
\Gamma^0$ and $g\in G_x$ define 
$$U_{x,g}:=\theta_{x,x}\otimes i_{\pi}(\varepsilon(g)),$$ and for each $e\in\Gamma^1$ 
define
\[
S_e:=
\begin{cases}
\theta_{r(e),s(e)}\otimes i_A(\chi_{Z([v,r(e)]1e)})i_{\pi}(\varepsilon(e))
& \text{if $\underleftarrow{e}\not=\overline{e}$}\\
\theta_{r(e),s(e)}\otimes i_A(\chi_{Z([v,r(e)])^c})
& \text{if $\underleftarrow{e}=\overline{e}$}.
\end{cases}
\]
Then the collection  $\{U_{x},S_e : x \in \G^0, e \in \G^1\}$ is a $\GG$-family in   
$\KK(\ell^2(\Gamma^0))\otimes 
\big(C(v\partial X_\GG ) \rtimes_\tau \pi_1(\GG,v)\big)$.
\end{prop}

We start the proof of this Proposition by explicitly calculating the initial and final projections of each $S_e$.  (Since $S_e$ is the tensor product of two partial isometries, it is clear that it is also a partial isometry.)

\begin{lem} \label{lem: initial and final projections}

Let $e \in \Gamma^1$, and let $S_e$ be defined as in Proposition \ref{prop: G family in tensor}.  Then
\begin{align*}
S_e^* S_e
&= \begin{cases}
 \theta_{s(e),s(e)} \otimes i_A( \chi_{ Z([v,s(e)] 1 \overline{e})^c} ) &\text{ if } e \not\in T^1 \text{ or } \underleftarrow{e} = \overline{e} \\
 \theta_{s(e),s(e)} \otimes i_A( \chi_{Z([v,s(e)])} &\text{ if } e \in T^1 \text{ and } \underleftarrow{e} \not= \overline{e},
 \end{cases} \\
S_e S_e^*
&= \begin{cases}
 \theta_{r(e),r(e)} \otimes i_A( \chi_{ Z([v,r(e)] 1 e)} ) &\text{ if } e \not\in T^1, \text{ or } e \in T^1 \text{ and } \underleftarrow{e} \not= \overline{e} \\
 \theta_{r(e),r(e)} \otimes i_A( \chi_{Z([v,r(e)])^c} ) &\text{ if }  \underleftarrow{e} = \overline{e}.
 \end{cases}
\end{align*}

\end{lem}

\begin{proof}
First suppose that $e \not\in T^1$.  We have $S_e = \theta_{r(e),s(e)}\otimes i_A(\chi_{Z([v,r(e)]1e)})i_{\pi}(\varepsilon(e))$.  Thus
\begin{align*}
S_e^* S_e
&= \theta_{s(e),s(e)} \otimes i_\pi(\varepsilon(e))^* i_A(\chi_{Z([v,r(e)] 1 e)} ) i_\pi(\varepsilon(e)) \\
&= \theta_{s(e),s(e)} \otimes i_A( \tau_{\varepsilon(\overline{e})} (\chi_{Z([v,r(e)] 1 e)})) \\
&= \theta_{s(e),s(e)} \otimes i_A(\chi_{\varepsilon(\overline{e}) Z([v,r(e)] 1 e)}). 
\end{align*}
Moreover, Lemma \ref{lem: actiondetails}\eqref{lem: actiondetails.4} implies
\[
\varepsilon(\overline{e})Z([v,r(e)] 1 e)
= Z([v,s(e)] 1 \overline{e})^c,
\]
so that $S_e^* S_e$ is as required.  
It is easy to see that in this case, 
\[
S_e S_e^*
= \theta_{r(e),r(e)} \otimes i_A( \chi_{Z([v,r(e)] 1 e)} ).
\]

Now suppose that $e \in T^1$ and $\underleftarrow{e} \not= \overline{e}$.  Then $[v,s(e)] = [v,r(e)] 1 e$, and $\varepsilon(e)$ is trivial.  Then we have
\[
S_e^* S_e
=\theta_{s(e),s(e)} \otimes i_A( \chi_{Z([v,r(e)] 1 e)} )=\theta_{s(e),s(e)} \otimes i_A( \chi_{Z([v,s(e)])} )
\]
and
\[
S_e S_e^*
= \theta_{r(e),r(e)} \otimes i_A( \chi_{Z([v,r(e)] 1 e)} ).
\]

Finally, suppose that $\underleftarrow{e} = \overline{e}$.  Then $S_e = \theta_{r(e),s(e)}\otimes i_A(\chi_{Z([v,r(e)])^c})$. The verification of the lemma in this case is similar to the previous case.
\end{proof}

\begin{proof}[Proof of Proposition~\ref{prop: G family in tensor}] 
For $x,y\in\Gamma^0$ with $x\not=y$ we have 
$U_{x,1}U_{y,1}=0$ since $\theta_{x,x}\theta_{y,y}=0$, so (G1) holds.

Next, note that for $\gamma \in \pi_1(\GG,v)$ and $\mu \in v\GG^*$, since $\tau_{\varepsilon(\gamma)}(\chi_{Z(\mu)}) = \chi_{\varepsilon(\gamma) Z(\mu)}$, covariance gives
\[
i_\pi(\varepsilon(\gamma)) i_A(\chi_{Z(\mu)})
= i_A(\tau_{\varepsilon(\gamma)}(\chi_{Z(\mu)})) i_\pi(\varepsilon(\gamma))
= i_A(\chi_{\varepsilon(\gamma) Z(\mu)}) i_\pi(\varepsilon(\gamma)).
\]
Now, for (G2) we first fix $e\in \Gamma^1$ with $\underleftarrow{e}\not=\overline{e}$ and 
$g\in G_e$.  We have
\begin{align*}
U_{r(e),\alpha_e(g)} S_e
&= \theta_{r(e),s(e)} \otimes i_\pi(\varepsilon(\alpha_e(g))) i_A(\chi_{Z([v,r(e)]1 e)}) i_\pi(\varepsilon(e)) \\
&= \theta_{r(e),s(e)} \otimes i_A(\chi_{\varepsilon(\alpha_e(g)) Z([v,r(e)] 1 e)}) i_\pi(\varepsilon(\alpha_e(g) e)) \\
&= \theta_{r(e),s(e)} \otimes i_A(\chi_{Z([v,r(e)] \alpha_e(g) e)}) i_\pi(\varepsilon(e \alpha_{\overline{e}}(g)))\quad\text{(by Lemma \ref{lem: actiondetails}\eqref{lem: actiondetails.1} and \eqref{eq: formerly appx (B)})} \\
&= \theta_{r(e),s(e)} \otimes i_A(\chi_{Z([v,r(e)] 1 e)}) i_\pi(\varepsilon(e)) i_\pi(\varepsilon(\alpha_{\overline{e}}(g))) \\
&= S_e U_{s(e),\alpha_{\overline{e}}(g)}.
\end{align*}
If $\underleftarrow{e} = \overline{e}$, the same calculation, but without the final factor $i_\pi(\varepsilon(e))$, gives the result.

Lemma \ref{lem: initial and final projections} implies (G3) directly. We need a number of cases to check (G4). First consider the case that $e\in\Gamma^1$ with $s(e)=v$.  In this case, it is not possible that $e \in T^1$ and $\underleftarrow{e} \not= \overline{e}$.  Thus by Lemma \ref{lem: initial and final projections} we have that
\[
S_e^* S_e
= \theta_{s(e),s(e)} \otimes i_A( \chi_{Z(1 \overline{e})^c}) \quad \mbox{and} \quad 
S_{\overline{e}} S_{\overline{e}}^*
= \theta_{s(e),s(e)} \otimes i_A( \chi_{Z(1 \overline{e})} ).
\]
We can use the covariance of $(i_A,i_\pi)$ and Lemma~\ref{lem: actiondetails} to see that for each $h\in\Sigma_{\overline{e}}$ we have  
 \[
U_{s(e),h} S_{\overline{e}} S_{\overline{e}}^* U_{s(e),h}^* = \theta_{s(e),s(e)} \otimes i_A( \chi_{Z(h \overline{e})} ). \]
For edges $f$ so that $r(f) = s(e)$ and $f \not= \overline{e}$, we have
\[
U_{s(e),h} S_f S_f^* U_{s(e),h}^* = \theta_{s(e),s(e)} \otimes i_A( \chi_{Z(h f)}),
\]
for each $h\in\Sigma_f$. Since
\[
Z(1 \overline{e})^c
= \bigcup_{\substack{r(f) = s(e),\, h \in \Sigma_f \\ hf \not= 1 \overline{e}}} Z(h f),
\]
it follows that (G4) holds in this case.

For edges $e$ with $s(e)\not=v$ we have three cases, two of which may be treated together. If either $e\not\in 
T^1$ or $e\in T^1$ with $\underleftarrow{e}=\overline{e}$, then we have
\[
S_e^*S_e=\theta_{s(e),s(e)}\otimes 
i_A(\chi_{Z([v,s(e)]1\overline{e})^c}).
\]

\noindent
If we denote by $f_0$ the unique edge in $T^1$ such that $r(f_0)=s(e)$ and 
$\underleftarrow{f_0}=\overline{f_0}$, then
\begin{align*}
&\{h f : r(f) = s(e),\, h \in \Sigma_f,\ h f \not= 1 \overline{e}\}\\
&\hspace{4cm}= \{h f_0 : h \in \Sigma_{f_0} \}
\cup \{h f : r(f) = s(e),\, f \not= f_0,\, h \in \Sigma_f,\, h f \not= 1 \overline{e}\}.
\end{align*}
Since $\underleftarrow{f_0} = \overline{f_0}$ and $f_0 \in T^1$, Lemma \ref{lem: initial and final projections} gives
\[
S_{f_0} S_{f_0}^* = \theta_{s(e),s(e)} \otimes i_A(\chi_{Z([v,s(e)])^c}). 
\]
Lemma \ref{lem: actiondetails}\eqref{lem: actiondetails.2} now gives
\[
U_{s(e),h} S_{f_0} S_{f_0}^* U_{s(e),h}^*=  \theta_{s(e),s(e)} \otimes i_A(\chi_{\varepsilon(h) Z([v,s(e)])^c}) = \theta_{s(e),s(e)} \otimes i_A(\chi_{Z([v,s(e)] h f_0)}),
\]
for each $h\in\Sigma_{f_0}$. For $f \not= f_0$ with $r(f) = s(e)$, we have
\[
S_f S_f^*
= \theta_{s(e),s(e)} \otimes i_A(\chi_{Z([v,s(e)] 1 f)}), 
\]
and so
\[
U_{s(e),h} S_f S_f^* U_{s(e),h}^*
= \theta_{s(e),s(e)} \otimes i_A(\chi_{Z([v,s(e)] h f)}),
\]
for each $h\in\Sigma_f$. Since
\[
 Z([v,s(e)] 1 \overline{e})^c
= Z([v,s(e)])^c \cup \bigcup_{\substack{ r(f) = s(e),\, h \in \Sigma_f \\ hf \not= 1 \overline{e},\, 1 f_0}} Z([v,s(e)] h f),
\]
it follows that (G4) holds in these cases.

The final case is when $e\in T^1$ with $\underleftarrow{e}\not=\overline{e}$. Then 
$[v,r(e)]1e=[v,s(e)]$ and 
\[
S_e^*S_e=\theta_{s(e),s(e)}\otimes i_A(\chi_{Z([v,s(e)])}).
\]
The edge $\overline{e}$ is the unique edge with range $s(e)$ and 
$\underleftarrow{(\overline{e})}=e$. We have
\begin{align*}
&\{h f : r(f) = s(e),\, h \in \Sigma_f,\, h f \not= 1 \overline{e}\}\\
&\hspace{3cm}= \{h \overline{e} : h \in \Sigma_{\overline{e}},\, h \not= 1 \}
\cup \{h f : r(f) = s(e),\  f \not= \overline{e},\, h \in \Sigma_f,\, h f \not= 1 \overline{e}\}.
\end{align*}
From Lemma \ref{lem: initial and final projections} we get
\[
S_{\overline{e}} S_{\overline{e}}^* = \theta_{s(e),s(e)} \otimes i_A( \chi_{Z([v,s(e)])^c}), \]
and hence Lemma \ref{lem: actiondetails}\eqref{lem: actiondetails.2} gives
\[
U_{s(e),h} S_{\overline{e}} S_{\overline{e}}^* U_{s(e),h}^*= \theta_{s(e),s(e)} \otimes i_A(\chi_{\varepsilon(h) Z([v,s(e)])^c})=  \theta_{s(e),s(e)} \otimes i_A(\chi_{\varepsilon(h) Z([v,s(e)] h \overline{e})}),
\]
for each $1\not=h\in\Sigma_{\overline{e}}$. For $r(f) = s(e)$, $f \not= \overline{e}$, and $h \in \Sigma_f$, we have 
\[
U_{s(e),h} S_f S_f^* U_{s(e),h}^*= \theta_{s(e),s(e)} \otimes i_A(\chi_{\varepsilon(h) Z([v,r(f)] 1 f)})= \theta_{s(e),s(e)} \otimes i_A(\chi_{Z([v,r(f)] h f)}).
\]
Since
\[
 Z([v,s(e)]
= \bigcup_{h \in \Sigma_{\overline{e}},\, h \not= 1} Z([v,s(e)] h \overline{e})
\cup \bigcup_{\substack{ r(f) = s(e),\, h \in \Sigma_f \\ f \not= \overline{e}}} Z([v,s(e)] h f)
\]
it follows that  (G4) holds in this case.  This completes the proof of Proposition~\ref{prop: G family in tensor}.
\end{proof}

To prove that $\Phi$ is an isomorphism, we construct in Proposition~\ref{prop: the Psi map} a homomorphism $\Psi$ from 
$\KK(\ell^2(\Gamma^0))\otimes 
\big(C ( v\partial X_\GG )\rtimes_\tau\pi_1 (\GG,v)\big)$ to $C^*(\GG)$, which we will prove is the inverse of $\Phi$. Before doing this we
give a technical lemma.   

\begin{lem}\label{lem: treeboundary}

Let $X$ be a locally finite nonsingular tree, and let $v \in X^0$.  Then $C(v \partial X)$ is the universal $C^*$-algebra generated by a family $\{ p_\mu : \mu \in v X^* \}$ subject to the following relations:

\begin{enumerate}

\item \label{lem: treeboundary.1} the $p_\mu$ are commuting projections; and

\item \label{lem: treeboundary.2} for all $\mu \in v X^*$ we have
$\displaystyle
p_\mu = \displaystyle \sum_{\substack{r(f) = s(\mu) \\ \mu f \text{ reduced}}} p_{\mu f}.
$

\end{enumerate}

\end{lem}

\begin{proof} 
	Let $A$ denote the universal $C^*$-algebra in the statement.  We first observe that the characteristic functions $\{ \chi_{Z(\mu)} : \mu \in v X^* \}$ satisfy the relations in the statement. Thus there is a $*$-homomorphism $A \to C(v \partial X)$ such that $p_\mu \mapsto \chi_{Z(\mu)}$.  Let $V$ be the complex vector space with basis $v X^*$.  Let $\{ w_\mu : \mu \in vX^* \}$ be the basis elements.  Define a linear map $L : V \to C(v \partial X)$ by $L(w_\mu) = \chi_{Z(\mu)}$.   Then the range of $L$ is a dense $*$-subalgebra in $C(v \partial X)$.  Let
	\[
	E = \bigl\{ w_\mu - \sum \{ w_{\mu f} : f \in s(\mu)X^1,\ \mu f \text{ is reduced} \} : \mu \in v X^* \bigr\}
	\]
	and let $M = \text{span}\, E$.  Then $M$ is contained in the kernel of $L$. We claim that in fact, $M = \text{ker\,} L$.  For the proof, let $z = \sum_{\mu \in vX^*} c_\mu w_\mu \in \ker L$, where only finitely many $c_\mu$ are nonzero.  Let $n = \max \{ |\mu| : c_\mu \not= 0 \}$.  Let $\mu \in v X^*$ with $|\mu| < n$. Then we have
	\[
		w_\mu
		= \displaystyle
		\sum_{\substack{f \in s(\mu) X^1 \\ \mu f \text{ reduced}}} w_{\mu f} + (w_\mu - \sum_{\substack{f \in  s(\mu) X^1 \\ \mu f \text{ reduced}}} w_{\mu f}) 
		\in \Big(\sum_{\substack{f \in s(\mu) X^1 \\ \mu f \text{ reduced}}} w_{\mu f}\Big) + M.
	\]
	Applying this inductively, we find that for a path $\mu$ with $|\mu| < n$, we have that
	\[
	w_\mu \in \Big(\sum_{\substack {\beta \in \mu X^* \\ |\beta| = n}} w_\beta
	\Big) + M.
	\]
	Then we have
	\[
		z
		= \sum_\mu c_\mu w_\mu \\
		\in  \Big(\sum_\mu c_\mu \sum_{\substack {\beta \in \mu X^* \\ |\beta| = n}} w_\beta\Big) + M \\
		= \Big(\sum_{\beta \in v X^n} \Big( \sum_{\{ \mu: \beta \in \mu X^*\}} c_\mu \Big) w_\beta\Big) + M.
	\]
	In particular, since $L(z) = 0$, we have that
	\[
	\sum_{\beta \in v X^n} \Big( \sum_{\{ \mu: \beta \in \mu X^*\}} c_\mu \Big) \chi_{Z(\beta)} = 0.
	\]
	Since the sets $Z(\beta)$ for $\beta \in vX^n$ are pairwise disjoint, it follows that for each $\beta \in vX^n$ we have $\sum_{\{ \mu: \beta \in \mu X^*\}} c_\mu = 0$.  Now the previous calculation shows that $z \in M$, finishing the proof that $M = \text{ker}\, L$.
	
	It follows now that $L$ descends to a (linear) isomorphism
	\[
	L_0 : V / M \to \text{span} \{\chi_{Z(\mu)} : \mu \in v X^* \}.
	\]
	By the universal property of $V$, there is a linear map $K : V / M \to A$ defined by $K(w_\mu + M) = p_\mu$.  Then $K \circ L_0^{-1} : \text{span} \{\chi_{Z(\mu)} : \mu \in v X^* \} \to A$ is a linear map.  It is a $*$-homomorphism since these characteristic functions have the same multiplication relations as the corresponding generators of $A$.  Since $\text{span} \{\chi_{Z(\mu)} : \mu \in v X^* \}$ is an increasing union of finite dimensional $C^*$-algebras, this extends to a $*$-homomorphism of $C(v \partial X)$ onto $A$, inverse to the canonical map of $A$ onto $C(v \partial X)$.
\end{proof}

For the remainder of the proof of Theorem~\ref{thm: main theorem}, we use the following notation.

\begin{ntn}\label{notation: t's etc}
	Let $\{U_x, S_e : x \in \Gamma^0, e \in\Gamma^1\}$ be a $\GG$-family. For each $e \in \Gamma^1$ we define
	\[
	T_e:=S_e + S_{\overline{e}}^*. 
	\]
	For each $\GG$-path $\mu=g_1e_1\dots g_n e_n$ we define
	\[
	T_{\mu}:=U_{r(e_1),g_1}T_{e_1}\dots U_{r(e_n),g_n}T_{e_n},
	\]
	and we let $T_x = U_{x,1}$ for each $x\in\Gamma^0$. Each $T_\mu$ is also a partial isometry, and we have $T_\mu^*T_\mu=U_{s(\mu),1}$ and $T_\mu T_\mu^*=U_{r(\mu),1}$.
	For the universal $\GG$-family $\{u_x,s_e:x\in\Gamma^0,e\in\Gamma^1\}$, each $t_e$ and $t_\mu$ is defined analogously. 
\end{ntn}

\begin{prop}\label{prop: the Psi map}
Let $\GG=(\Gamma,G)$ be a locally finite nonsingular graph of countable groups. Then there is a homomorphism
\[
\Psi: \KK(\ell^2 (\Gamma^0) )\otimes 
\big( C(v\partial X_\GG ) \rtimes_\tau \pi_1(\GG,v)\big)\to C^*(\GG)
\]

\noindent
satisfying
\begin{align*}
\Psi(\theta_{x,y} \otimes 1) &= t_{[x,y]},\\
\Psi(\theta_{v,v} \otimes i_A(\chi_{Z(\mu)})) &= s_\mu s_\mu^*,\\
\Psi(\theta_{v,v} \otimes i_{\pi} ( \varepsilon(g,x)) ) &= t_{[v,x]} u_{x,g} t_{[x,v]} \text{ and}\\
\Psi(\theta_{v,v} \otimes i_{\pi} ( \varepsilon(e)) ) &= t_{[v,r(e)]} t_e t_{[s(e),v]},
\end{align*}

\noindent
for each $x,y\in\Gamma^0$, $\mu \in 
v\GG^*$, $g\in G_x$ and $e\in\Gamma^1$.
\end{prop}

We will use the following lemma in the proof of this result.

\begin{lem}\label{lem: actionmoves2}
Let $e,f \in \G^1$ with $s(f) = r(e)$, $g \in \Sigma_{e}$, $h \in \Sigma_{f}$ and $\mu\in r(e)\GG^*$. 
\begin{enumerate}
\item 
 \label{lem: actionmoves2.1} If  $g e \not= 1 \overline{f}$, then $ u_{r(f),h} t_f u_{r(e),g} s_e = u_{r(f),h} s_f u_{r(e),g} s_e$. 
 \item \label{lem: actionmoves2.3} $t_f s_{\overline{f}} = u_{r(f),1} - s_f s_f^*$.
\item \label{lem: actionmoves2.2} If $\mu = g e \mu'$ for some $\mu'$ and $g e \not= 1 f$, then $t_f s_{\overline{f}} s_\mu = s_\mu$. 

\end{enumerate}
\end{lem}

\begin{proof}
For \eqref{lem: actionmoves2.1}, we note that since $ge \not= 1 \overline{f}$, it follows from (G4) that $(u_{r(e),g} s_e)(u_{r(e),g} s_e)^* \le s_f^* s_f$, and hence is orthogonal to $s_{\overline{f}} s_{\overline{f}}^*$.  Part \eqref{lem: actionmoves2.1} now follows from the fact that $t_f = s_f + s_{\overline{f}}^*$.  For \eqref{lem: actionmoves2.3}, we use that $s_fs_{\overline{f}}=0$ to get 
\[
t_f s_{\overline{f}} = (s_f + s_{\overline{f}}^*) s_{\overline{f}} = s_{\overline{f}}^* s_{\overline{f}} = u_{r(f),1} - s_f s_f^*.
\]
For \eqref{lem: actionmoves2.2}, first note that $ge\not=1f$ implies that $s_fs_f^*$ and $s_\mu s_\mu^*$ are orthogonal. We can then use (2) to get $t_f s_{\overline{f}} s_\mu = (u_{r(f),1} - s_f s_f^*)s_\mu = s_\mu$.
\end{proof}

\begin{proof}[Proof of Proposition~\ref{prop: the Psi map}]
Straightforward calculations show that $\{t_{[x,y]} : x,y \in \Gamma^0 \}$ is a family of matrix units in $C^*(\GG)$; we check the case when $[x,y] = 1 e_1 \dots 1 e_m \dots 1 e_{m+k}$ and $[y,z]=1 \overline{e}_{m+k} \dots 1 \overline{e}_{m+1} 1 f_1 \dots 1 f_n$, where $e_m \not=\overline{f_1}$. Then observe that $[x,z]=1 e_1 \dots 1 e_m 1 f_1\dots 1 f_n$, and we have
\begin{align*}
t_{[x,y]} t_{[y,z]} &= t_{1 e_1 \dots 1 e_m} t_{1 e_{m+1} \dots 1 e_{m+k}} t_{1 \overline{e}_{m+k} \dots 1 \overline{e}_m} t_{1 f_1 \dots 1 f_n}\\
&= t_{1 e_1 \dots 1 e_m} t_{1 e_{m+1} \dots 1 e_{m+k}} t_{1 e_{m+1} \dots 1 e_{m+k}}^* t_{1 f_1\dots 1 f_n}\\
&= t_{1 e_1 \dots 1 e_m} u_{s(e_m),1} t_{1 f_1 \dots 1 f_n}\\
&= t_{1 e_1 \dots 1 e_m 1 f_1 \dots 1 f_n}\\
&= t_{[x,z]}.
\end{align*}

\noindent
We denote by $j_{\KK} : \KK(\ell^2(\Gamma^0)) \to C^*(\GG)$ the homomorphism satisfying $j_{\KK}(\theta_{x,y}) = t_{[x,y]}$. 

We now want to build a covariant representation of $( C ( v\partial X_\GG ), \pi_1(\GG,v), \tau)$ in $C^*(\GG)$, which gives us a homomorphism $j_A\times j_\pi : C ( v\partial X_\GG ) \rtimes_\tau \pi_1(\GG,v) \to C^*(\GG)$ whose range commutes with the range of $j_{\KK}$. We do this by first building a covariant representation in the corner $u_{v,1} C^*(\GG) u_{v,1}$.  

We claim that $\{ s_\mu s_\mu^* : \mu \in v\GG^* \}$ is a collection of projections satisfying the hypotheses of Lemma~\ref{lem: treeboundary}.  To see this, first note that it is observed in Notation~\ref{ntn: S mu} that the $s_\mu$ are partial isometries, and hence that the $s_\mu s_\mu^*$ are projections.  We next check Lemma \ref{lem: treeboundary}\eqref{lem: treeboundary.2}.  For each $\mu=g_1 e_1 \dots g_n e_n \in v\GG^*$ we use (G4) to get
\begin{align*}
s_\mu s_\mu^*
&= u_{v,g_1} s_{e_1} \dots u_{v,g_n} s_{e_{n}} s_{e_{n}}^* u_{v,g_n}^* \dots s_{e_1}^* u_{v,g_1}^* \\
&= u_{v,g_1} s_{e_1} \dots u_{v,g_n} s_{e_{n}} (s_{e_n}^* s_{e_n}) s_{e_{n}}^* u_{v,g_n}^* \dots s_{e_1}^* u_{v,g_1}^* \\
&= u_{v,g_1} s_{e_1} \dots u_{v,g_n} s_{e_{n}} \left(\sum_{\substack{ r(f)=s(e_n),\, h\in \Sigma_f \\ hf \not= 1 \overline{e_n}}} u_{s(e_n),h} s_f s_f^* u_{s(e_n),h}^*\right) s_{e_{n}}^* u_{v,g_n}^* \dots s_{e_1}^* u_{v,g_1}^* \\
&= \sum_{\substack{ r(f)=s(e_n),\, h\in \Sigma_f \\ hf \not= 
1 \overline{e_n}}} s_{\mu h f} s_{\mu h f}^*.
\end{align*}
Note that it follows from this, and induction, that if $\mu$ and $\nu$ are \textit{comparable}, say if $\nu = \mu \nu'$, then $s_\nu s_\nu^* \le s_\mu s_\mu^*$, and hence $s_\nu s_\nu^*$ and $s_\mu s_\mu^*$ commute.  On the other hand, if $\mu$ and $\nu$ are not comparable, then we may write $\mu = \eta g e \mu'$ and $\nu = \eta h f \nu'$ with $g e \not= h f$.  Then by (G4),
\[
s_\nu^* s_\mu
= s_{\nu'}^* (u_{r(f),h} s_f)^* s_\eta^* s_\eta u_{r(e),g} s_e s_{\mu'}
= s_{\nu'}^* (u_{r(f),h} s_f)^* (u_{r(e),g} s_e) s_{\mu'}
= 0,
\]
and hence $s_\nu s_\nu^* s_\mu s_\mu^* = 0$, and again $s_\mu s_\mu^*$ and $s_\nu s_\nu^*$ commute.  Thus the claim holds, and we can apply Lemma~\ref{lem: treeboundary} to get a homomorphism $j_{A,v} : C ( v\partial X_\GG ) \to u_{v,1} C^*(\GG) u_{v,1}$ satisfying $j_{A,v} (\chi_{Z(\mu)}) = s_\mu s_\mu^*$ for all $\mu\in v\GG^*$.

For each $x\in\Gamma^0$, $g\in G_x$ and $e\in\Gamma^1$ we now define
\[
\tilde{u}_{x,g}:=t_{[v,x]} u_{x,g} t_{[x,v]} \quad \mbox{and} \quad\tilde{u}_e :=t_{[v,r(e)]} t_e t_{[s(e),v]}.
\]

\noindent
Note that $\tilde{u}_{x,g}$ and $\tilde{u}_e$ are elements of the corner $u_{v,1} C^*(\GG) u_{v,1}$. Straightforward calculations show that:
\begin{enumerate}
\item[(1)] $\tilde{u}_{x,g}\tilde{u}_{x,h}=\tilde{u}_{x,gh}$, for all $x \in \G^0$ and $g, h \in G_x$;
\item[(2)] $\tilde{u}_{x,g}^*=\tilde{u}_{x,g^{-1}}$, for all $x \in G^0$ and $g \in G_x$;
\item[(3)] $\tilde{u}_e^*\tilde{u}_e=\tilde{u}_e\tilde{u}_e^*=u_{v,1}$, for all $e \in \G^1$; 
\item[(4)] $\tilde{u}_{\overline{e}}=\tilde{u}_e^*$, for all $e \in \G^1$;
\item[(5)] $\tilde{u}_e=u_{v,1}$, for all $e \in T^1$; and 
\item[(6)] $\tilde{u}_{r(e),\alpha_e(g)}\tilde{u}_e=\tilde{u}_e\tilde{u}_{s(e),\alpha_{\overline{e}}(g)}$, for all $e \in \G^1$ and $g \in G_e$.
\end{enumerate}
For example, to see that (6) holds, first note that (G2) gives
\begin{align*}
u_{r(e),\alpha_e(g)} t_e = u_{r(e),\alpha_e(g)} (s_e+s_{\overline{e}}^*)&=u_{r(e),\alpha_e(g)} s_e + (s_{\overline{e}} u_{s(\overline{e}),\alpha_e(g^{-1})})^* \\
&=s_e u_{r(e),\alpha_{\overline{e}}(g)} + (u_{r(\overline{e}),\alpha_{\overline{e}}(g^{-1})} s_{\overline{e}})^*\\
&=t_e u_{s(e),\alpha_{\overline{e}}(g)}.
\end{align*}
Then
\begin{align*}
\tilde{u}_{r(e),\alpha_e(g)} \tilde{u}_e & =t_{[v,r(e)]} u_{r(e),\alpha_e(g)} t_{[r(e),v]} t_{[v,r(e)]} t_e t_{[s(e),v]} \\
&= t_{[v,r(e)]} u_{r(e),\alpha_e(g)} t_e t_{[s(e),v]}\\
&= t_{[v,r(e)]} t_e u_{s(e),\alpha_{\overline{e}}(g)}t_{[s(e),v]}\\
&= t_{[v,r(e)]} t_e t_{[s(e),v]} t_{[v,s(e)]} u_{s(e),\alpha_{\overline{e}}(g)} t_{[s(e),v]}\\
&= \tilde{u}_e \tilde{u}_{s(e),\alpha_{\overline{e}}(g)},
\end{align*}
which is (6). It follows from (1)--(6) that the $\tilde{u}_{x,g}$ and $\tilde{u}_e$ define a unitary representation of $\pi_1(\GG,v)$ in $u_{v,1} C^*(\GG) u_{v,1}$. We denote this unitary representation by $j_{\pi,v}$.  

We now claim that $(j_{A,v}, j_{\pi,v})$ is a covariant representation of $( C( v\partial X_\GG), \pi_1(\GG,v),\tau)$ in the corner $u_{v,1} C^*(\GG) u_{v,1}$. It suffices to prove that
\begin{equation} \label{eq: covariance g}
j_{\pi,v} (\varepsilon(x,g)) j_{A,v} (\chi_{Z(\mu)}) j_{\pi,v} (\varepsilon(x,g))^* = j_{A,v}(\tau_{\varepsilon(x,g)} (\chi_{Z(\mu)}))
\end{equation}
for all $x\in\Gamma^0$, $g\in G_x$ and $\mu \in v\GG^*$, and
\begin{equation} \label{eq: covariance e}
j_{\pi,v}(\varepsilon(e)) j_{A,v}(\chi_{Z(\mu)}) j_{\pi,v}(\varepsilon(e))^* = j_{A,v} (\tau_{\varepsilon(e)}(\chi_{Z(\mu)}))
\end{equation}
for all $e\in\Gamma^1$ and $\mu\in v\GG^*$.  For this, we use Lemmas \ref{lem: actionmoves1} and \ref{lem: actionmoves2}.  Notice that the first three parts of Lemma \ref{lem: actionmoves1}, and the three parts of Lemma \ref{lem: actionmoves2}, correspond to each other, one treating the action of $F(\GG)$ on $\partial W_\GG$ and the other treating calculations in $C^*(\GG)$ involving the elements $t_e$, $u_{x,g}$ and $s_e$.  We note that
\[
j_{\pi,v}(\varepsilon(x,g) j_{A,v}(\chi_{Z(\mu)}) j_{\pi,v}(\varepsilon(x,g))^*
= t_{[v,x]} u_{x,g} t_{[x,v]} s_\mu (t_{[v,x]} u_{x,g} t_{[x,v]} s_\mu)^* 
\]
and
\[
j_{A,v}(\tau_{\varepsilon(x,g)}(\chi_{Z(\mu)}))
= j_{A,v}(\chi_{[v,x] g [x,v] Z(\mu)}).
\]
Let $[v,x] g [x,v] = h_1 f_1 \dots h_n f_n$ and $\mu = g_1 e_1 \dots g_m e_m$.  Then we must compare
\begin{align*}
u_{r(f_1),h_1} t_{f_1} \dots u_{r(f_n),h_n} t_{f_n} \; & u_{r(e_1),g_1} s_{e_1} \dots u_{r(e_m),g_m} s_{e_m} \\
&\cdot (u_{r(f_1),h_1} t_{f_1} \dots u_{r(f_n),h_n} t_{f_n} \; u_{r(e_1),g_1} s_{e_1} \dots u_{r(e_m),g_m} s_{e_m})^*
\end{align*}
with 
\[
j_{A,v} (\chi_{h_1 f_1 \dots h_n f_n Z(g_1 e_1 \dots g_m e_m)} ).
\]
The result of applying Lemma \ref{lem: actionmoves1} to $h_n f_n Z(\mu)$ is to obtain either $Z(\mu_1)$, or the complement $r(\mu_1) \partial W_\GG \setminus Z(\mu_1)$, where $\mu_1$ is a $\GG$-path.  The result of applying Lemma \ref{lem: actionmoves2} to the element $u_{r(f_n),h_n} t_{f_n} s_\mu (u_{r(f_n),h_n} t_{f_n} s_\mu)^*$ is to obtain either $s_{\nu_1} s_{\nu_1}^*$, or $u_{r(\nu_1),1} - s_{\nu_1} s_{\nu_1}^*$, where $\nu_1$ is a $\GG$-path.  The parallel structures of these two lemmas ensures that $\mu_1 = \nu_1$, and that either both result in the difference, or neither results in the difference.  We may repeat this with $h_{n-1} f_{n-1} Z(\mu_1)$ or $h_{n-1} f_{n-1} (r(\mu_1) \partial W_\GG \setminus Z(\mu_1))$, and with $$u_{r(f_{n-1}),h_{n-1}} t_{f_{n-1}} s_{\mu_1} s_{\mu_1}^* (u_{r(f_{n-1}),h_{n-1}} t_{f_{n-1}})^*$$ or $$u_{r(f_{n-1}),h_{n-1}} t_{f_{n-1}} (u_{r(\mu_1),1} - s_{\mu_1} s_{\mu_1}^* (u_{r(f_{n-1}),h_{n-1}} t_{f_{n-1}})^*,$$ and so on.  After $n$ iterations, we have established \eqref{eq: covariance g}.  An analogous argument establishes \eqref{eq: covariance e}.  Hence $(j_{A,v}, j_{\pi,v})$ is a covariant representation of $( C( v\partial X_\GG), \pi_1(\GG,v),\tau)$ in the corner $u_{v,1} C^*(\GG) u_{v,1}$.

We may now define homomorphisms\footnote{Here, $M(C^*(\GG))$ denotes the multiplier algebra of $C^*(\GG)$. Recall that that the multiplier algebra $M(A)$ of a $C^*$-algebra $A$ is the largest unital $C^*$-algebra that contains $A$ as an essential ideal. See \cite{Mur} for more details.}  $j_A : C( v\partial X_\GG) \to M ( C^*(\GG) )$ and $j_\pi : \pi_1(\GG,v) \to \UU M ( C^*(\GG) )$ by
\[
j_A (f) = \sum_{x \in \Gamma^0} t_{[x,v]} j_{A,v} (f) t_{[v,x]}
\]
and
\[
j_\pi(\gamma) = \sum_{x \in \Gamma^0} t_{[x,v]} j_{\pi,v} (\gamma) t_{[v,x]},
\]
where $f \in C(v\partial X_\GG)$ and $\gamma \in \pi_1(\GG,v)$.  The sums we consider (here, and after) converge in the strict topology of $M(C^*(\GG))$, since the $t_{[x,y]}$ are matrix units.  We have 
\[
j_A(1)= \sum_{x \in \Gamma^0} t_{[x,v]} j_{A,v} (1) t_{[v,x]} =\sum_{x\in\Gamma^0} u_{x,1} = 1_{M ( C^*(\GG) )},
\]
so $j_A$ is nondegenerate. For each $f\in C( v\partial X_\GG)$ and $\gamma \in \pi_1(\GG,v)$ we have
\begin{align*}
j_\pi ( \gamma ) j_A (f) j_\pi(\gamma)^* &= \sum_{x \in \Gamma^0} t_{[x,v]} j_{\pi,v} (\gamma) j_{A,v}(f) j_{\pi,v} (\gamma)^* t_{[v,x]}\\
&=
\sum_{x \in \Gamma^0} t_{[x,v]} j_{A,v} (\tau_\gamma(f)) t_{[v,x]}\\
&=j_A (\tau_\gamma(f) ).
\end{align*}
So $(j_A, j_\pi)$ is a covariant representation.

The universal property thus gives a homomorphism $j_A \times j_\pi$ from $C ( v\partial X_\GG )\rtimes_\tau \pi_1 (\GG,v)$ to $M(C^*(\GG))$. Since for all $x,y \in \Gamma^0$ we have $j_A(f) t_{[x,y]} = t_{[x,v]} j_{A,v} (f) t_{[v,y]} = t_{[x,y]} j_A (f)$, and similarly $j_\pi ( \gamma ) t_{[x,y]} = t_{[x,y]} j_\pi ( \gamma )$, it follows that the range of $j_A \times j_\pi$ commutes with the range of $j_\KK$. We thus get the desired homomorphism $\Psi := j_\KK \otimes ( j_A \times j_\pi)$ from  $\KK( \ell^2 ( \Gamma^0) )\otimes \big(C( v\partial X_\GG) \rtimes_\tau \pi_1 ( \GG,v ) \big)$ to $C^*(\GG)$.
\end{proof}

\begin{lem}\label{lem: Psi surjective}
The homomorphism $\Psi : \KK( \ell^2 ( \Gamma^0 ) )\otimes 
\big( C( v\partial X_\GG ) \rtimes_\tau \pi_1 (\GG,v) \big) \to C^*(\GG)$ from Proposition~\ref{prop: the Psi map} is surjective.
\end{lem}

\begin{proof}
It suffices to show that each $u_{x,g}$ and $s_e$ lie in the image.  For $x \in \Gamma^0$ and $g \in G_x$ we have
\[
u_{x,g}=t_{[x,v]} (t_{[v,x]} u_{x,g} t_{[x,v]}) t_{[v,x]} = j_\KK (\theta_{x,v}) j_{\pi,v} (\varepsilon(g) ) j_\KK( \theta_{v,x}) = j_\KK (\theta_{x,v} ) j_\pi(\varepsilon(g)) j_\KK(\theta_{v,x}),
\]
so each $u_{x,g}$ is in the image of $\Psi$.

For $e \in \Gamma^1$ we claim that
\[
s_e = \begin{cases}
j_\KK (\theta_{r(e),s(e)}) j_A(\chi_{Z([v,r(e)]1e)})  j_\pi(\varepsilon(e)) &
\text{if }
\underleftarrow{e} \ne \overline{e} \\
j_\KK (\theta_{r(e),s(e)}) j_A(\chi_{Z([v,r(e)])^c})
& \text{if }
\underleftarrow{e} = \overline{e}.
\end{cases}
\]
First suppose that  $\underleftarrow{e} \ne \overline{e}$, so that $[v,r(e)]1e$ is a reduced $\GG$-path. Using the description of the maps defined in Proposition~\ref{prop: the Psi map}, we compute
\begin{align*}
j_\KK (\theta_{r(e),s(e)}) & j_A(\chi_{Z([v,r(e)]1e)})  j_\pi(\varepsilon(e))\\
&=
t_{[r(e),s(e)]}  \Big( \sum_{x \in \Gamma^0} t_{[x,v]} j_{A,v} (\chi_{Z([v,r(e)]1e}) t_{[v,x]} \Big)\Big( \sum_{x \in \Gamma^0} t_{[x,v]} j_{\pi,v} (\varepsilon(e)) t_{[v,x]} \Big)\\
&=
t_{[r(e),s(e)]}  (t_{[s(e),v]} s_{[v,r(e)]}s_e s_e^* s_{[v,r(e)]}^* t_{[v,s(e)]})( t_{[s(e),v]} t_{[v,r(e)]} t_e t_{[s(e),v]} t_{[v,s(e)]})\\
&=
(t_{[r(e),v]}s_{[v,r(e)]}s_e)(t_{[r(e),v]}s_{[v,r(e)]}s_e)^* t_e.
\end{align*}
Write $[v,r(e)]_T = f_1 \dots f_k.$ Repeated applications of Lemma~\ref{lem: actionmoves2}\eqref{lem: actionmoves2.2} gives
\[
t_{[r(e),v]} s_{[v,r(e)]}s_e
= t_{\overline{f_k}} \dots t_{\overline{f_2}} t_{\overline{f_1}}  s_{f_1}s_{f_2} \dots s_{f_k} s_e
= t_{\overline{f_k}} \dots t_{\overline{f_2}} s_{f_2} \dots s_{f_k} s_e = \dots 
= s_e.
\]
Consequently, we have
\begin{align*}
j_\KK (\theta_{r(e),s(e)}) j_A(\chi_{Z([v,r(e)]1e)})  j_\pi(\varepsilon(e))
= s_e s_e^* t_e = s_e
\end{align*}
and so in this case $s_e$ is in the image of $\Psi$.

We now consider the case where $\underleftarrow{e}=\overline{e}$, so in particular $\underleftarrow{(\overline{e})} \ne e$. Since $e \in T^1$, it follows that $\varepsilon(\overline{e}) = 1$. The preceding calculations give \[s_{\overline{e}}= j_\KK (\theta_{s(e),r(e)}) j_A (\chi_{Z ( [v,s(e)]1\overline{e})})
= j_\KK (\theta_{s(e),r(e)}) j_A (\chi_{Z ( [v,r(e)])}),\]
so that
\[
s_e^*
= t_{\overline{e}} - s_{\overline{e}} = j_\KK(\theta_{s(e),r(e)}) - j_\KK(\theta_{s(e),r(e)}) j_A (\chi_{Z ( [v,r(e)])}) = j_\KK(\theta_{s(e),r(e)}) j_A (\chi_{Z ([v,r(e)])^c}).
\]
Hence,
$
s_e
= j_\KK(\theta_{r(e),s(e)}) j_A (\chi_{Z ( [v,r(e)])^c}),
$ and so $\Psi$ is surjective.
\end{proof}

We need one more result.

\begin{lem}\label{lem: key identities for PhiPsi}  Let $\{ U_x, S_e : x\in \G^0, e \in \G^1\}$ be the $\GG$-family from Proposition~\ref{prop: G family in tensor}.
	\begin{enumerate}
		\item For each $e\in\Gamma^1$ we have $T_e=\theta_{r(e),s(e)}\otimes 
		i_\pi(\varepsilon(e)).$ In particular, for $e\in T^1$ we have 
		$T_e=\theta_{r(e),s(e)}\otimes 1$.
		\item For each $x,y\in\Gamma^0$ we have 
		$T_{[x,y]}=\theta_{x,y}\otimes 
		1$.
		\item \label{eq: cap S mu take 2} For each $\mu=g_1e_1\dots g_ne_n\in v\GG^*$ we have
		\[
		S_\mu =\theta_{v,s(e_n)}\otimes 
		i_A(\chi_{Z(\mu)})i_{\pi}(\mu [s(\mu),v]).
		\]
	\end{enumerate}
\end{lem} 

\begin{proof}
	For (1), first suppose that $e\not\in T^1$. Then $\varepsilon(e) Z([v,s(e)]1\overline{e}) = Z([v,r(e)]1e)^c$, by Lemma \ref{lem: actiondetails}\eqref{lem: actiondetails.4}. Now covariance of $(i_A, i_\pi)$ and the previous identity gives
	\begin{align*}
		T_e&=S_e+S_{\overline{e}}^*\\
		&= \theta_{r(e),s(e)}\otimes 
		\Big(i_A(\chi_{Z([v,r(e)]1e)}) i_{\pi}(\varepsilon(e))
		+i_{\pi}(\varepsilon(\overline{e}))^*i_A(\chi_{Z([v,r(\overline{e})]1\overline{e})})\Big)\\
		&=\theta_{r(e),s(e)}\otimes 
		\Big(i_A(\chi_{Z([v,r(e)]1e)}) i_{\pi}(\varepsilon(e))
		+i_{\pi}(\varepsilon(e))i_A(\chi_{Z([v,s(e)]1\overline{e})}) \Big)\\
		&= \theta_{r(e),s(e)}\otimes 
		\Big(i_A(\chi_{Z([v,r(e)]1e)}) i_{\pi}(\varepsilon(e))+
		i_A(\chi_{Z([v,r(e)]1e)^c}) i_{\pi}(\varepsilon(e))\Big)\\
		&= \theta_{r(e),s(e)} \otimes i_\pi(\varepsilon(e)).
	\end{align*}   
	
	Now suppose that $e\in T^1$. Then either $\underleftarrow{e}=\overline{e}$ 
	and $\underleftarrow{(\overline{e})}\not=e$, or $\underleftarrow{e}\not=\overline{e}$ and 
	$\underleftarrow{(\overline{e})}=e$. Suppose 
	$\underleftarrow{e}\not=\overline{e}$. Then
	\begin{align*}
		T_e=S_e+S_{\overline{e}}^* &= \theta_{r(e),s(e)}\otimes \Big(
		i_A(\chi_{Z([v,r(e)]1e)})+i_A(\chi_{Z([v,r(\overline{e})])^c})\Big)\\
		&= \theta_{r(e),s(e)}\otimes \Big(
		i_A(\chi_{Z([v,s(e)])})+i_A(\chi_{Z([v,s(e)])^c})\Big)\\
		&=\theta_{r(e),s(e)}\otimes 1.
	\end{align*}
	
	\noindent
	A similar calculation gives $T_e=\theta_{r(e),s(e)}\otimes 1$ in the case 
	$\underleftarrow{e}=\overline{e}$.

	For (2), if $x,y\in\Gamma^0$ and $[x,y]=1e_1\dots 
	1e_n$, then
	\begin{align*}
		T_{[x,y]}&=U_{x,1}T_{e_1}\dots U_{r(e_n),1}T_{e_n}\\
		&= (\theta_{x,x}\otimes 1)(\theta_{x,s(e_1)}\otimes 1)\dots 
		(\theta_{r(e_n),y}\otimes 1)\\
		&= \theta_{x,y}\otimes 1.
	\end{align*}
	So (2) holds.
	
	We prove (\ref{eq: cap S mu take 2}) by induction on the length of $\mu$. 
	First 
	consider $\mu=ge\in v\GG^1$. Then $\underleftarrow{e}=v$ and $[v,r(e)]=v$. 
	Covariance of $(i_A, i_\pi)$ gives 
	$i_\pi(g) i_A(\chi_{Z(1e)}) = i_A(\chi_{Z(ge)}) i_\pi(g)$, and hence
	\begin{align*}
		S_\mu= U_{v,g}S_e &= \theta_{v,s(e)}\otimes 
		i_\pi(g) i_A(\chi_{Z(1e)}) i_\pi(\varepsilon(e))\\
		&= \theta_{v,s(e)}\otimes i_A(\chi_{Z(ge)})i_\pi(g\varepsilon(e))\\
		&= \theta_{v,s(e)}\otimes i_A(\chi_{Z(ge)})i_\pi(ge[s(e),v]),
	\end{align*}
	
	\noindent
	which is (\ref{eq: cap S mu take 2}) for the path $ge$ in this case.
	
	Now suppose (\ref{eq: cap S mu take 2}) holds for $\mu=g_1 e_1\dots 
	g_{n-1} e_{n-1}$, 
	and let $g_n e_n\in s(\mu)\GG^1$ with $g_ne_n \not= 1\overline{e_{n-1}}$. First 
	suppose that $\underleftarrow{e_n} \not=\overline{e_n}$. Then
	\begin{align*}
		S_{\mu g_ne_n} &= S_\mu U_{v,g_n} S_{e_n}\\
		&= \theta_{v,s(e_n)} \otimes 
		i_A(\chi_{Z(\mu)})i_\pi(\mu [s(\mu),v] ) i_\pi(\varepsilon(g_n))
		i_A(\chi_{Z([v,r(e_n)]1e_n)}) i_{\pi}(\varepsilon(e_n))
	\end{align*}
	
	\noindent
	The product of $\mu [s(\mu),v]$ and $\varepsilon(g_n)$ in $\pi_1(\GG,v)$ is 
	\[
	\mu [s(\mu),v] 1\varepsilon(g_n)=\mu [s(\mu),v] 1 [v,s(\mu)] g_n [s(\mu),v]
	=\mu g_n [s(\mu),v] .
	\]
	
	\noindent
	We now have
	\[
	\mu g_n [s(\mu),v] Z([v,r(e_n)]1e_n)=Z(\mu g_ne_n)
	\]
	
	\noindent
	and
	\[
	\mu g_n [s(\mu),v] \varepsilon(e_n)=\mu g_n e_n [s(e_n),v] .
	\]
	
	\noindent
	We use covariance of $(i_A, i_\pi)$ and these identities to get 
	\begin{align*}
		S_{\mu g_ne_n} &= \theta_{v,s(e_n)}\otimes 
		i_A(\chi_{Z(\mu)}) i_A(\chi_{Z(\mu g_ne_n)})i_\pi(\mu g_ne_n [s(e_n),v] )\\
		&= \theta_{v,s(e_n)}\otimes i_A(\chi_{Z(\mu g_n e_n)}) i_\pi( \mu 
		g_ne_n [s(e_n),v] ),
	\end{align*}
	which is (\ref{eq: cap S mu take 2}) for the path $\mu g_n e_n$ in this case.  
	
	Now suppose $\underleftarrow{e_n}=\overline{e_n}$. Then 
	$[s(\mu),v]=[r(e_n),v]$ has $e_n$ as its rangemost edge. Hence
	\[
	\mu g_n [s(\mu),v] Z( [v,r(e_n)])^c = Z(\mu g_ne_n),
	\]
	by Lemma~\ref{lem: actionmoves1}\eqref{lem: actionmoves1.3}.  So
	\begin{align*}
		S_{\mu g_n e_n} &= \theta_{v,s(e_n)}\otimes 
		i_A(\chi_{Z(\mu)})i_\pi(\mu g_n [s(\mu),v]) i_A(\chi_{Z([v,r(e)])^c})\\
		&= \theta_{v,s(e_n)}\otimes 
		i_A(\chi_{Z(\mu)}) i_A(\chi_{Z(\mu g_ne_n)}) i_\pi(\mu g_n [s(\mu),v] )\\
		&= \theta_{v,s(e_n)}\otimes i_A(\chi_{Z(\mu g_n e_n)}) i_\pi(\mu 
		g_n [s(\mu),v] )\\
		&= \theta_{v,s(e_n)}\otimes i_A(\chi_{Z(\mu g_ne_n)}) i_\pi(\mu 
		g_ne_n [s(e_n),v] ),
	\end{align*}
	which is (\ref{eq: cap S mu take 2}) for the path $\mu g_n e_n$ in this case. 
	This completes the 
	proof of \eqref{eq: cap S mu take 2} for all $\mu \in v\GG^*$.
\end{proof}

We can now finish the proof of our main theorem.

\begin{proof}[Proof of Theorem~\ref{thm: main theorem}]
We proved in Proposition~\ref{prop: G family in tensor} that the collections of
\[
U_{x,g}:=\theta_{x,x} \otimes i_{\pi}(\varepsilon(g))
\]

\noindent
and
\[
S_e:=
\begin{cases}
\theta_{r(e),s(e)}\otimes i_A ( \chi_{Z ( [v,r(e)]1e)} ) i_{\pi} ( \varepsilon(e) )
& \text{if $\underleftarrow{e}\not=\overline{e}$}\\
\theta_{r(e),s(e)}\otimes i_A (\chi_{Z ( [v,r(e)] )^c} )
& \text{if $\underleftarrow{e}=\overline{e}$}
\end{cases}
\] 
are a $\GG$-family in $\KK ( \ell^2 ( \Gamma^0 ) )\otimes 
\big( C ( v\partial X_\GG ) \rtimes_\tau \pi_1 ( \GG,v ) \big)$. The universal 
property of $C^*(\GG)$ now 
gives a homomorphism $\Phi : C^*(\GG) \to \KK ( \ell^2 ( \Gamma^0 ) )\otimes 
\big( C ( \partial( X_{\GG,v} ) \rtimes_\tau \pi_1 ( \GG,v ) \big)$ satisfying 
$\Phi(u_{x,g}) = U_{x,g}$ and 
$\Phi(s_e) = S_e$, for all $x \in \Gamma^0$, $g\in G_x$ and $e \in \Gamma^1$. We 
claim that $\Phi$ is an isomorphism, 
with inverse $\Psi$ as given in Proposition~\ref{prop: the Psi map}. 

We proved in Lemma~\ref{lem: Psi surjective} that $\Psi$ is surjective. Thus it suffices to prove that 
$\Phi\circ\Psi=\id$, and for this it suffices to prove that $\Phi\circ\Psi$ is 
the identity on the following elements: $\theta_{x,y}\otimes 1$ for all 
$x,y \in \Gamma^0$; $\theta_{v,v} \otimes 
i_A ( \chi_{Z( \mu )} )$ for all $\mu \in 
v\GG^*$; $\theta_{v,v} \otimes i_{\pi} ( \varepsilon(g) )$ for all $x\in\Gamma^0$, 
$g \in G_x$; and $\theta_{v,v} \otimes i_{\pi} (\varepsilon (e) )$ for all 
$e\in\Gamma^1$.  For this, we use the identities in Lemma~\ref{lem: key identities for PhiPsi}.

Fix $x,y \in\Gamma^0$. Then we know from Lemma~\ref{lem: key identities for 
PhiPsi}(2) that
\[
\Phi ( \Psi ( \theta_{x,y} \otimes 1) ) = \Phi ( t_{[x,y]})=T_{[x,y]} = \theta_{x,y} \otimes 1.
\]
Next, fix $\mu\in v\GG^*$. Then we know from Lemma~\ref{lem: key identities for 
PhiPsi}\eqref{eq: cap S mu take 2} that 
\begin{align*}
\Phi ( \Psi ( \theta_{v,v} \otimes i_A ( \chi_{Z(\mu)}) ) )&=\Phi( s_\mu s_\mu^* )\\
&=S_\mu  S_\mu^* \\
&=\theta_{v,v} \otimes 
i_A ( \chi_{Z( \mu )}) i_{\pi} (\mu [s(\mu),v] ) \big( 
i_A (\chi_{Z ( \mu)} ) i_{\pi} (\mu [s(\mu),v] )\big)^*\\
&= \theta_{v,v} \otimes i_A ( \chi_{Z ( \mu)} ).
\end{align*}
Now fix $e\in\Gamma^1$. Then we know from Lemma~\ref{lem: key identities for 
PhiPsi}(1) and (2) that
\begin{align*}
\Phi ( \Psi ( \theta_{v,v} \otimes i_\pi ( \varepsilon(e) ) ) ) &=T_{[v,r(e)]} U_{r(e),1} T_e T_{[s(e),v]} U_{v,1}\\
&= ( \theta_{v,r(e)} \otimes 1 ) ( \theta_{r(e),s(e)} \otimes 
i_\pi ( \varepsilon(e) ) )( \theta_{s(e),v} \otimes 1 )\\
&= \theta_{v,v} \otimes i_\pi ( \varepsilon(e) ).
\end{align*}
Finally, fix $x \in \Gamma^0$ and $g \in G_x$. Then we know from Lemma~\ref{lem: key 
identities for 
PhiPsi}(2) that
\[
\Phi ( \Psi ( \theta_{v,v} \otimes i_\pi ( \varepsilon(g) ) ) ) = T_{[v,x]} U_{x,g} T_{[x,v]} U_{v,1} = \theta_{v,v} \otimes  i_\pi ( \varepsilon(g) ).
\]

\noindent
Hence $\Phi \circ \Psi = \id$, and $\Phi$ is an isomorphism with inverse $\Psi$.
\end{proof}

\begin{rmk}\label{rem:groupoidCStar}
Recall from Section~\ref{subsec: groupoid} that the fibred product groupoid $F(\GG) * \partial W_\GG$ is a Hausdorff \'etale groupoid, with unit space $\partial W_\GG$. It is easily seen that $v\partial W_\GG = v \partial X_\GG$ is a compact-open transversal in the sense of \cite[Example 2.7]{MRW}.  By \cite[Theorem 2.8]{MRW}, it follows that the semidirect product groupoid $\pi_1(\GG,v) \ltimes v\partial X_\GG$ and $F(\GG) * \partial W_\GG$ are equivalent groupoids, and hence have stably isomorphic $C^*$-algebras.
\end{rmk}

\section{On the action $\pi_1(\GG,v)\curvearrowright v\partial X_{\GG}$}\label{sec: structural props}

In this section we examine properties of the action of the fundamental group $\pi_1(\GG,v)$ on the boundary $v\partial X_\GG$. In Section~\ref{subsec: minimality} we characterise when the action is minimal, and in Section~\ref{subsec: loc contractivity} we give a sufficient condition under which it is locally contractive. Characterising topological freeness turns out to be a harder problem, and in Section~\ref{subsec: top freeness} we discuss some specific examples. We finish in Section~\ref{subsec: amenability} with a short word on amenability of the action. Recall that we assume all our graphs of groups have countable vertex groups, and hence the fundamental group is countable (and will be equipped with the discrete topology).

\subsection{Minimality}\label{subsec: minimality}

Recall that the action of a discrete group $\Lambda$ on a locally compact Hausdorff space $Y$ is \textit{minimal} if every orbit of points of $Y$ is dense in $Y$.  A Hausdorff \'etale groupoid $G$ is \textit{minimal} if every orbit in $G^0$ is dense in $G^0$.  We typically check minimality by considering an arbitrary open set $U \subset Y$ and point $y \in Y$, and proving that there is $\lambda \in \Lambda$ such that $\lambda y \in U$.  

In this section, we will first give a condition on the boundary $\partial W_\GG$ which is equivalent to minimality of the action of $\pi_1(\GG,v)$ on $v\partial X_\GG$.  Then we will give an explicit, checkable condition on $\GG$ which is equivalent to the negation of the first condition. We begin with some definitions.

\begin{dfn}\label{def: no cancellation}
	Let $\mu_1$, $\ldots$, $\mu_k$ be reduced $\GG$-words so that $s(\mu_i) = r(\mu_{i+1})$ for $1 \leq i < k$. We say that the concatenation $\mu_1 \dots \mu_k$ \textit{has no cancellation} if after reduction, the length of $\mu_1 \dots \mu_k$ is $\sum_{i=1}^k|\mu_i|$. In other words, the process of putting $\mu_1 \dots \mu_k$ into reduced form does not result in any instance of the form $e 1 \overline{e}$ for an edge $e \in \Gamma^1$.
\end{dfn}

\begin{dfn}\label{def: lies on}  Let $e \in \G^1$ and $\xi \in \partial W_\GG$, with $\xi = h_1 f_1 h_2 f_2 \dots$.  We say that $e$~\emph{lies on}~$\xi$ if $e = f_i$ for some $i$.
\end{dfn}

\begin{dfn}\label{def: can flow to}
	
	Let $e,f \in \Gamma^1$.  We say that $f$ {\em can flow to} $e$ if $f$ lies on $\xi$ for some $\xi \in Z(1e)$, and $f$ is not the rangemost edge of $\xi$.  We say that $\xi \in\partial W_\GG$ {\em can flow to} $e\in\Gamma^1$ if $f$ can flow to $e$ for some $f$ that lies on $\xi$.
	
\end{dfn}

The following lemma describes the situation of Definition \ref{def: can flow to} more precisely.  The proof is elementary, and is omitted.

\begin{lem} \label{def: suitability} Let $e$, $f \in \Gamma^1$.  Then $f$ can flow to $e$ if and only if at least one of the following conditions holds:
	
	\begin{enumerate}
		
		\item \label{def: suitability.1}  there is a reduced $\GG$-word $\nu$ with $r(\nu) = s(e)$, $s(\nu) = r(f)$, $|\nu| \ge 1$, and such that $1e\nu 1f$ has no cancellation;
		
		\item \label{def: suitability.2} $s(e) = r(f)$ and $f \not= \overline{e}$; or
		
		\item \label{def: suitability.3}  $s(e) = r(f)$, $f = \overline{e}$ and $\alpha_{\overline{e}}$ is not surjective.
		
	\end{enumerate}
	
\end{lem}

We recall from Remark~\ref{rem:groupoidCStar} that $\pi_1(\GG,v) \ltimes v\partial X_\GG$ and $F(\GG) * \partial W_\GG$ are equivalent groupoids.  It follows that either both are minimal or neither is minimal.  Thus both groupoids are used in the statement of the next theorem, as the second will be used in the course of the proof.

\begin{thm}\label{thm: minimality}
	
	The action of $\pi_1(\GG,v)$ on $v\partial X_\GG$ is minimal if and only if $\xi$ can flow to $e$ for every  $\xi \in \partial W_\GG$ and $e \in \G^1$.
	
\end{thm}

\begin{proof}
	We start with the ``if'' direction. Suppose $\mu=g_1e_1\dots g_ne_n\in v\GG^*$ and $\xi=h_1f_1h_2f_2\dots\in v\partial X_\GG$ with $\xi\not\in Z(\mu)$. We need to find $\gamma\in\pi_1(\GG,v)$ with $\gamma \xi\in Z(\mu)$. In the case that $e_n=f_i$ for some $i$, the fundamental group element
	\[
	\gamma= g_1 e_1 \dots g_n e_n 1\overline{f_i}h_i^{-1}\dots h_2^{-1}\overline{f_1}h_1^{-1}
	\]
	satisfies $\gamma\xi=\mu h_{i+1}f_{i+1}\dots\in Z(\mu)$. 
	
	Now suppose $e_n$ does not lie on $\xi$. We know by assumption that there is some $f_i$ that can flow to $e_n$. If Lemma~\ref{def: suitability}\eqref{def: suitability.1} holds, then there is some reduced $\GG$-word $\nu$ with $r(\nu) = s(e_n)$, $s(\nu) = r(f_i)$, $|\nu| \ge 1$, and such that $1e_n\nu f_i$ has no cancellation. Then   
	\[
	\gamma= \mu\nu h_i^{-1}\overline{f_{i-1}}\dots h_2^{-1}\overline{f_1}h_1^{-1}
	\] 
	satisfies $\gamma\xi=\mu\nu f_i h_{i+1}f_{i+1}\dots\in Z(\mu)$. If Lemma~\ref{def: suitability}\eqref{def: suitability.2} holds, then  
	\[
	\gamma= \mu h_i^{-1}\overline{f_{i-1}}\dots h_2^{-1}\overline{f_1}h_1^{-1}
	\] 
	satisfies $\gamma\xi=\mu 1f_ih_{i+1}f_{i+1}\dots\in Z(\mu)$. If Lemma~\ref{def: suitability}\eqref{def: suitability.3} holds, then we choose nontrivial $h\in\Sigma_{f_i}$. Then
	\[
	\gamma=\mu h h_i^{-1}\overline{f_{i-1}}\dots h_2^{-1}\overline{f_1}h_1^{-1}
	\]
	satisfies $\gamma\xi= \mu h \overline{e_n}h_{i+1}f_{i+1}\dots\in Z(\mu)$.
	
	For the ``only if'' direction we suppose that $\xi=h_1f_1h_2f_2\dots\in\partial W_\GG$ and $e\in\Gamma^1$ have the property that $e$ does not lie on $\xi$, and no $f_i$ flows to $e\in\Gamma^1$. We claim that $F(\GG)\cdot\xi\cap Z(1e)=\emptyset$, and hence that the action is not minimal. Suppose for contradiction that $F(\GG)\cdot\xi\cap Z(1e)\not=\emptyset$, and $\gamma\in F(\GG)$ with $\gamma\xi\in Z(1e)$. We know that, after reduction, $\gamma\xi$ has the form $\beta f_ih_{i+1}f_{i+1}\dots$ for some reduced $\GG$-word $\beta$ and some $i$. We claim that the length of $\beta$ is at least one. For if not, then $\gamma \xi\in Z(1e)$ forces $e$ to lie on $\xi$, contradicting our assumption. So $|\beta|\ge 1$, and hence $\beta=1e\beta'$ for some reduced $\GG$-word $\beta'$ for which $1e\beta'$ has no cancellation. If $|\beta'|\ge 1$, then $f_i$ flows to $e$, contradicting our assumption. On the other hand, if $|\beta'| = 0$, then $\beta' \in G_{s(e)}$, and $s(e) = r(f_i)$. Now our assumptions imply that $f_i = \overline{e}$ and that $\alpha_{\overline{e}}$ is onto. Write $\beta'=\alpha_{\overline{e}}(h)$ for some $h\in G_{\overline{e}}$. Then
	\[
	\gamma\xi=\beta f_i h_{i+1}f_{i+1}\dots = 1e\beta'\overline{e}h_{i+1}f_{i+1}\dots = 1e\overline{e}\alpha_e(h)h_{i+1}f_{i+1}\dots = \alpha_e(h)h_{i+1}f_{i+1}\dots,
	\]
	which we know is not in $Z(1e)$ because $e\not= f_{i+1}$. So we get a contradiction, and hence we must have $F(\GG)\cdot\xi\cap Z(1e)=\emptyset$. This means the action is not minimal, and we are done.
\end{proof}

We now aim to give a characterisation of when the action is {\em not} minimal in terms of readily checkable conditions on the underlying graph of groups $\GG$.  For this, we begin by defining a number of subgraphs of $\Gamma$.  In each case, we give the edge and vertex sets, and the range and source maps are then the restriction of the range and source maps for $\G$.

\begin{dfn}  We define the graphs $E_{s,e}$ and $E_{s,\eta,f_j}$ as follows.
	\begin{enumerate}
		\item[(1)] Suppose $e\in\Gamma^1$ satisfies $s(e)\not=r(e)$. We define $E_{s,e}^0$ to be $\{s(e)\}$ together with the set of vertices $x\in\Gamma^0$ such that there exists a reduced path $e_1\dots e_n\in s(e)\Gamma^* x$ with $e_1\not=\overline{e}$. We define $E_{s,e}^1$ to be the set of edges appearing on these paths, together with their reversals.
		
		\item[(2)] Suppose $e\in\Gamma^1$ satisfies $s(e)=r(e)$. We define $E_{s,e}^0$ to be $\{s(e)\}$ together with the set of vertices $x\in\Gamma^0$ such that there exists a reduced path $e_1\dots e_n\in s(e)\Gamma^* x$ with $e_1\not\in\{e,\overline{e}\}$. We define $E_{s,e}^1$ to be the set of edges appearing on these paths, together with their reversals.
		
		\item[(3)] Suppose $\eta = f_1 f_2 \dots f_m$ is a minimal cycle with $m\ge 2$. We define $E_{s,\eta,f_j}^0$ to be $\{s(f_j)\}$ together with the set of vertices $x\in\Gamma^0$ such that there exists a reduced path $e_1\dots e_n\in s(f_j)\Gamma^* x$ with $e_1\not\in\{f_{j+1},\overline{f_j}\}$ (where we index mod $m$). We define $E_{s,\eta,f_j}^1$ to be the set of edges appearing on these paths, together with their reversals.
	\end{enumerate}
\end{dfn}
The graph $E_{s,e}$ might be termed the \emph{graph upstream from $e$}.

\begin{dfn}\label{def: treelike}
	
	We say that $\GG$ is \textit{treelike at the edge $e$} if the subgraph $E_{s,e}$ of $\Gamma$ is a tree, and if for each edge $f$ of $E_{s,e}$ that points towards $s(e)$, we have that $\alpha_{\overline{f}}$ is onto. If $\eta = e_1 e_2 \dots e_n$ is a minimal cycle of length greater than one, we say that $\GG$ is \textit{treelike at $\eta$} if for each $1 \le i \le n$, the graph $E_{s,\eta,e_i}$ is a tree, and if for each edge $f$ of $E_{s,\eta,e_i}$ that points towards $s(e_i)$ we have that $\alpha_{\overline{f}}$ is onto.
	
\end{dfn}

\noindent We note that if $E_{s,e}$ (or $E_{s,\eta,e_i}$) is larger than a single vertex, then it must be infinite, by nonsingularity of $\GG$.

\begin{dfn}\label{def:ConstantTree}
	
	Let $\GG$ be a locally finite nonsingular graph of groups, and let $e \in \Gamma^1$.  We say that $\GG$ has a \textit{constant tree at $e$} if $\GG$ is treelike at $e$, $E_{s,e}$ is nontrivial, and if for all $f \in E_{s,e}^1$ we have that $\alpha_f$ is onto. (We say that a tree is \emph{trivial} if it consists of just a single vertex; the tree $E_{s,e}$ might be trivial, in which case it consists of only the vertex $s(e)$.)
	
\end{dfn}

We are now ready to state our characterisation of non-minimality.  

\begin{thm} \label{thm: nonminimality}
	The action of $\pi_1(\GG,v)$ on $v\partial X_\GG$ is \emph{not} minimal if and only if there exists an edge $e \in \Gamma^1$ such that $\alpha_{\overline{e}}$ is surjective, and one of the following holds:
	\begin{enumerate}
		\item[(a)] \label{thm: treelike_one} $e$ is a loop i.e. $r(e) = s(e)$, and $\GG$ is treelike at $e$;
		\item[(b)] \label{thm: treelike_two} $e$ lies on a minimal cycle $\eta$ of length greater than one, $\alpha_{\overline{f}}$ is surjective for each edge $f$ that lies on $\eta$, and $\GG$ is treelike at $\eta$; or
		\item[(c)] \label{thm: treelike_three} $e$ does not lie on any minimal cycle of $\Gamma$, $\GG$ is treelike at $e$, and there exists $\xi \in r(e)\partial W_\GG$ such that $e$ does not lie on $\xi$.
	\end{enumerate}
\end{thm}

The proof of this result is fairly technical.  The basic idea is to describe how nonminimality can occur.  In cases (a) and (b) of the theorem, it is $\overline{e}^\infty$, respectively $\overline{\eta}^\infty$, whose orbit does not intersect $Z(1 e)$.  In case (c) it is the postulated path $\xi$ that has this property.

\begin{proof}  We show that there is $\xi\in\partial W_\GG$ and $e\in\Gamma^1$   such that $\xi$ cannot flow to $e$ if and only if $\alpha_{\overline{e}}$ is surjective and either (a), (b) or (c) holds. The result will then follow from Theorem~\ref{thm: minimality}.
	
	We start with the ``if'' direction.  Suppose that there exists an edge $e \in \G^1$ such that $\alpha_{\overline e}$ is surjective.  Assume first that (a) holds, so $e$ is a loop and $\GG$ is treelike at $e$.  We claim that $1\overline{e}1\overline{e}\dots\in r(e)\partial W_\GG$ cannot flow to $e$; i.e. that $\overline{e}$ cannot flow to $e$.  If it did, then as $\alpha_{\overline{e}}$ is surjective it must be Lemma~\ref{def: suitability}\eqref{def: suitability.1} that holds.  Let $\nu$ be  a reduced $\GG$-word with $r(\nu) = s(e) = r(e)$, $s(\nu) = r(\overline{e}) = r(e)$, $|\nu| \geq 1$ and such that $1e \nu 1\overline{e}$ is reduced.  Then $\nu = 1 \nu_1 g_2 \nu_2 \dots g_m \nu_m g_{m+1}$ with $m \ge 1$, and $\nu_1 \not= \overline{e}$.  If $\nu_1 = e$ we can delete $1 \nu_1$, and the resulting shortened $\GG$-word still works.  Thus we may assume that $\nu_1 \not= e$, $\overline{e}$.  Then $\nu_1 \dots \nu_m$ is a path in $E_{s,e}^*$.  Since $\GG$ is treelike at $e$, it must be the case that $s(\nu_m) \not= r(e)$, a contradiction. 
	
	Now assume that (b) holds, with $e$ lying on a minimal cycle $\eta = e_1\dots e_n$.  Say $e=e_1$.  We claim that $\overline{e_{n}}\dots \overline{e_1} \, \overline{e_{n}}\dots \overline{e_1}\dots\in r(e)\partial W_\GG$ cannot flow to $e$; i.e. that $\overline{e_i}$ cannot flow to $e$ for all $i$.  Suppose for contradiction that there is $i$ such that $\overline{e_i}$ can flow to $e$.  Since $\eta$ is minimal, it is possible that $s(e) = r(\overline{e_i})$ if and only if $i = 1$. Then Lemma \ref{def: suitability}\eqref{def: suitability.2} does not hold, and since $\alpha_{\overline{e}}$ is surjective we know that Lemma \ref{def: suitability}\eqref{def: suitability.3} does not hold.  Therefore it must be Lemma \ref{def: suitability}\eqref{def: suitability.1} that holds. Let $\nu \in s(e) \GG^* s(e_i)$ be a reduced $\GG$-word such that $1e \nu 1\overline{e_i}$ is reduced.  Then $\nu = g_1 \nu_1 \dots g_m \nu_m g_{m+1}$ with $m \ge 1$ and $\nu_1 \not= \overline{e}$.  Write $\nu_1 \dots \nu_m = \eta^p e_1 \dots e_j \beta$ for some path $\beta \in \Gamma^*$, where the rangemost edge of $\beta$ is not $e_{j+1}$.  Then $|\beta| > 0$ since $e_j \overline{e_i}$ cannot be reduced. Since $\alpha_{\overline{e_j}}$ is surjective, the rangemost edge of $\beta$ is not $\overline{e_j}$.  Therefore $\beta$ is a path in the tree $E_{s,\eta,e_j}$.  Since $\GG$ is treelike at $\eta$, all edges of $\beta$ must point towards $s(e_j)$.  But then $s(\nu)$ does not lie on $\eta$, a contradiction.
	
	We now assume that (c) holds. We first claim that if $\beta=\beta_1\dots\beta_k\in r(e)\Gamma^*s(e)$ is a reduced path, then $\beta_k=e$. Suppose not for contradiction; say $\beta_1\dots\beta_k\in r(e)\Gamma^*s(e)$ is a reduced path with $\beta_k\not=e$. Then each $\overline{\beta_i}$ is an edge in $E_{s,e}$ because $\overline{\beta_k}\dots\overline{\beta_1}\in s(e)\Gamma^*$ with $\overline{\beta_k}\not=\overline{e}$. Now $\overline{\beta_k}\dots\overline{\beta_1}e$ is a cycle in $E_{s,e}$, which contradicts that $E_{s,e}$ is a tree. So the claim holds.
	
	We now claim that $\xi\in r(e)\partial W_\GG$ from (c) cannot flow to $e$, which we prove by contradiction.  Let $\xi = h_1 f_1 h_2 f_2 \dots$.  First suppose that $\xi$ can flow to $e$ in the sense of Lemma~\ref{def: suitability}\eqref{def: suitability.1}: there are $i$ and a reduced $\GG$-word $\nu=g_1\nu_1\dots g_n\nu_n g_{n+1}$ with $r(\nu) = s(e)$, $s(\nu) = r(f_i)$, $|\nu| \ge 1$, and such that $1e\nu 1f_i$ has no cancellation. Since $\alpha_{\overline{e}}$ is surjective we know that $\nu_1 \not= \overline{e}$. Consider $f_1\dots f_{i-1}\overline{\nu_n}\dots\overline{\nu_1}\in r(e)\Gamma^* s(e)$.  Since $e$ does not lie on $\xi$, the reduction of this path is a reduced path in $r(e)\Gamma^* s(e)$ whose sourcemost edge is not $e$. But this contradicts the previous claim, so we must have that $\xi$ cannot flow to $e$ in the sense of Lemma~\ref{def: suitability}\eqref{def: suitability.1}.     
	
	Now suppose that $\xi$ can flow to $e$ in the sense of Lemma~\ref{def: suitability}\eqref{def: suitability.2}; say, $r(f_i)=s(e)$ and $f_i\not=\overline{e}$. Then the reduction of $f_1\dots f_{i-1}\in r(e)\Gamma^* s(e)$ is a reduced path in $r(e)\Gamma^* s(e)$ whose sourcemost edge is not $e$. But this again contradicts the previous claim, so we must have that $\xi$ cannot flow to $e$ in the sense of Lemma~\ref{def: suitability}\eqref{def: suitability.2}. Finally, we know that $\xi$ cannot flow to $e$ in the sense of Lemma~\ref{def: suitability}\eqref{def: suitability.3}, because $\alpha_{\overline{e}}$ is surjective. Hence $\xi\in r(e)v\partial X_\GG$ cannot flow to $e$. This completes the proof of the ``if'' direction.
	
	We now prove the ``only if'' direction. Suppose $\xi\in\partial W_\GG$, $e\in\Gamma^1$ and $\xi$ cannot flow to $e$. Let $\gamma_1 \dots \gamma_k$ be a path in $\Gamma^*$ with $r(\gamma_1) = s(e)$ and $s(\gamma_k) = r(\xi)$.  Let $\xi'$ be the reduction of $1 \gamma_1 \dots 1 \gamma_k \xi$ to an infinite $\GG$-path.  Then $r(\xi') = s(e)$.  Moreover, since $\xi$ cannot flow to $e$, it follows that $\xi'$ cannot flow to $e$.  It then follows that $e$ does not lie on $\xi'$, and that $\xi' \in Z(1 \overline{e})$.  Therefore, replacing $\xi$ with $\xi'$ if necessary, we may assume in addition that $\xi \in Z(1 \overline{e})$.  We let $\xi = h_1 f_1 h_2 f_2 \dots$, so that $h_1 f_1 = 1 \overline{e}$.
	
	It now follows that $\alpha_{\overline{e}}$ is surjective.  Suppose for contradiction that there is $1 \not= g \in \Sigma_{\overline{e}}$.  Then $\nu = g$ satisfies Lemma \ref{def: suitability}\eqref{def: suitability.1} for $e$ and $\xi$, a contradiction.
	
	Now we must prove that (a), (b) or (c) holds.  We will treat (a) and (b) together.  Suppose that $e$ lies on a minimal cycle $\eta = e_1 \dots e_n \in \Gamma^*$; say $e = e_1$. We will interpret the indices on the edges of $\eta$ modulo $n$.  We first show that $e_i$ does not lie on $\xi$ for all $i$.  We already know this for $i=1$.  If $e_k$ lies on $\xi$ with $k > 1$, say $e_k = f_j$, then $\nu = 1 e_2 \dots 1 e_k h_{j+1}$ satisfies Lemma \ref{def: suitability}\eqref{def: suitability.1} for $e$ and $f_{j+1}$, a contradiction.
	
	We next show that $\alpha_{\overline{e_i}}$ is surjective for all $i$.  Suppose for contradiction that there is $1 \not= g \in \Sigma_{e_j}$ for some $j$.  Then $j > 1$.  Now $\nu = 1 e_2 \dots 1 e_j g \overline{e_j} \dots 1 \overline{e_2}$ satisfies Lemma \ref{def: suitability}\eqref{def: suitability.1} for $e$ and $\xi$, a contradiction.
	
	Next, we claim that the only reduced paths in $r(e) \Gamma^* r(e)$ are multiplies of $\eta$ and $\overline{\eta}$.  Suppose for contradiction that this is not the case.  Then there are $1 \le i \le j$ with $j - i < n$, and a reduced path $\beta_1 \dots \beta_\ell \in s(e_i) \Gamma^* r(e_j)$ with $\beta_1 \not= \overline{e_i}$, $e_{i+1}$ and $\beta_\ell \not= e_{j-1}$, $\overline{e_j}$.  Then $\nu = 1 e_2 \dots 1 e_i 1 \beta_1 \dots 1 \beta_\ell 1 \overline{e_{j-1}} \dots 1 \overline{e_2}$ satisfies Lemma \ref{def: suitability}\eqref{def: suitability.1} for $e$ and $\xi$, a contradiction.  This establishes the claim.  It follows that each $E_{s,\eta,e_i}$ (or $E_{s,e}$ if $\eta$ is a loop) is a tree.  For convenience, we let $E_{s,\eta,e_i}$ denote $E_{s,e}$ in the case that $\eta$ is a loop.  For vertices $x$ and $y$ in $E_{s,\eta,e_i}^0$ we will let $[x,y] = 1 \gamma_1 \dots 1 \gamma_p$, where $\gamma_1 \dots \gamma_p$ is the unique reduced path in $E_{s,\eta,e_i}^*$ with $r(\gamma_1) = x$ and $s(\gamma_p) = y$.
	
	To complete the proof of (a) or (b) we must show that for each $i$, and for every edge $f$ in $E_{s,\eta,e_i}$ that points towards $s(e_i)$, we have that $\alpha_{\overline{f}}$ is surjective.  Suppose for contradiction that $f \in E_{s,\eta,e_j}^1$ points towards $s(e_j)$ and that there is $1 \not= g \in \Sigma_{\overline{f}}$.  Then $\nu = 1 e_2 \dots 1 e_j [s(e_j),r(f)] 1 f g \overline{f} [r(f),s(e_j)] 1 \overline{e_{j-1}} \dots 1 \overline{e_2}$ satisfies Lemma \ref{def: suitability}\eqref{def: suitability.1} for $e$ and $\xi$, a contradiction.
	
	Now suppose that $e$ does not lie on any minimal cycle in $\Gamma^*$.  In this case $e$ is not an edge of $E_{s,e}$.  We show that $E_{s,e}$ is a tree.  Suppose for contradiction that $E_{s,e}$ contains a reduced cycle $\beta_1 \dots \beta_m$.  Let $\delta_1 \dots \delta_p$ be a reduced path in $E_{s,e}^*$ with $r(\delta_1) = s(e)$, $s(\delta_p) = r(\beta_k)$ for some $k$, and such that $\delta_i \not= \beta_j$ for all $i$ and $j$.  Relabeling, we may assume $k = 1$.  Then $\nu = 1 \delta_1 \dots 1 \delta_p 1 \beta_1 \dots 1 \beta_m 1 \overline{\delta_p} \dots 1 \overline{\delta_1}$ satisfies Lemma \ref{def: suitability}\eqref{def: suitability.1} for $e$ and $\xi$, a contradiction.
	
	Finally, let $f$ be an edge of $E_{s,e}$ pointing towards $s(e)$.  We use the same notation as in cases (a) and (b) for paths in the tree $E_{s,e}$.  We show that $\alpha_{\overline{f}}$ is surjective.  Suppose for contradiction that there is $1 \not= g \in \Sigma_{\overline{f}}$.  Then $\nu = [s(e),r(f)] 1 f g \overline{f} [r(f),s(e)]$ satisfies Lemma \ref{def: suitability}\eqref{def: suitability.1} for $e$ and $\xi$, a contradiction. 
\end{proof}

\begin{rmk}
	
	The last condition (c) in the statement of Theorem \ref{thm: nonminimality} can be expressed without mentioning infinite paths.  Namely, there exists $\xi \in r(e) \partial W_\GG$ not containing $e$ if and only if $E_{s,\overline{e}}$ either is infinite, contains a cycle, or contains a path $\gamma$ such that $\alpha_\gamma$ and $\alpha_{\overline{\gamma}}$ are both not surjective (where $\alpha_\gamma = \alpha_e$ with $e$ the rangemost edge of $\gamma$).
	
\end{rmk}

The characterisation of nonminimality in Theorem \ref{thm: nonminimality} can be simplified.  

\begin{thm} \label{thm: nonminimality second take}
	
	The action of $\pi_1(\GG,v)$ on $v\partial X_\GG$ is \emph{not} minimal if and only if one of the following holds:
	\begin{enumerate}
		\item[(c')] \label{thm: treelike_three'} There is an edge $e \in \Gamma^1$ such that $\alpha_{\overline{e}}$ is surjective, $e$ does not lie on any minimal cycle of $\Gamma$, $\GG$ is treelike at $e$, and there exists $\xi \in r(e)\partial W_\GG$ such that $e$ does not lie on $\xi$; or
		\item[(d)] \label{thm: treelike_four} $\Gamma$ is a minimal cycle $e_1 e_2 \cdots e_n$ such that $\alpha_{\overline{e_i}}$ is surjective for all $i$. 
	\end{enumerate}

\end{thm}

\begin{proof}
	We show that the hypotheses of the theorem are equivalent to those of Theorem \ref{thm: nonminimality}.  If (c') holds, then $\alpha_{\overline{e}}$ is surjective and (c) holds.  If (d) holds then either (a) or (b) holds (and $\alpha_{\overline{e}}$ is surjective, in case (a)).
	
	Conversely, if either (a) or (b) holds, and either $E_{s,e}$, or $E_{s,\eta,e_i}$ for some $i$, is nontrivial, then any edge in $E_{s,e}$, respectively $E_{s,\eta,e_i}$, satisfies (c').  If $E_{s,e}$, respectively $E_{s,\eta,e_i}$ for all $i$, is trivial, then we are in case (d).  If (c) holds, then since $\alpha_{\overline{e}}$ is assumed surjective we have that (c') holds.
\end{proof}

If the action is minimal, then we can say more about $\GG$ in the presence of nontrivial treelike behaviour.

\begin{prop} \label{prop: minimal and treelike}
	
	Suppose that the action of $\pi_1(\GG,v)$ on $v\partial X_\GG$ is minimal, there is $e \in \Gamma^1$ such that $\GG$ is treelike at $e$, and $E_{s,e}$ is nontrivial. Then $\Gamma$ is an infinite ray $e_1 e_2 \cdots$, and $\alpha_{\overline{e_i}}$ is surjective for all $i$.
	
\end{prop}

To prove this result we use the following lemma.

\begin{lem} \label{lem: minimal and treelike}
	
	Suppose that the action of $\pi_1(\GG,v)$ on $v\partial X_\GG$ is minimal, there is $e \in \Gamma^1$ such that $\GG$ is treelike at $e$, and $E_{s,e}$ is nontrivial. Then the following hold.
	
	\begin{enumerate}[(1)]
		
		\item $e$ does not lie on any cycle of $\Gamma$.
		
		\item $\alpha_{\overline{e}}$ is surjective.
		
		\item $|\Gamma^1 r(e)| \le 2$.
		
		\item If $\Gamma^1 r(e) = \{ \overline{e}, e' \}$ with $e' \not= \overline{e}$, then $\GG$ is treelike at $e'$.
		
	\end{enumerate}
	
\end{lem}

\begin{proof}
	Let $f \in s(e) E_{s,e}^1$.
	
For (1), note that if $e$ lies on a cycle $\gamma$, then $\xi = \gamma^\infty$ and any edge of $E_{s,e}^1$ pointing towards $s(e)$ satisfies Theorem \ref{thm: nonminimality second take}(c'), contradicting minimality.
	
For (2), suppose for contradiction that $\alpha_{\overline{e}}$ is not surjective.  Let $1 \not= g \in \Sigma_{\overline{e}}$.  We claim that $\alpha_e$ is surjective.  For suppose not.  Let $1 \not= h \in \Sigma_e$.  Put $\xi = 1 \overline{e}(h e g \overline{e})^\infty \in r(e) \partial W_\GG$.  Then $\xi$ and $f$ satisfy Theorem \ref{thm: nonminimality second take}(c'), contradicting minimality.  Therefore $\alpha_e$ is surjective.  Now, nonsingularity of $\GG$ implies that there is $e_2 \in \Gamma^1 r(e)$ with $e_2 \not= \overline{e}$.  We claim that $\alpha_{e_2}$ is surjective.  For if not, let $1 \not= h \in \Sigma_{e_2}$.  Let $\xi_2 = 1 \overline{e} (1 \overline{e_2} h e_2 1 e g \overline{e})^\infty$.  Then $\xi_2$ and $f$ satisfy Theorem \ref{thm: nonminimality second take}(c'), contradicting minimality.  Repeat this process to obtain $e_3$, $e_4$, $\ldots$, and put $\xi_\infty = 1 \overline{e} 1 \overline{e_2} 1 \overline{e_3} \cdots \in r(e) \partial W_\GG$.  Then $\xi_\infty$ and $f$ satisfy Theorem \ref{thm: nonminimality second take}(c'), contradicting minimality.  Hence we must have $\alpha_{\overline{e}}$ surjective.
	
For (3), suppose for contradiction that $|\Gamma^1 r(e)| > 2$.  Then there are $f_1 \not= f_2 \in \Gamma^1 r(e) \setminus \{ \overline{e} \}$.  If either $\overline{f_1}$ or $\overline{f_2}$ is the rangemost edge of an infinite reduced path in $\Gamma$, say $\overline{f_1} d_1 d_2 \cdots \in \Gamma^\infty$, then $\xi = 1 \overline{e} 1 \overline{f_1} 1 d_1 1 d_2 \cdots$ and $f$ satisfy Theorem \ref{thm: nonminimality second take}(c'), contradicting minimality.  Therefore neither $\overline{f_1}$ nor $\overline{f_2}$ has this property.  It follows that $E_{s,\overline{f_1}}$ and $E_{s,\overline{f_2}}$ are finite trees.  We will write $[x,y]$ for the unique reduced path between vertices $x$ and $y$ in one of these trees.  For $i = 1$, 2 choose $w_i \in E_{s,\overline{f_i}}^0$ such that $|w_i E_{s,\overline{f_i}}| = 1$.  Let $f_i'$ be the sourcemost edge of $[r(e),w_i]$. By nonsingularity of $\GG$ we must have that $\alpha_{\overline{f_i'}}$ is not surjective.  Let $1 \not= g_i \in \Sigma_{\overline{f_i'}}$.  Put $\xi' = 1 \overline{e} ([r(e),w_1] g_1 [w_1,r(e)] [r(e),w_2] g_2 [w_2,r(e)])^\infty$.  Then $\xi'$ and $f$ satisfy Theorem \ref{thm: nonminimality second take}(c'), contradicting minimality. So $|\Gamma^1 r(e)| \le 2$.
	
Finally, for (4), let $\Gamma^1 r(e) = \{ \overline{e},e' \}$ with $e' \not= \overline{e}$.  By (3), and the definition of $E_{s,e'}$, we have that $E_{s,e'}^1 = E_{s,e}^1 \cup \{ e, \overline{e} \}$.  Then by (2) we have that $\GG$ is treelike at $e'$.
\end{proof}

\begin{proof}[Proof of Proposition~\ref{prop: minimal and treelike}]
	If $|\Gamma^1 r(e)| = 2$, let $\overline{e} \not= f_1 \in \Gamma^1 r(e)$.  By Lemma \ref{lem: minimal and treelike} it follows that $f_1$ also satisfies the hypotheses of Lemma \ref{lem: minimal and treelike}.  Repeating this argument with $f_1$, we have either that $\Gamma^1 r(f_1) = \{ \overline{f_1} \}$, or that $\Gamma^1 r(f_1) = \{ \overline{f_1}, f_2 \}$ with $f_2 \not= \overline{f_1}$, and that $f_2$ satisfies the hypotheses of Lemma \ref{lem: minimal and treelike}.  If this process repeats indefinitely, let $\xi = 1 \overline{e} 1 \overline{f_1} 1 \overline{f_2} \cdots$, and let $f$ be as in the proof of Lemma \ref{lem: minimal and treelike}.  Then $\xi$ and $f$ satisfy Theorem \ref{thm: nonminimality second take}(c'), contradicting minimality.  Therefore there is $n \ge 0$ such that $\Gamma^1 r(f_n) = \{ \overline{f_n} \}$.  By Lemma \ref{lem: minimal and treelike} we know that $\GG$ is treelike at $f_n$.  To finish the proof, we show that $\Gamma$ is an infinite ray $f_n f_{n-1} \cdots f_1 e \cdots$.  For this, it suffices to show that $\Gamma$ has no branching.  Suppose to the contrary that $d_1$, $d_2 \in \Gamma^1$ both point toward $r(f_n)$, and satisfy $r(d_1) = r(d_2)$.  Let $\xi \in Z(1 f_n 1 f_{n-1} \cdots 1 d_1)$.  Then $\xi$ and $d_2$ satisfy Theorem \ref{thm: nonminimality}(c'), contradicting minimality.
\end{proof}

\subsection{Local contractivity}\label{subsec: loc contractivity}

Recall that the action of a discrete group $\Lambda$ on a locally compact Hausdorff space $Y$ is \textit{locally contractive} (called {\em locally contracting} in \cite{A2}, and \textit{local boundary action} in \cite{LS}) if for every nonempty open set $U \subseteq Y$ there is an open set $V \subseteq U$ and $\lambda \in \Lambda$ such that $\lambda \overline{V} \subsetneq V$.  In this section, we give a sufficient condition for the action of $\pi_1(\GG,v)$ on $v\partial X_\GG$ to be locally contractive.  We begin with some definitions.

\begin{dfn}
	We say a $\GG$-path $g_1 e_1 \dots g_n e_n$ is {\em repeatable} if $r(e_1) = s(e_n)$ and $g_1 e_1 \not= 1 \overline{e_n}$. 
\end{dfn}

\noindent
So a repeatable $\GG$-path is a reduced $\GG$-loop which ends with an edge and is such that concatenation with itself has no cancellation.

\begin{dfn}
	We say that a reduced $\GG$-loop of the form $\eta = g_1 f_1 \dots g_m f_m$ has an \textit{entrance at $s(f_i)$} if
	\[
	\sum_{r(f) = s(f_i)} |\Sigma_f| \ge 3.
	\]
\end{dfn}

\noindent 
Note that this means that there exists $f \in \G^1$ with $r(f) = s(f_i)$ and $h \in \Sigma_f$ such that $h f \not= g_{i+1} f_{i+1}$, $1 \overline{f_i}$.

\begin{thm}\label{thm: loc cont}
	
	The action of $\pi_1(\GG,v)$ on $v\partial X_\GG$ is locally contractive if for every edge $e \in \Gamma^1$ there is a repeatable $\GG$-path $\eta$ with an entrance, so that $\eta$ contains an edge that can flow to $e$.
	
\end{thm}

The idea of the proof is as follows.  A repeatable $\GG$-path $\eta$ defines both an element of $\pi_1(\GG,v)$ and a cylinder set $Z(\eta)$.  Moreover, it is clear that $\eta Z(\eta) \subseteq Z(\eta)$.  If in addition $\eta$ has an entrance, then it can be shown that the containment is proper.  The trick in the proof is to access this phenomenon inside an arbitrary cylinder set.  (This technique is fairly standard in $C^*$-theory; see, for example, \cite{KPR}, although that paper was written with the pre-Australian convention on paths in a graph.)

\begin{proof}
	Let $U$ be a nonempty open subset of $v\partial X_\GG$. Let $\mu$ be a $\GG$-path such that $Z(\mu) \subseteq U$, and let $e$ be the source-most edge of $\mu$.  Let $\eta = g_1 f_1 \dots g_m f_m$ be a repeatable $\GG$-path as in the statement.  Without loss of generality assume that $f_1$ can flow to $e$, and that $\eta$ has an entrance at $s(f_i)$. Let $f \in \Gamma^1$ with $r(f) = s(f_i)$ and $h \in \Sigma_f$ be such that $h f \not= g_{i+1} f_{i+1}$, $1 \overline{f_i}$. 
	
	Suppose first that (1) of Lemma~\ref{def: suitability} holds. Then there is a reduced $\GG$-word $\nu$ such that $r(\nu) = s(e)$, $s(\nu) = r(f_1)$, $|\nu| \ge 1$, and the concatenation $1e \nu 1f_1$ has no cancellation. Upon replacing $\nu$ by $\nu g$ for an appropriate choice of $g$, we may assume that the concatenation $\mu \nu \eta$ has no cancellation. Let $\gamma = \mu \nu \eta \nu^{-1} \mu^{-1} \in \pi_1(\GG)$. Let $V = Z(\mu \nu \eta)$, where by $\mu \nu \eta$ we mean the reduced form of the concatenation of $\mu$, $\nu$ and $\eta$, and with the last group element removed. Then $V$ is a compact-open subset of $U$.  We have that $\gamma V = Z(\mu \nu \eta \eta) \subseteq V$.  We show that the containment is proper.  Let $\eta_1 = g_1 f_1 \dots g_{i-1} f_{i-1}$, $\eta_2 = g_i f_i \dots g_k f_k$ (and if $i = 1$ we set $\eta_1 = 1_{s(\eta)}$ and $\eta_2 = \eta$).  Then $Z(\mu \nu \eta \eta_1 hf) \subseteq V$, since the concatenation $\eta \eta_1 hf$ has no cancellation.  Moreover, since $\mu \nu \eta \eta_1 hf$ and $\mu \nu \eta \eta$ differ, $\gamma V \cap Z(\mu \nu \eta \eta_1 hf) = \emptyset$.  Hence $\gamma V \not= V$.
	
	If $f_1$ can flow to $e$ in the sense of Lemma~\ref{def: suitability}\eqref{def: suitability.2}, then we apply the argument above to $\nu=1_{s(e)}$. If $f_1$ can flow to $e$ in the sense of Lemma~\ref{def: suitability}\eqref{def: suitability.3}, then we choose $1\not=g\in\Sigma_{f_1}$ and we apply the argument above to $\nu=g$.
\end{proof}

There is a strong connection between minimality and local contractivity. 

\begin{prop}\label{prop: min and lc}
	Let $\GG=(G,\Gamma)$ be a locally finite nonsingular graph of countable groups. If the action of $\pi_1(\GG)$ on $v\partial X_\GG$ is minimal, then exactly one of the following holds:
	\begin{enumerate}
		\item[(1)] the action is locally contractive;
		\item[(2)] $\Gamma$ is an infinite ray $e_1e_2e_3\dots$ and each $\alpha_{\overline{e_i}}$ is surjective; or
		\item[(3)] $\Gamma$ is a finite ray $e_1e_2\dots e_n$, $|\Sigma_{e_1}|=|\Sigma_{\overline{e_n}}|=2$ and $|\Sigma_{f}|=1$ for all $f\in\Gamma^1\setminus\{e_1,\overline{e_n}\}$. 
	\end{enumerate}	
\end{prop}

We already know from Proposition~\ref{prop: minimal and treelike} that if the action is minimal and there is an edge $e$ such that $\GG$ is treelike at $e$ with $E_{s,e}$ a nontrivial tree, then $\GG$ must be the infinite ray from (2) above. To prove Proposition~\ref{prop: min and lc} we separate out the graphs of groups with no nontrivial treelike behaviour at any edge. (Note that minimality is not assumed in the next result.)

\begin{lem}\label{lem: no treelike bits}
	Let $\GG=(G,\Gamma)$ be a locally finite nonsingular graph of countable groups. If there is no edge $e\in\Gamma^1$ with $\GG$ treelike at $e$ and $E_{s,e}$ a nontrivial tree, then exactly one of the following holds: 
	\begin{enumerate}
		\item[(1)] the action is locally contractive;
		\item[(2')] $\Gamma$ is a minimal cycle and $\alpha_e$ is surjective for all edges $e$; or
		\item[(3)] $\Gamma$ is a finite ray $e_1e_2\dots e_n$, $|\Sigma_{e_1}|=|\Sigma_{\overline{e_n}}|=2$ and $|\Sigma_{f}|=1$ for all $f\in\Gamma^1\setminus\{e_1,\overline{e_n}\}$. 
	\end{enumerate}	
\end{lem}

\begin{proof}

We first note that the boundary in both cases (2') and (3) only consists of two points, and hence the action is not locally contractive. So (1), (2') and (3) are mutually exclusive.

We first assume that $\GG$ contains a repeatable $\GG$-path $\eta=h_1f_1\dots h_nf_n$ with an entrance. Fix an edge $e\in\Gamma^1$.  We claim that there is a repeatable $\GG$-path with an entrance that flows to $e$. It is obvious that $\eta$ flows to $e$ in the case that $s(e)=s(f_j)$ for some $j$, and $e\not= \overline{f_{j+1}}$. Now suppose $e=\overline{f_j}$ for some $j$. Consider the path $\eta':=h_n'\overline{f_n}h_{n-1}'\overline{f_{n-1}}\dots h_1'\overline{f_1}$, where $h_i':=h_{i+1}$ (index mod n) if $f_{i+1}=\overline{f_i}$, or $h_i':=1$ otherwise. Then $\eta'$ is a repeatable $\GG$-path with an entrance which obviously flows to $e=\overline{f_j}$. The last case to consider on $e$ is when $s(e)$ is not a vertex on $\eta$. Choose $j$ and a reduced $\GG$-word $\nu = g_1e_1\dots g_me_m\in s(e)\GG^*r(f_j)$ such that $\nu 1f_j$ has no cancellation. If $g_1e_1\not=1\overline{e}$, then $\eta$ flows to $e$ via $\nu$. If $g_1e_1=1\overline{e}$ and $|\Sigma_{\overline{e}}|\ge 2$, then choose $1\not= g\in\Sigma_{\overline{e}}$ and use $g\nu$ to see that $\eta$ flows to $e$. If $|\Sigma_{\overline{e}}|=1$, then the nonsingularity of $\GG$ means that there is an edge $f\in\Gamma^1\setminus\{\overline{e}\}$ with $r(f)=s(e)$. Since $\GG$ is not treelike at $f$ with $E_{s,f}$ nontrivial, we can choose a reduced $\GG$-word $\beta$ with range and source $s(f)$. (Note that $\beta$ can be chosen to be a nontrivial element in $\Sigma_{\overline{f}}$ if $\GG$ is treelike at $f$ and $E_{s,f}=\{s(f)\}$.) We see that $\eta$ flows to $e$ via $1f\beta 1\overline{f}\nu$. This completes the proof of the claim. We can now use Theorem~\ref{thm: loc cont} to see that the action is locally contractive.

Now suppose that $\GG$ does not contain a repeatable $\GG$-path with an entrance. We first suppose $\Gamma$ contains a minimal cycle $e_1\dots e_n$. Then $1e_1\dots 1e_n$ is a repeatable $\GG$-path. Since this $\GG$-path cannot have an entrance, we must have $\Gamma^1=\{e_i,\overline{e_i}:1\le i\le n\}$, and each $\alpha_{e_i}$ and $\alpha_{\overline{e_i}}$ surjective. So we are in case (2'). Now suppose that $\Gamma$ does not contain any minimal cycles; that is, $\Gamma$ is a tree. Because of the assumptions on $\GG$, there must exist a ray $f_1\dots f_m$ in $\Gamma$ with $|\Sigma_{f_1}|,|\Sigma_{\overline{f_m}}|\ge 2$. With $1\not=g\in\Sigma_{f_1}$ and $1\not=h\in \Sigma_{\overline{f_m}}$, the $\GG$-path $gf_11f_2\dots 1f_mh\overline{f_m}1\overline{f_{m-1}}\dots1\overline{f_1}$ is repeatable. Since this $\GG$-path cannot have an entrance, we must have $\Gamma^1=\{f_i,\overline{f_i}:1\le i\le m\}$, $|\Sigma_{f_1}|,|\Sigma_{\overline{f_m}}|= 2$, and $\alpha_f$ surjective for all $f\in\Gamma^1\setminus\{e_1,\overline{e_m}\}$. So we are in case (3). 
\end{proof}

\begin{proof}[Proof of Proposition~\ref{prop: min and lc}]
We know from the proof of Lemma~\ref{lem: no treelike bits} that the finite ray in (3) is not locally contractive. In the case of (2) we take $v$ to be the range of the infinite ray, and then since each $\alpha_{\overline{e_i}}$ is surjective, the fundamental group can be identified with the vertex group $G_v$. The action of this vertex group does not map cylinder sets properly inside themselves, and hence the action is not locally contractive. So (1), (2) and (3) are mutually exclusive.  
	
We assume that the action is minimal. We know immediately from Theorem~\ref{thm: nonminimality second take} that case (2') of Lemma~\ref{lem: no treelike bits} cannot hold. We know from Proposition~\ref{prop: minimal and treelike} that if $\GG$ is treelike at an edge $e$ with $E_{s,e}$ nontrivial, then (2) holds. We know from Lemma~\ref{lem: no treelike bits} that if there is no such edge, then (1) or (3) must hold.
\end{proof}

\subsection{Topological freeness}\label{subsec: top freeness}

We recall the following definition from \cite{AS}. 

\begin{dfn}
	The action of a discrete group $\Lambda$ on a compact Hausdorff space $Y$ is \textit{topologically free} if for each $t \in \Lambda \setminus \{ 1 \}$ we have that $\{ y \in Y : t y \not= y \}$ is dense in $Y$.
\end{dfn}

The importance of topological freeness can be seen from the Corollary to Theorem 2 in \cite{AS}:  the reduced $C^*$-algebra $C(Y) \rtimes_r \Lambda$ is simple if and only if the action of $\Lambda$ on $Y$ is minimal and topologically free.  

It follows that if $\Lambda$ is countable then the action is topologically free if the set of points in $Y$ having trivial isotropy is dense in $Y$.  It seems to be a difficult problem to characterise graphs of groups whose fundamental group acts topologically freely on the boundary of the Bass--Serre tree.  We are able to give definitive results in two special cases, that of graphs of groups with trivial vertex groups (Section~\ref{sec:undirected graphs}) and {\em generalised Baumslag--Solitar groups} (Section~\ref{sec: GBS}). 

We now describe a family of actions called odometers; these actions are {\em free}, in the sense that every point has trivial isotropy. After our discussion on odometers, we present an example to illustrate that topological freeness is independent of minimality.

\begin{example} \label{example: subodometers}
	
	Let $G_0$ be a discrete group, and let $G_0 \supseteq G_1 \supseteq G_2 \supseteq \cdots$ be a decreasing sequence of subgroups of finite index.  We insist that $[G_0 : G_1] > 1$, but allow $[G_{i-1} : G_i] = 1$ for $i > 1$.  We consider the graph $\Gamma$ that is an infinite ray: $\Gamma^1 = \{ e_i, \overline{e_i} : i \ge 1 \}$ with $s(e_i) = r(e_{i+1})$ for all $i$, and $r(e_i) \not= r(e_j)$ when $i \not= j$. For $i \ge 1$ we set $G_{s(e_i)} = G_{e_i} := G_i$, and $G_{r(e_1)} := G_0$.  We let $\alpha_{e_i}$ be the inclusion of $G_i$ into $G_{i-1}$, and $\alpha_{\overline{e_i}}$ be the identity map.
The boundary $v\partial X_\GG$ can then be realised as the inverse limit $$\underleftarrow{\lim}\; G_0 / G_i = \{ (x_i  G_i)_{i=0}^\infty : x_i \in G_0, x_i^{-1} x_{i+1} \in G_i, i \ge 0 \}.$$  

Let $v = r(e_1)$. Since $\alpha_{\overline{e_i}}$ is surjective for all $i$, the only reduced $\GG$-words based at $v$ are those of length zero.  Thus $\pi_1(\GG,v) = G_0$.  The action of $G_0$ on $v\partial X_\GG$ is given by $t \cdot (x_i G_i)_{i=0}^\infty = (tx_i G_i)_{i=0}^\infty$.  Actions of this type are called \textit{subodometers} (see \cite{CP}).  We will refer to such graphs of groups as \textit{subodometer graphs of groups}.  There is a map of $G_0 \to v\partial X_{\GG}$ having dense range, given by $t \in G_0 \mapsto (tG_i)_{i=0}^\infty$.  Thus the action of $G_0$ is effective if $\bigcap_{i=0}^\infty G_i = \{ 1 \}$ (i.e. if $G_0$ is residually finite; see \cite{CP} Definition 1).  We give a simple characterisation of effective subodometer graphs of groups.
	
	\begin{prop} \label{prop: effective subodometer}
		
		Let $G_0 \supseteq G_1 \supseteq \cdots$ define a subodometer graph of groups $\GG$.  Let $H = \bigcap_{i=0}^\infty G_i$.  Then $\GG$ is effective if and only if $\bigcap_{x \in G_0} x H x^{-1} = \{1\}$.  (This last intersection is called the {\em normal core} of $H$.)
		
	\end{prop}
	
	\begin{proof}
		We first prove the \textit{if} direction.  Let $t \in G \setminus \{1\}$.  If $t \not\in H$ then  $t \cdot (G_i)_{i=0}^\infty = (t G_i)_{i=0}^\infty \not= (G_i)_{i=0}^\infty$.  Hence $t$ does not act as the identity.  If $t \in H$ then by hypothesis there is $x \in G_0$ such that $x^{-1} t x \not\in H$.  Then there is $j$ such that $x^{-1} t x \not\in G_j$, i.e. such that $t x G_j \not= x G_j$.  Then $t \cdot (x G_i)_{i=0}^\infty \not= (x G_i)_{i=0}^\infty$, and hence $t$ does not act as the identity.  For the \textit{only if} direction, suppose that the normal core of $H$ is nontrivial.  Let $1 \not= t \in \bigcap_{x \in G_0} x H x^{-1}$.  Then $tx \in x H \subseteq x G_i$ for all $x \in G_0$ and all $i$.  But then $t x G_i = x G_i$ for all $x$ and $i$.  It follows that $t$ acts trivially on $v\partial X_\GG$.
	\end{proof}

\begin{rmk}\label{rmk: normal and odometer}
We also note that by Theorem \ref{thm: nonminimality second take} the action of $G_0$ on $v\partial X_\GG$ is minimal (see also \cite{CP} section 3.1).  If the $G_i$ are normal subgroups of $G_0$, the dynamical system $(v\partial X_\GG, G_0)$ is called an {\em odometer} (and $\GG$ is an {\em odometer graph of groups}). In this case, the map $G_0\to v\partial X_{\GG}$ from Example~\ref{example: subodometers} is a homomorphism. It follows that if the $G_i$ are normal subgroups then the action is topologically free if and only if the action is free, if and only if $G_0\to v\partial X_{\GG}$ is injective, if and only if $\bigcap_{i=0}^\infty G_i = \{1\}$.
\end{rmk}

	\begin{prop} \label{prop: subodometers are tracial}
		
		Let $\GG$ be a subodometer graph of groups.  Then $C(v\partial X_\GG) \rtimes_r \pi_1(\GG,v)$ has a faithful tracial state.
		
	\end{prop}
	
	\begin{proof}
		For each $i$ we denote normalised counting measure on $G_0 / G_i$ by $\mu_i$.  The quotient maps $q_i : G_0 / G_{i+1} \to G_0 / G_i$  satisfy ${q_i}_*(\mu_{i+1}) = \mu_i$ (where ${q_i}_*$ is the push-forward map on measures:  ${q_i}_*(\nu) = \nu \circ q_i^{-1}$).  Then $\mu = \varprojlim \mu_i$ is a probability measure on $v\partial X_\GG$ with full support and invariant for the action of $G_0$, and hence integration against $\mu$ is a faithful state on $C(v\partial X_\GG)$.  We compose with the faithful conditional expectation $E : C(v\partial X_\GG) \rtimes_r G_0 \to C(v\partial X_\GG)$ to obtain a faithful tracial state $\tau = \int E(\cdot) \, d\mu$.
	\end{proof}
	
	\begin{cor} \label{cor: amenable subodometer}
		
		Let $\GG$ be a subodometer graph of groups.  Then $G_0$ acts amenably on $v\partial X_\GG$ if and only if $G_0$ is amenable.
		
	\end{cor}
	
	\begin{proof}
		If $G_0$ is amenable then any action of $G_0$ is amenable (\cite{Ren}, II.3.10).  Conversely, suppose that the action is amenable.  Since the action admits a finite invariant measure, by the proof of Proposition \ref{prop: subodometers are tracial}, Proposition 4.3.3 of  \cite{Z} implies that $G_0$ is amenable.
	\end{proof}
	
	\begin{cor} \label{cor: subodometers stably finite}
		
		Let $\GG$ be a subodometer graph of groups with $\pi_1(\GG,v) \ (= G_0)$ amenable.  Then $C^*(\GG)$ is stably finite.
		
	\end{cor}
	
	Finally, we consider the special case of odometers where $G_0 \cong \Z$ (and hence $G_i \cong \Z$ for all $i$); these are the classic odometers (see \cite{D}, VIII.4, although we allow $[G_i : G_{i+1}] = 1$).  It is clear that $\bigcap_{i=0}^\infty G_i$ is trivial if and only if $[G_i : G_{i+1}] > 1$ for infinitely many $i$, and in this case the action is free (and minimal).  Moreover since $\Z$ is an amenable group, the action on the boundary is amenable.  In fact, the $C^*$-algebras obtained from such graphs of groups are well-known.
	
	\begin{prop} \label{prop: classic odometer}
		
		Let $\GG$ be an odometer graph of groups with $G_i \cong \Z$ for all $i$ and $[G_i : G_{i+1}] > 1$ for infinitely many $i$.  Then $C(v\partial X_\GG) \rtimes \pi_1(\GG,v)$ is a Bunce--Deddens algebra.
		
	\end{prop}
	
	\begin{proof}
		This follows from \cite{D}, section VII.4, after we observe that the $C^*$-algebras are unchanged if we delete edges $e_i$ of $\Gamma$ for which $\alpha_{e_i}$ is surjective.
	\end{proof}
	
	We remark also that the Bunce--Deddens algebras are simple, nuclear, stably finite, tracial $C^*$-algebras.
\end{example}

\begin{example} \label{example: not top free}
	
	We give an example of an effective graph of groups whose associated action on the boundary of its Bass--Serre tree is minimal, but not topologically free. We first present the tree, $X$, and an action of a group, $\Lambda$.  Let
	\begin{align*}
		X^0 &= \{ v \} \cup \{ x_{i,m} : i \in \Z / 2\Z,\ m \in \bigcup_{n \ge 0} \Z / 2^{n+1} \Z \} \\
		X^1 &= \{ e_{i,m} : i \in \Z / 2\Z,\ m \in \bigcup_{n \ge 0} \Z / 2^{n+1} \Z \} \\
		s(e_{i,m}) &= x_{i,m} \\
		r(e_{i,m}) &= 
		\begin{cases}
			x_{i, m + 2^n \Z}, &\text{ if } m \in \Z / 2^{n+1}\Z \\
			v, &\text{ if } m = \Z.
		\end{cases}
	\end{align*}
	The portion of the tree with $n \le 2$ is pictured here:
	
	\[
	\begin{tikzpicture}[xscale=2][yscale=.75]
	
	\node (v) at (0,0) [circle] {$v$};
	\node (u10) at (1,0) [circle] {$x_{0,\Z}$};
	\node (u22) at (2,2) [circle] {$x_{0,2\Z}$};
	\node (u2m2) at (2,-2) [circle] {$x_{0,1+2\Z}$};
	\node (u33) at (3,3) [circle] {$x_{0,4\Z}$};
	\node (u31) at (3,1) [circle] {$x_{0,2+4\Z}$};
	
	\node (u3m1) at (3,-1) [circle] {$x_{0,1+4\Z}$};
	\node (u3m3) at (3,-3) [circle] {$x_{0,3+4\Z}$};
	\node (um10) at (-1,0) [circle] {$x_{1,\Z}$};
	\node (um22) at (-2,2) [circle] {$x_{1,2\Z}$};
	\node (um2m2) at (-2,-2) [circle] {$x_{1,1+2\Z}$};
	\node (um33) at (-3,3) [circle] {$x_{1,4\Z}$};
	\node (um31) at (-3,1) [circle] {$x_{1,2+4\Z}$};
	\node (um3m1) at (-3,-1) [circle] {$x_{1,1+4\Z}$};
	\node (um3m3) at (-3,-3) [circle] {$x_{1,3+4\Z}$};
	
	\draw[-latex,thick] (u10) -- (v) node[pos=0.5, inner sep=0.5pt, anchor=south] {$e_{0,\Z}$};
	\draw[-latex,thick] (u22) -- (u10) node[pos=0.5, inner sep=0.5pt, anchor=south east] {$e_{0,2\Z}$};
	\draw[-latex,thick] (u2m2) -- (u10) node[pos=0.5, inner sep=0.5pt, anchor=north east] {$e_{0,1+2\Z}$};
	\draw[-latex,thick] (u33) -- (u22) node[pos=0.5, inner sep=0.5pt, anchor=south east] {$e_{0,4\Z}$};
	\draw[-latex,thick] (u31) -- (u22) node[pos=0.5, inner sep=0.5pt, anchor=south west] {$e_{0,2+4\Z}$};
	\draw[-latex,thick] (u3m1) -- (u2m2) node[pos=0.5, inner sep=0.5pt, anchor=south east] {$e_{0,1+4\Z}$};
	\draw[-latex,thick] (u3m3) -- (u2m2) node[pos=0.5, inner sep=0.5pt, anchor=south west] {$e_{0,3+4\Z}$};
	
	\draw[-latex,thick] (um10) -- (v) node[pos=0.5, inner sep=0.5pt, anchor=south] {$e_{1,\Z}$};
	\draw[-latex,thick] (um22) -- (um10) node[pos=0.5, inner sep=0.5pt, anchor=south west] {$e_{1,2\Z}$};
	\draw[-latex,thick] (um2m2) -- (um10) node[pos=0.5, inner sep=0.5pt, anchor=north west] {$e_{1,1+2\Z}$};
	\draw[-latex,thick] (um33) -- (um22) node[pos=0.5, inner sep=0.5pt, anchor=south west] {$e_{1,4\Z}$};
	\draw[-latex,thick] (um31) -- (um22) node[pos=0.5, inner sep=0.5pt, anchor=south east] {$e_{1,2+4\Z}$};
	\draw[-latex,thick] (um3m1) -- (um2m2) node[pos=0.5, inner sep=0.5pt, anchor=south west] {$e_{1,1+4\Z}$};
	\draw[-latex,thick] (um3m3) -- (um2m2) node[pos=0.5, inner sep=0.5pt, anchor=south east] {$e_{1,3+4\Z}$};
	
	\end{tikzpicture}
	\]
	
	Define automorphisms $a$ and $b$ of $X$ by $a v = v$, $a x_{i,m} = x_{i+1,m}$, $b v = v$, $b x_{0,m} = x_{0,m + 1}$, and $b x_{1,m} = x_{1,m}$.  We let $\Lambda$ be the subgroup of Aut$(X)$ generated by $a$ and $b$.  Let $X_i$, $i \in \Z / 2\Z$ be the subtree with vertices $\{ x_{i,m} : m \text{ arbitrary} \}$.  It is clear that $\partial X = \partial X_0 \sqcup \partial X_1$, that $b$ acts trivially on $X_1$, that $a$ interchanges $X_0$ and $X_1$, and that $aba$ acts trivially on $X_0$.  We may identify a finite path $e_{0,\Z} e_{0,m_1} e_{0,m_2} \cdots$ in $X_0$ with the nested sequence $\Z \supseteq m_1 \supseteq m_2 \supseteq \cdots$ of cosets.  Then $\partial X_0$ is identified with the $2$-adic integers $\Z_2$.  It is clear that $b$ acts on $\partial X_0$ as $\cdot + 1$; this is an example of an odometer as in Proposition \ref{prop: classic odometer}. Thus the $C^*$-crossed product $C(\Z_2) \rtimes \Z$ is isomorphic to the Bunce--Deddens algebra denoted BD$(2^\infty)$, a simple $C^*$-algebra (\cite{D}).  Similarly, $aba$ acts as $\cdot + 1$ on $\partial X_1 \cong \Z_2$.  Thus we may write $\partial X = \Z / 2\Z \times \Z_2$, and then $b \cdot (i,\xi) = (i,\xi + 1)$ if $i = 0$, $b \cdot (i,\xi) = (i,\xi)$ if $i = 1$, and $a \cdot (i,\xi) = (i + 1,\xi)$.  Since the element $b$ acts freely and minimally on $\partial X_0$, the element $aba$ acts freely and minimally on $\partial X_1$, and the element $a$ interchanges the two parts of the boundary, it is immediate that $\Lambda$ acts minimally and effectively on $\partial X$.  However, each point of $\partial X_0$ is fixed by $aba$, and each point of $\partial X_1$ is fixed by $b$.  Thus the action is not topologically free.
	
	Now we identify the graph of groups corresponding to this action.  It is easily seen that the stabiliser groups of the vertices are given as follows.  The group $\Lambda_{x_{0,m}}$ is free abelian on generators $aba$ and $b^{2^n}$, if $m \in \Z / 2^n\Z$, while $\Lambda_{x_{1,m}}$ is free abelian on generators $b$ and $(aba)^{2^n}$, if $m \in \Z / 2^n\Z$, and $\Lambda_v = \Lambda \cong \Z^2 \rtimes \Z / 2\Z$ (since $v$ is a global fixed point of $\Lambda$).  We describe the quotient graph by letting $e_n = \Lambda \cdot e_{i,m}$ for any $i$ and any $m \in \Z / 2^n \Z$, $u_n = s(e_n)$, $v = \Lambda \cdot v = r(e_0)$, and $r(e_n) = u_{n+1}$.  The vertex groups are given by $G_v = \Lambda$ and $G_{u_n} = \Z^2$.  The edge groups are given by $G_{e_n} = \Z^2$, with edge maps $\alpha_{e_0} =$ inclusion, $\alpha_{\overline{e_n}} =$ identity for $n \ge 0$, $\alpha_{e_n} = ( \begin{smallmatrix} 2&0 \\ 0&1 \end{smallmatrix} )$ for $n > 0$.  Here is a sketch of the graph of groups (it is an odometer in the sense of Example \ref{example: subodometers}):
	\[
	\begin{tikzpicture}[xscale=3]
	
	\node (00) at (0,0) [circle] {$\Z^2 \rtimes \frac{\Z}{2\Z}$};
	\node (10) at (1,0) [circle] {$\Z^2$};
	\node (20) at (2,0) [circle] {$\Z^2$};
	\node (30) at (3,0) [circle] {$\Z^2$};
	
	\draw[-latex,thick] (10) -- (00)
	node[pos=0.8, inner sep=0.5pt, below=2pt] {$i$}
	node[pos=0.5, inner sep=0.5pt, above=2pt] {$\Z^2$}
	node[pos=0.2, inner sep=0.5pt, below=2pt] {$(\begin{smallmatrix}1&0 \\ 0&1 \end{smallmatrix} )$};
	\draw[-latex,thick] (20) -- (10)
	node[pos=0.8, inner sep=0.5pt, below=2pt] {$(\begin{smallmatrix}2&0 \\ 0&1 \end{smallmatrix} )$}
	node[pos=0.5, inner sep=0.5pt, above=2pt] {$\Z^2$}
	node[pos=0.2, inner sep=0.5pt, below=2pt] {$(\begin{smallmatrix}1&0 \\ 0&1 \end{smallmatrix} )$};
	\draw[-latex,thick] (30) -- (20)
	node[pos=0.8, inner sep=0.5pt, below=2pt] {$(\begin{smallmatrix}2&0 \\ 0&1 \end{smallmatrix} )$}
	node[pos=0.5, inner sep=0.5pt, above=2pt] {$\Z^2$}
	node[pos=0.2, inner sep=0.5pt, below=2pt] {$(\begin{smallmatrix}1&0 \\ 0&1 \end{smallmatrix} )$};
	
	\draw[-latex,dashed] (3.5,0) -- (30);
	
	\end{tikzpicture}
	\]
	We note that the effectiveness of $\GG$ can be easily seen using Propositon \ref{prop: effective subodometer}.  Finally, we identify the $C^*$-algebra 
	\[
	C(\partial X) \rtimes \Lambda = C(\partial X) \rtimes (\Z^2 \rtimes \Z / 2\Z) = \big(C(\partial X) \rtimes \Z (\begin{smallmatrix}1\\0\end{smallmatrix}) \rtimes \Z (\begin{smallmatrix}0\\1\end{smallmatrix})\big) \rtimes \Z / 2\Z. 
	\]
	We have already observed that $C(\Z_2) \rtimes \Z \cong BD(2^\infty)$.  On the other hand, for any $C^*$-algebra $A$, if $\Z$ acts trivially on $A$ then $A \rtimes \Z \cong A \otimes C^*(\Z) \cong A \otimes C(\T)$.  Then
	\[
	C(\partial X) \rtimes \Z (\begin{smallmatrix}1\\0\end{smallmatrix}) \cong ( C(\Z_2) \oplus C(\Z_2) ) \rtimes \Z  (\begin{smallmatrix}1\\0\end{smallmatrix})\cong BD(2^\infty) \oplus ( C(\Z_2) \otimes C(\T) ).
	\]
	Similarly,
	\begin{align*}
		C(\partial X) \rtimes \Z (\begin{smallmatrix}1\\0\end{smallmatrix}) \rtimes \Z (\begin{smallmatrix}0\\1\end{smallmatrix})
		&\cong ( BD(2^\infty) \oplus ( C(\Z_2) \otimes C(\T) ) ) \rtimes \Z (\begin{smallmatrix}0\\1\end{smallmatrix}) \\
		&\cong (BD(2^\infty) \otimes C(\T) ) \oplus ( BD(2^\infty) \otimes C(\T) ).
	\end{align*}
	Since $\Z / 2\Z$ acts by interchanging the two summands, we have
	\[
	C(\partial X) \rtimes \Lambda
	\cong M_2 \otimes BD(2^\infty) \otimes C(\T).
	\]
	
\end{example}

\begin{rmk}\label{rem:BroiseStuff}
In Section 4 of~\cite{BAP}, Broise-Alamichel and Paulin fix a locally finite tree $T$ and a discrete subgroup $\Lambda$ of $\Aut(T)$ which admits a Patterson--Sullivan measure $(\mu_v)_{v \in T^0}$ of positive finite dimension, so that the fixed point sets of nontrivial elliptic elements of $\Lambda$ in the boundary of $T$ have $\mu_v$-measure zero.  In Remarque 4.1, they prove that this zero measure condition is true when $T$ is a uniform tree, $\Lambda$ is a (uniform or nonuniform) lattice in $\Aut(T)$, and an additional hypothesis on the elliptic elements of $\Lambda$ holds.  (See Remark~\ref{rem:LatticeStuff} for the characterisation of lattices in $\Aut(T)$).

Let $\GG = (\G,G)$ be a locally finite graph of groups with fundamental group $\Lambda:=\pi_1(\GG,v)$ and Bass--Serre tree $T$.  We suppose $\GG$ is a graph of trivial groups.  Then $\Lambda$ is a free subgroup of $\Aut(T)$, of rank equal to the first Betti number of the graph $\G$.  Now a necessary condition for the existence of Patterson--Sullivan measure on $\Lambda$ is that the critical exponent of $\Lambda$ is finite and positive (see~\cite{BAP}); under our assumption that $\GG$ is a graph of trivial groups, this condition holds if and only if the first Betti number of $\G$ is finite and greater than one.  Free groups have no nontrivial elliptic elements, hence if $\Lambda$ admits a Patterson--Sullivan measure then the zero measure condition of~\cite{BAP} is vacuously true.  Note that $\Lambda$ will be a lattice in $\Aut(T)$ if and only if the underlying graph $\G$ is finite, in which case $\Lambda$ is a uniform lattice.  On the other hand, a GBS graph of groups (see Section~\ref{sec: GBS}) has infinite vertex groups, so its fundamental group is not discrete in $\Aut(T)$.  Since~\cite{BAP} establishes topological freeness for certain discrete subgroups of $\Aut(T)$, including certain lattices containing torsion, the overlap of our topological freeness results with theirs is the case of graphs of trivial groups where the underlying graph $\G$ has finite Betti number greater than one. (Recall that the Betti number of $\G$ is the cardinality of the set of edges in the complement of any maximal subtree in $\Gamma$.) Such a $\G$ could be a finite or an infinite graph.
\end{rmk}

\subsection{Amenability}\label{subsec: amenability}

For minimality, topological freeness and local contractivity of an action to be useful properties for the corresponding full crossed product, we need the action to be amenable, and hence the crossed product nuclear. In our setting we can use results of \cite{BO} to find a large class of graphs of groups whose fundamental group acts amenably on the boundary.

\begin{thm}\label{thm: nuclearity of C*(GG)}
Let $\GG=(\Gamma,\GG)$ be a locally finite nonsingular graph of countable groups. If each vertex group $G_v$ is amenable, then the action of the fundamental group $\pi_1(\GG,v)$ on the boundary $v\partial X_\GG$ is amenable, and hence $C^*(\GG)$ is nuclear. 
\end{thm}

\begin{proof}
The result follows directly from Theorem~\ref{thm: main theorem}, \cite[Proposition~5.2.1]{BO} and \cite[Lemma~5.2.6]{BO}.
\end{proof}

In fact, the locally finite and nonsingular assumptions are not needed to apply the results of \cite{BO} to see that the action is amenable (although the definition of $v\partial X_\GG$ must be adjusted for the non-locally finite case).

\section{Graphs of trivial groups}\label{sec:undirected graphs}

In this section we consider graphs of trivial groups, and the action of their fundamental group, which will always be a free group, on the boundary of the Bass--Serre tree. For $\GG=(G,\G)$ a graph of trivial groups, we use $\Gamma$ in place of $\GG$ in our notation; so the fundamental group, boundary and graph of groups algebra are denoted $\pi_1(\Gamma,v)$, $v\partial X_\Gamma$ and $C^*(\Gamma)$, respectively. We start by characterising when the action of $\pi_1(\Gamma,v)$ on $v\partial X_\Gamma$ is topologically free. We then prove that all simple $C^*(\Gamma)$ are necessarily UCT Kirchberg algebras.  

\subsection{Topological freeness for graphs}\label{subsec:TopFreeForUnG}

We have the following characterisation of topological freeness:

\begin{thm} \label{thm: top free graph}
	
	Let $\Gamma$ be a nonsingular locally finite graph, and let $v \in \Gamma^0$.  The action of $\pi_1(\Gamma,v)$ on $v \partial X_\Gamma$ is topologically free if and only if $\Gamma$ is not a minimal cycle.
	
\end{thm}

To prove this result we need a lemma, whose proof is omitted, since it follows from the definition of the fundamental group in the case of a graph of trivial groups together with a basic result in algebraic topology (see, for instance, Proposition 1A.2 of the reference~\cite{hatcher}). First recall that if $T$ is a maximal subtree of the graph $\Gamma$, and $e \in \Gamma^1$, then $\varepsilon(e) = [v, r(e)] e [s(e), v]$. (Since all vertex groups are trivial, there can be no confusion from our use of a path in $\Gamma$ to denote an element of the fundamental group.) If $e_1$, $\ldots$, $e_k \in \Gamma^1$ then
\[
\varepsilon(e_1) \cdots \varepsilon(e_k) = [v,r(e_1)] e_1 [s(e_1), r(e_2)] e_2 \cdots [r(e_{k-1}), s(e_k)] e_k [s(e_k), v].
\]
Also recall from, say, \cite{Se}, that an {\em orientation} of a graph $\Gamma$ is a set $\Gamma_+ \subseteq \Gamma^1$ such that for each $e \in \Gamma^1$ exactly one of $e$ and $\overline{e}$ belongs to $\Gamma_+$.

\begin{lem} \label{lem: free basis}
	
	Let $\Gamma$ be a graph, $T$ a maximal subtree, $B = \Gamma^1 \setminus T^1$, and $\Gamma_+$ an orientation for $\Gamma$. Fix $v \in \Gamma^0$.  Let $G = \pi_1(\Gamma,v) = v \Gamma^* v$.  Then $G$ is freely generated by $\{ \varepsilon(e) : e \in B \}$.
	
\end{lem}

%
%

\begin{proof}[Proof of Theorem~\ref{thm: top free graph}]
	If $\Gamma$ is a minimal cycle, $\gamma$, then $v \partial X_\Gamma = \{ \gamma^\infty, \overline{\gamma}^\infty \}$, and $\pi_1(\Gamma,v) \cong \Z$ acts trivially.  Conversely, suppose that $\Gamma$ is not a minimal cycle.  The proof depends on the fundamental group of $\Gamma$.  Let $T$ and $B$ be as in Lemma \ref{lem: free basis}. If $\pi_1(\Gamma,v)$ is trivial, then all points of $v \partial X_\Gamma$ have trivial isotropy.  Suppose next that $\pi_1(\Gamma,v)$ is nonabelian.  Then there exist $e_1$, $e_2 \in B$ with $e_2 \not= e_1$, $\overline{e_1}$.  Let $\mu \in v \Gamma^*$.  Write $\mu = [v,r(f_1)] f_1 [s(f_1),r(f_2)] f_2 \cdots [s(f_{k-1}),r(f_k)] f_k [s(f_k), s(\mu)]$, with $f_i \in B$ and $f_{i+1} \not= \overline{f_i}$.  Let $d$ be the sourcemost edge of $\mu$.  If $\Gamma$ is treelike at $d$, choose $\xi_0 \in s(d)E_{s,d}^\infty$, and set $\xi = \mu \xi_0$.  Then $\xi \in Z(\mu)$.  Since $\xi_0$ is an infinite path in a tree, its edges are all distinct, hence the sequence of edges is aperiodic. (Recall that a sequence $y_1y_2\dots$ is {\em aperiodic} if $y_my_{m+1}\dots=y_ny_{n+1}\dots \implies m=n$.)  Therefore $\xi$ has trivial isotropy.  If $\Gamma$ is not treelike at $d$, then $E_{s,d}$ contains a cycle, hence an edge $f \in B$.  Let $\beta \in s(d) E_{s,d}^* r(f)$ be a shortest path.  Choose an aperiodic sequence $(i_1,i_2, \ldots) \in \{ 1,2 \}^\infty$ such that $e_{i_1} \not= \overline{f}$.  Put $\xi = \mu \beta f \varepsilon(e_{i_1}) \varepsilon(e_{i_2}) \cdots$.  Then $\xi \in Z(\mu)$, and again, since the sequence of edges in $\xi$ is eventually aperiodic, the isotropy at $\xi$ is trivial.
	
	Finally we suppose that $\pi_1(\Gamma,v)$ is infinite cyclic.  Then there is $e \in \Gamma^1$ such that $B = \{ e, \overline{e} \}$.  Choose $v = r(e)$.  Since $\varepsilon(e)$ is a minimal cycle there must be an edge $f$ not lying on $\varepsilon(e)$ or $\varepsilon(\overline{e})$.  Note that if $f$ can flow to $e$, then also $f$ can flow to $\overline{e}$ (if $\eta$ is a minimal path from $r(f)$ to a vertex in $\varepsilon(e)$, then $e [s(e),r(\eta)]\eta f$ and $\overline{e} [v,r(\eta)] \eta f$ are both reduced paths).  
	
	We claim that $f$ and $\overline{f}$ cannot both flow to $e$.  For if they do, we see from Lemma \ref{def: suitability} that there are $\nu_1$, $\nu_2 \in \Gamma^*$ such that $e \nu_1 f \overline{\nu_2}\, \overline{e}$ has no cancellation.  Choosing $\nu_1$ and $\nu_2$ as short as possible, we may assume that $e$ does not lie on $\nu_1$ or on $\nu_2$.  If $\overline{e}$ lies on $\nu_1$, we may write $\nu_1 = \nu_1' \overline{e} \nu_1''$, where $\overline{e}$ does not lie on $\nu_1'$.  Then $\nu_1'$ is a cycle, hence must contain an edge from $B$ different from $e$ and $\overline{e}$, a contradiction.  Similarly, $\overline{e}$ does not lie on $\nu_2$.  Then $\nu_1 f \overline{\nu_2}$ is a cycle, hence contains an edge from $B$ different from $e$ and $\overline{e}$, again a contradiction.  This establishes the claim.  For definiteness suppose that $\overline{f}$ cannot flow to $e$ (and hence, nor to $\overline{e}$).  Then $\Gamma$ is treelike at $f$, since a cycle in $E_{s,f}$ would contain an element of $B$ different from $e$ and $\overline{e}$.  
	
	Now we prove that the action is topologically free.  Let $\mu = e_1 \cdots e_n \in v \Gamma^*$.  First suppose that $e_n$ does not lie on $\varepsilon(e)$ or $\varepsilon(\overline{e})$.  Since $r(\mu) = v$, it is clear that $e_n$ can flow to $e$.  Therefore $\overline{e_n}$ cannot flow to $e$, hence $\Gamma$ is treelike at $e_n$.  Let $\xi_0 \in s(e_n) E_{s,e_n}^\infty$, and set $\xi = \mu \xi_0$.  Then $\xi \in Z(\mu)$, and since the edges of $\xi_0$ are all distinct, the isotropy at $\xi$ is trivial.  Next suppose that $e_n$ lies on $\varepsilon(e)$ (for definiteness).  Let $f$ be an edge not lying on $\varepsilon(e)$ or $\varepsilon(\overline{e})$ such that $\overline{f}$ cannot flow to $e$.  Then $f$ flows to $e$.  Choose a shortest path $\beta$ with $s(\beta) = r(f)$ and $r(\beta)$ on $\varepsilon(e)$.  We may write $\varepsilon(e) = \gamma_1 \gamma_2 \gamma_3$, where $e_n$ is the sourcemost edge of $\gamma_1$, and $s(\gamma_2) = r(\beta)$.  Let $\xi_0 \in s(f) E_{s,f}^\infty$ and put $\xi = \mu \gamma_2 \beta f \xi_0$.  Then $\xi \in Z(\mu)$ and the isotropy at $\xi$ is trivial.
\end{proof}

\begin{cor} \label{cor: min top free graph}
	
	Let $\Gamma$ be a nonsingular locally finite graph.  If the action of $\pi_1(\Gamma,v)$ on $v \partial X_\Gamma$ is minimal, then it is also topologically free.
	
\end{cor}

\begin{proof}
This follows from Theorem~\ref{thm: nonminimality second take}(d) and Theorem~\ref{thm: top free graph}.
\end{proof}

\subsection{$C^*$-algebras associated to graphs}\label{subsec:SimpleCStarGamma}

Let $\Gamma$ be a nonsingular locally finite graph.  As observed in Section \ref{subsec: top freeness}, $C^*(\Gamma)$ is simple if and only if the action of $\pi_1(\Gamma,v)$ on $v \partial X_\Gamma$ is minimal and topologically free. It follows from Corollary \ref{cor: min top free graph} that $C^*(\Gamma)$ is simple if and only if the action is minimal. For a nonsingular graph, conditions (2) and (3) of Proposition \ref{prop: min and lc} cannot occur, so if the action is minimal, then it is also locally contractive.  By Theorem \ref{thm: nuclearity of C*(GG)} we know that $C^*(\Gamma)$ is nuclear.  Moreover, since $C^*(\Gamma)$ is the $C^*$-algebra of an \'etale groupoid, it satisfies the UCT (\cite{T}).  We have established the following.

\begin{thm} \label{thm: graph Kirchberg}
	
	Let $\Gamma$ be a nonsingular locally finite graph. If $C^*(\Gamma)$ is simple it is a UCT Kirchberg algebra.
	
\end{thm}

Recall, as we did in Section~\ref{subsec: top freeness}, that the \textit{first Betti number} of $\Gamma$ is the cardinality of the set of edges (not including reverse edges) in the complement of any maximal subtree in $\Gamma$; alternatively, the cardinality of any free basis of the fundamental group.

\begin{lem} \label{lem: simple undirected graph}
	
	Let $\Gamma$ be a nonsingular locally finite graph. The following are equivalent.
	
	\begin{enumerate}
		
		\item \label{lem: simple undirected graph.1} $C^*(\Gamma)$ is simple.
		
		\item \label{lem: simple undirected graph.2} The action of $\pi_1(\Gamma,v)$ on $v \partial X_\Gamma$ is minimal.
		
		\item \label{lem: simple undirected graph.3} There are no edges at which $\GG$ has a constant tree (in the sense of Definition~\ref{def:ConstantTree}), and the fundamental group of $\Gamma$ is nonabelian.
		
	\end{enumerate}
	In this case $C^*(\Gamma)$ is a Kirchberg algebra.
	
\end{lem}

\begin{proof}
	$(1)\iff(2)$ follows from \cite{AS} and Corollary \ref{cor: min top free graph}. $(2)\implies (3)$ follows from Proposition \ref{prop: minimal and treelike} because if $\Gamma$ has a constant tree, then $\Gamma$ is a ray; but a ray is a singular graph. For $(3)\implies (2)$, suppose the action is not minimal.  Theorem \ref{thm: nonminimality second take} gives two alternatives.  If (c') holds, then there is an edge $e$ not lying on a cycle such that $\Gamma$ is treelike at $e$.  Then $\{e, \overline{e} \} \cup \Gamma_{s,e}$ is a constant tree.  If (d) holds then the first Betti number of $\Gamma$ equals one.
	
	The final statement follows from Theorem \ref{thm: graph Kirchberg}.
\end{proof}

\begin{cor} \label{cor: infinite betti number}
	
	Let $\Gamma$ be a nonsingular locally finite graph with $C^*(\Gamma)$ a simple $C^*$-algebra.  The first Betti number of $\Gamma$ is finite if and only if $\Gamma$ is finite.
	
\end{cor}

The following theorem follows from calculations in \cite{CLM} and \cite{II}.    

\begin{thm} \label{thm: simple undirected graph}
	
	Let $\Gamma$ be a nonsingular locally finite graph with $C^*(\Gamma)$ a simple $C^*$-algebra.  Let $2 \le n \le \infty$ be the first Betti number of $\Gamma$.  Then $C^*(\Gamma)$ is the UCT Kirchberg algebra with $K_0 = \Z^n \oplus \Z/(n-1)\Z$ if $n < \infty$, $K_0 = \Z^n$ if $n = \infty$, and $K^1 = \Z^n$.  If $\Gamma$ is finite, then the class of the identity is the generator of the torsion subgroup of $K_0$, while if $\Gamma$ is infinite, then $C^*(\Gamma)$ is stable.
	
\end{thm}

We remark on the contrast with $C^*$-algebras of directed graphs, that include all simple AF algebras (up to stable isomorphism). It follows from the Kirchberg-Phillips classification theorem (\cite{Bla2}) that $C^*(\Gamma)$ is characterised up to isomorphism by $K$-theory (and the position of the class of the unit). It follows from Theorem \ref{thm: simple undirected graph} that there are relatively few Kirchberg algebras obtained from undirected graphs.  (Again, this is a sharp contrast to the case of directed graphs. Any pair of abelian groups can be realised as the $K$-groups of a Kirchberg algebra.  If the $K_1$ group is free abelian, the pair can be realised as the $K$-theory of a Kirchberg algebra obtained from a directed graph \cite{Sz}.)

\section{Generalised Baumslag--Solitar graphs of groups}\label{sec: GBS}

A \textit{generalised Baumslag--Solitar graph of groups}, or \textit{GBS graph of groups}, is a graph of groups $\GG = (G,\G)$ in which all edge and vertex groups are infinite cyclic.  In contrast with some of the literature, we allow the graph $\G$ to be infinite. In this section we characterise the GBS graphs of groups for which the action of $\pi_1(\GG,v)$ on $v\partial X_\GG$ is topologically free.  A precise statement is given in Theorem~\ref{thm: gbs top free}. We finish with some results on the $C^*$-algebras associated to GBS graphs of groups.

\subsection{Topological freeness for GBS graph of groups}\label{subsec:TopFreeForGBS}

Note that a GBS graph of groups $\GG = (G,\G)$ is locally finite if and only if the graph $\G$ is locally finite.  Throughout this section we assume that $\GG$ is a locally finite nonsingular GBS graph of groups.  

\begin{ntn}\label{ntn: gbs}
	We use the following notation:
	
	\begin{enumerate}
		
		\item We use additive notation for the vertex and edge groups.
		
		\item Suppose that a generator has been chosen in each vertex and edge group.  For $e \in \Gamma^1$ we let $\omega_e$ denote the nonzero integer such that $\alpha_e$ is given by multiplication with $\omega_e$.  Then $|\omega_e|$ is independent of the choices of generators.
		
		\item For each $e \in \Gamma^1$ we choose $\Sigma_e = \{0,1,\ldots,|\omega_e| - 1\}$.
		
		\item For $q \in \Q$ we let $\langle q \rangle$ denote the smallest positive denominator that can be used to express $q$ as a fraction.
		\end{enumerate}
\end{ntn}
		
\begin{rmk}  The positive integer $|\omega_e|$ is equal to the index $[G_{r(e)}:\alpha_e(G_e)]$, which is denoted $i(e)$ in some other works using graphs of groups, such as Bass--Kulkarni~\cite{BK} (note that in~\cite{BK}, $\alpha_e$ maps $G_e$ into the initial vertex of the edge $e$). 
\end{rmk} 		

We also need the following definition.
		
\begin{dfn} For a $\GG$-word $\gamma = g_1 e_1 \dots g_n e_n g_{n+1}$, we define the {\em signed index ratio} $q(\gamma) \in \Q^\times$ by
		\[
		q(\gamma) = \prod_{i=1}^n \frac{\omega_{\overline{e_i}}}{\omega_{e_i}}.
		\]
		Note that $q(\gamma_1 \gamma_2) = q(\gamma_1) q(\gamma_2)$ when the $\GG$-paths $\gamma_1$ and $\gamma_2$ can be concatenated i.e. when $s(\gamma_1) = r(\gamma_2)$.  In particular, $q$ restricts to a homomorphism $\pi_1(\GG,v) \to \Q^\times$.	
\end{dfn}

\begin{rmk}
The signed index ratio $q$ is related to maps appearing in other works on graphs of groups and (generalised) Baumslag--Solitar groups.   In Definition 6.3 of Forester~\cite{F} and Section 2.3 of Clay--Forester~\cite{CF}, the restriction of $q$ to the generalised Baumslag--Solitar group $\pi_1(\GG,v)$ is termed the {\em signed modular homomorphism}.  The analogous map on Baumslag--Solitar groups is denoted by $\varphi$ in Kropholler~\cite{K}.  In Section 1 of Bass--Kulkarni~\cite{BK}, which considers general locally finite graphs of groups, for $e$ an edge  $q(e)$ is defined to equal the positive rational $i(e)/i(\overline{e})$.  This map $q$ is then extended to edge-paths and pre-composed with the projection from $\pi_1(\GG,v)$ to the fundamental group of the graph $\G$ to obtain a homomorphism $\pi_1(\GG,v) \to \Q^\times_{>0}$.   
\end{rmk}

We can now state our characterisation of topological freeness. Recall the notion that $\GG$ has a constant tree at an edge from Definition~\ref{def:ConstantTree}.

\begin{thm} \label{thm: gbs top free}
	
	Let $\GG$ be a locally finite nonsingular GBS graph of groups.  The action of $\pi_1(\GG,v)$ on $v\partial X_\GG$ is topologically free if and only if the following conditions hold:
	
	\begin{enumerate}
		
		\item \label{thm: gbs top free.1} $\GG$ has no constant trees; and
		
		\item \label{thm: gbs top free.2} there exists $\xi = g_1 e_1 g_2 e_2 \dots \in \partial W_\GG$ such that $\limsup_{n \to \infty} \langle q(\gamma_n) \rangle = \infty$, where $\gamma_n = g_1 e_1 \dots g_n e_n$.
		
	\end{enumerate}
	
\end{thm}

%

To prove Theorem~\ref{thm: gbs top free} we need a series of results. We start with the effect of a group element $g$ ``passing through" a $\GG$-word $\gamma$.

\begin{lem} \label{lem: pass through}
	
	Let $\gamma = g_1 e_1 \ldots g_n e_n g_{n+1}$ be a $\GG$-word.  Let $g \in G_{r(\gamma)}$ and $g' \in G_{s(\gamma)}$ be such that $g \gamma = \gamma g'$.  Then $g' = q(\gamma) g$, and hence $\langle q(\gamma) \rangle$ divides $g$.
	
\end{lem}

\begin{proof}
	Let $\theta_1 \in G_{e_1}$ and $0 \le \tau_1 < |\omega_{e_1}|$ be such that $g + g_1 = \tau_1 + \omega_{e_1} \theta_1$.  Then
	\[
	(g + g_1) e_1 = (\tau_1 + \alpha_{e_1}(\theta_1)) e_1
	= \tau_1 e_1 \alpha_{\overline{e_1}}(\theta_1)
	= \tau_1 e_1 (\omega_{\overline{e_1}} \theta_1).
	\]
	Let $\theta_2 \in G_{e_2}$ and $0 \le \tau_2 < |\omega_{e_2}|$ be such that $\omega_{\overline{e_1}} \theta_1 + g_2 = \tau_2 + \omega_{e_2} \theta_2$.  Then
	\[
	(\omega_{\overline{e_1}} \theta_1 + g_2) e_2
	= (\tau_2 + \alpha_{e_2}(\theta_2)) e_2
	= \tau_2 e_2 \alpha_{\overline{e_2}}(\theta_2)
	= \tau_2 e_2 (\omega_{\overline{e_2}} \theta_2).
	\]
	Inductively we find integers $\theta_i \in G_{e_i}$ and $0 \le \tau_i < |\omega_{e_i}|$ for $1 \le i < n$ such that $\omega_{\overline{e_i}} \theta_i + g_{i+1} = \tau_{i+1} + \omega_{e_{i+1}} \theta_{i+1}$.  Then
	\begin{align*}
		g \gamma
		&= (g + g_1) e_1 g_2 e_2 \dots g_n e_n g_{n+1} \\
		&= \tau_1 e_1 (\omega_{\overline{e_1}} \theta_1 + g_2) e_2 \dots \\
		&= \tau_1 e_1 \tau_2 e_2 (\omega_{\overline{e_2}} \theta_2 + g_3) e_3 \dots \\
		&= \dots \\
		&= \tau_1 e_1 \tau_2 e_2 \dots \tau_n e_n (\omega_{\overline{e_n}} \theta_n + g_{n+1}).
	\end{align*}
	It follows that $\tau_i = g_i$ for $i \le n$.  Therefore
	\begin{align*}
		g &= \omega_{e_1} \theta_1 \\
		\omega_{\overline{e_1}} \theta_1 &= \omega_{e_2} \theta_2 \\
		\omega_{\overline{e_2}} \theta_2 &= \omega_{e_3} \theta_3 \\
		&\dots \\
		\omega_{\overline{e_{n-1}}} \theta_{n-1} &= \omega_{e_n} \theta_n.
	\end{align*}
	Thus we have
	\[
	\omega_{\overline{e_n}} \theta_n
	= \frac{\omega_{\overline{e_n}}}{\omega_{e_n}} \omega_{\overline{e_{n-1}}} \theta_{n-1}
	= \frac{\omega_{\overline{e_n}}}{\omega_{e_n}} \frac{\omega_{\overline{e_{n-1}}}} {\omega_{e_{n-1}}} \omega_{\overline{e_{n-2}}} \theta_{n-2}
	= \dots
	= g \prod_{i=1}^n \frac{\omega_{\overline{e_i}}}{\omega_{e_i}}
	= g q(\gamma).
	\]
	We conclude that $g' = g q(\gamma)$.
\end{proof}

In the next lemma we describe certain (rooted) subtrees that might be present in $X_{\GG,v}$.  These are instances where the boundary of the subtree can be identified with an odometer, as in Proposition \ref{prop: classic odometer}.

\begin{lem} \label{lem: odometer}
	
	Let $f_1 f_2 \dots$ be a sequence in $\Gamma^1$ such that $s(f_i) = r(f_{i+1})$ and $f_{i+1} \not= \overline{f_i}$ for all $i$.
	
	\begin{enumerate}
		
		\item \label{lem: odometer.1} There is a choice of generators of the $G_{f_i}$ and $G_{r(f_i)}$ such that $\omega_{f_i}$ and $\omega_{\overline{f_i}}$ are all positive.
		
	\end{enumerate}
	
	Suppose further that $\omega_{\overline{f_i}} = 1$ for all $i$, and that $\omega_{f_i} > 1$ for infinitely many $i$.  Let $S$ be the set of all infinite $\GG$-paths with edge sequence $f_1 f_2 \dots$.  (Thus $S$ is homeomorphic to the Cantor set $\prod_{i=1}^\infty \frac{\Z}{\omega_{f_i}\Z}$.)
	
	\begin{enumerate}
		\addtocounter{enumi}{1}
		
		\item \label{lem: odometer.2} Let $\gamma = h_1 f_1 \dots h_n f_n$ be a $\GG$-path, and let $g$, $g' \in \Z$.  Then $g \gamma = \gamma g'$ if and only if $\prod_{i=1}^n \omega_{f_i}$ divides $g$, and in this case $g' = g (\prod_{i=1}^n \omega_{f_i})^{-1}$.
		
		\item \label{lem: odometer.3} $G_{r(f_1)}$ acts freely on $S$.
		
		\item \label{lem: odometer.4} If $\xi = h_1 f_1 h_2 f_2 \dots \in S$ is such that $h_i \not= 0$ for infinitely many $i$, and $h_i \not= \omega_{f_i} - 1$ for infinitely many $i$, then for any $g$, $g \xi$ and $\xi$ differ in at most finitely many coefficients.
		
	\end{enumerate}
	
\end{lem}

\begin{proof}
	For part \eqref{lem: odometer.1}, choose arbitrarily generators $a_i$ for $G_{r(f_i)}$ and $b_i$ for $G_{f_i}$.  Replacing $b_1$ by $-b_1$ if necessary, we may assume that $\omega_{f_1} > 0$.  Then replacing $a_2$ by $-a_2$ if necessary, we may assume that $\omega_{\overline{f_1}} > 0$. This process may be continued with $b_2$, then $a_3$, and so on.
	
	Part \eqref{lem: odometer.2} follows from Lemma \ref{lem: pass through}.  
	
	For part \eqref{lem: odometer.3}, it was shown in Example \ref{example: subodometers} that $S$ is a compact group (abelian in this case), and that $G_{r(f_1)} \subseteq S$ acts by translation.  This implies part \eqref{lem: odometer.3}.
	
	For part \eqref{lem: odometer.4}, recall from Example \ref{example: subodometers} that when $S$ is realised as an inverse limit of finite cyclic groups, the generator of $G_{r(f_1)}$ acts as $+ 1$.  Then if $h_i = \omega_{f_i} - 1$ for $i < k$, and $h_k < \omega_{f_k} - 1$, then adding 1 to $\xi$ results in $\xi' = h_1' f_1 h_2' f_2 \cdots$, where $h_i' = 0$ for $i < k$, $h_k' = h_k + 1$, and $h_i' = h_i$ for $i > k$.  Similarly, if $h_i = 0$ for $i < k$ and $h_k > 0$, then subtracting 1 from $\xi$ results in $\xi' = h_1' f_1 h_2' f_2 \cdots$, where $h_i' = \omega_{f_i} - 1$ for $i < k$, $h_k' = h_k - 1$, and $h_i' = h_i$ for $i > k$.  Part \eqref{lem: odometer.4} follows from these observations.
\end{proof}

In the next result, we give a sufficient condition for certain cylinder sets to contain a point with trivial isotropy.  Recall that for $q \in \Q$ we let $\langle q \rangle$ denote the smallest positive denominator that can be used to express $q$ as a fraction.

\begin{lem} \label{lem: trivial isotropy}
	
	Let $\xi = g_1 e_1 g_2 e_2 \dots \in \partial W_\GG$.  Suppose that $\limsup_{n \to \infty} \langle q(g_1 e_1 \dots g_n e_n) \rangle = \infty$.  Then for each $\ell$, $Z(g_1 e_1 \dots g_\ell e_\ell)$ contains a point with trivial isotropy.
	
\end{lem}

\begin{proof}
	For each $n$ we will let $\gamma_n = g_1 e_1 \dots g_n e_n$.  We consider first the case that the sequence of edges appearing in $\xi$ is aperiodic.  We claim that $G_{r(\xi)}$ does not contain nontrivial elements of the isotropy at $\xi$.  For this, let $1 \le n_1 < n_2 < \dots$ with $\langle q(\gamma_{n_i}) \rangle > i$ for all $i$.  If $g \in G_{r(\xi)}$ with $g \xi = \xi$, then by Lemma \ref{lem: pass through} we must have that $\langle q(\gamma_{n_i}) \rangle$ divides $g$ for all $i$, and hence we must have $g = 0$.
	
	Now suppose that $\beta \xi = \xi$ for some $\beta \in \pi_1(\GG,r(\xi))$ with $|\beta| > 0$.  We may write $\beta = \beta_1 \gamma_k^{-1}$ without cancellation, for some reduced $\GG$-word $\beta_1$, where $\beta_1 g_{k+1} e_{k+1}$ also has no cancellation (where these calculations are performed in the groupoid $W_\GG$).  By the assumed aperiodicity it must be the case that $\beta_1 = \gamma_k g$, for some $g \in G_{s(e_k)}$.  It follows that $g \xi' = \xi'$, where $\xi' = g_{k+1} e_{k+1} \dots$.  Since $\xi'$ satisfies the same hypotheses as $\xi$, the previous argument shows that $g = 0$, and hence that $\beta = 0$, a contradiction.  We now have that for each $\ell$, $\xi \in Z(\gamma_\ell)$ is a point with trivial isotropy.
	
	Now we consider the case that the sequence of edges in $\xi$ is eventually periodic.  Fix $\ell$, and let $m \ge \ell$ and $p>0$ be such that $e_{m+1} \dots e_{m+p}$ is a period of the edge sequence of $\xi$.  Put $\delta = g_{m+1} e_{m+1} \dots g_{m+p} e_{m+p}$.  For $n > m$ write $n = m + kp + r$, where $k \ge 0$ and $0 \le r < p$.  Note that $q(\gamma_p) = q(\gamma_m) q(\delta)^k q(\delta')$, where $\delta' = h_1 e_{m+1} \dots h_r e_{m+r}$ for some $h_1$, $\ldots$, $h_{r}$.  The set of $q(\delta')$ for all such $\delta'$ is a finite set.  Therefore if $q(\delta) \in \Z$ then $\langle q(\gamma_n) \rangle$ is bounded, in contradiction to the hypothesis of the lemma.  Therefore $\langle q(\delta) \rangle > 1$.
	
	There are two subcases.  First, suppose that $|\omega_{\overline{e_{m+i}}}| = 1$ for all $1 \le i \le p$.  Since $\langle q(\delta) \rangle > 1$ there exists $1 \le j \le p$ such that $|\omega_{e_{m+j}}| > 1$.  Then we may use Lemma \ref{lem: odometer}.  Let $S$ be as in that lemma.  Since $G_{r(e_1)}$ acts freely on $S$, there is no isotropy in $G_{r(e_1)}$.  Choose $\zeta = h_1 e_1 h_2 e_2 \dots \in S \cap Z(\gamma_\ell)$ such that the sequence $(h_{n + kp + j})_{k=1}^\infty$ is aperiodic, is nonzero infinitely often, and is not equal to $|\omega_{e_{m+j}}| - 1$ infinitely often.  Suppose that $\beta$ is a $\GG$-word such that $\beta \zeta = \zeta$. We may again write $\beta = \beta_1 g \gamma_k^{-1}$, where $\beta_1$ is a $\GG$-path, $g \in G_{s(\beta_1)}$, and $\beta_1 g g_{k+1} e_{k+1}$ has no cancellation. By Lemma \ref{lem: odometer}\eqref{lem: odometer.4}, the element $g$ changes at most finitely many coefficients in the tail of coefficients of $\zeta$. Since the sequence of those coefficients is aperiodic, we must have that $\beta_1 = \gamma_k$.  But then $g$ must be trivial to achieve $\beta \zeta = \zeta$.  Thus the isotropy at $\zeta$ cannot contain an element of positive length, and so the point $\zeta \in Z(\gamma_\ell)$ has trivial isotropy.
	
	Finally, suppose that $|\omega_{\overline{e_{m+i}}}| > 1$ for some $i$, $1 \le i \le p$.  Replacing $m$ by $m+i$, if necessary, we may assume that $i = p$.  Choose $d_1$, $d_1'$, $\ldots$ $d_p$, $d_p' \in \Z$ such that $d_1 \not= 0$, $d_p' \not= 0$, and
	\[
	\beta := d_p' \overline{e_{m+p}} \dots d_1' \overline{e_{m+1}} d_1 e_{m+1} \dots d_p e_{m+p}
	\]
	is a $\GG$-path.  By Lemma \ref{lem: pass through} we know that if $g$, $g' \in \Z$ are such that $g \beta = \beta g'$ then $g = g'$.  We will construct a point $\eta \in Z(\gamma_\ell)$ having trivial isotropy.  Choose an aperiodic sequence $\sigma_1$, $\sigma_2$, $\ldots$ in the set $\{\delta,\beta\}$.  Since $d_p' \not= 0$ and $g_{m+1} e_{m+1} \not= 0 \overline{e_{m+p}}$, we may define the infinite $\GG$-path $\eta = \gamma_m \sigma_1 \sigma_2 \dots \in Z(\gamma_\ell)$.  Note that the sequence of edges appearing in $\eta$ is aperiodic.  This is because the sequence of edges in $\beta$ is unchanged by reversal, whereas this is not true for $\delta$ since $\langle q(\delta) \rangle > 1$.  As before, this forces the isotropy at $\eta$ to be a subgroup of $G_{r(\eta)}$.  Since $\langle q(\delta)^k \rangle \to \infty$, Lemma \ref{lem: pass through} implies that the isotropy at $\eta$ is trivial.
\end{proof}

We need one more result to prove Theorem~\ref{thm: gbs top free}.

\begin{lem} \label{lem: constant tree}
	
	Let $\GG$ be a locally finite nonsingular GBS graph of groups.  Suppose that $\GG$ has a constant tree at $e \in \Gamma^1$.  Let $f \in E_{s,e}^1$ with $r(f) = s(e)$.  Then every point in $Z(0 f)$ has nontrivial isotropy.
	
\end{lem}

\begin{proof}
	Let $\xi = 0 f h_2 f_2 h_3 f_3 \dots \in Z(0f)$.  Since $\alpha_{f_i}$ is onto for all $i$, we have that all $h_i$ equal 0, and $G_{r(f)}$ fixes $\xi$.
\end{proof}

We can now prove our characterisation of topological freeness.

\begin{proof}[Proof of Theorem~\ref{thm: gbs top free}]
	
	For the ``if'' direction, we let $\mu \in \GG^*$ define a cylinder set in $\partial W_\GG$; we show that $Z(\mu)$ contains a point of trivial isotropy. Now let $e$ denote the sourcemost edge of $\mu$; so $\mu = \mu' e$ (with no cancellation) for some $\GG$-word $\mu'$. Let $\xi = g_1 e_1g_2e_2 ...$.  Assume that $\xi$ can flow to $e$, and that $\xi$ is as in (2). There is $n$ such that for some $\GG$-word $\nu \in s(e) \GG^* r(e_n)$, the concatenation $e\nu e_n$ has no cancellation. By (2) and Lemma \ref{lem: trivial isotropy} we know that there is $\eta \in Z(g_1 e_1 ... g_n e_n)$ with trivial isotropy.  Write $\eta = g_1 e_1 ... g_n e_n \eta'$.  (So $\eta'$ also has trivial isotropy.)  Let $\gamma = 0 e \nu \overline{e_n} g_n^{-1} ... \overline{e_1} g_1^{-1}$.  Then $\mu' \gamma \eta = \mu \nu \eta'$ (without cancellation) also has trivial isotropy, and is in the cylinder set $Z(\mu)$.
	
	Suppose now that $\xi$ cannot flow to $e$.  We claim that $\xi$ can flow to $\overline{e}$.  To see this, suppose first that there is a $\Gamma$-path $f_1 \dots f_k$, $k \ge 1$, which is of minimal length such that $r(f_1) = s(e)$ and $s(f_k)$ lies on $\xi$; say, $s(f_k) = r(e_\ell)$.  Since $f_1\dots f_k$ is of minimal length, we have $f_k \not= \overline{e_\ell}$.  If $f_1 \not= \overline{e}$, then the $\GG$-path $0 e 0 f_1 \dots 0 f_k 0 e_\ell$ has no cancellation, contradicting the assumption that $\xi$ cannot flow to $e$.  Therefore $f_1 = \overline{e}$.  It follows that $\xi$ can flow to $\overline{e}$.  On the other hand, suppose that no such minimal $\Gamma$-path exists.  Then it must be the case that $s(e)$ lies on $\xi$; say, $s(e) = r(e_\ell)$.  If $e_\ell \not= \overline{e}$, then $\xi$ can flow to $e$, a contradiction.  Therefore $e_\ell = \overline{e}$, and so again we have that $\xi$ can flow to $\overline{e}$.
	
	There are now two cases.  First, suppose that $\GG$ is not treelike at $e$.  Then there is a reduced $\GG$-loop $\gamma$ at $s(e)$ such that $0e \gamma \overline{e}$ is reduced.  Since $\xi$ can flow to $\overline{e}$, there is a reduced $\GG$-word $\gamma'$, and $\ell \in \N$, such that $0\overline{e} \gamma' e_\ell$ is reduced.  Then $\mu \gamma \overline{e} \gamma' e_\ell g_{\ell+1} f_{\ell+1} \dots \in Z(\mu)$ has trivial isotropy.
	
	Second, suppose that $\GG$ is treelike at $e$.  Since $\GG$ has no constant trees, there must exist $\eta \in s(e) \partial E_{s,e}$ such that $|\omega_f| > 1$ for infinitely many edges $f$ lying on $\eta$.  It follows that $\eta$ satisfies the hypothesis of Lemma \ref{lem: trivial isotropy}, and hence there is $\eta'$ having the same edge sequence as $\eta$ and having trivial isotropy.  Then $\mu \eta' \in Z(\mu)$ has trivial isotropy.  Thus in all cases, $Z(\mu)$ contains a point with trivial isotropy, and hence the action is topologically free. 
	
	For the ``only if'' direction, suppose that the action is topologically free.  Lemma \ref{lem: constant tree} implies that \eqref{thm: gbs top free.1} holds.  We suppose that \eqref{thm: gbs top free.2} is false, and derive a contradiction by showing that every point of $\partial W_\GG$ has nontrivial isotropy.  Let $\xi = g_1 e_1 g_2 e_2 \dots \in \partial W_\GG$, and let $\gamma_n = g_1 e_1 \dots g_n e_n$.  By assumption, $\{ \langle q(\gamma_n) \rangle : n = 1,\, 2,\, \ldots \}$ is bounded.  Let $C$ be the least common multiple of $\{ \langle q(\gamma_n) \rangle : n = 1,\,2,\, \ldots \}$.  Then for each $n$ we have $C \cdot \gamma_n = \gamma_n C q(\gamma_n)$, by Lemma \ref{lem: pass through}.  It follows that $C \cdot Z(\gamma_n) \subseteq Z(\gamma_n)$ for all $n$.  Since $\bigcap_{n=1}^\infty Z(\gamma_n) = \{ \xi \}$, we have that $C \xi = \xi$.
\end{proof}

\begin{rmk}\label{rmk: unimodular}
One instance where condition \eqref{thm: gbs top free.2} holds is when there is a $\GG$-loop $\gamma$ such that $|q(\gamma)| \not= 1$.  The reason is that in this case, either $|\langle q(\gamma) \rangle| > 1$ or $|\langle q(\gamma^{-1}) \rangle| > 1$.  Then the infinite iteration of $\gamma$, or of $\gamma^{-1}$, will give an infinite reduced path satisfying \eqref{thm: gbs top free.2}.  Of course, if $\GG$ is finite, this is the only way that \eqref{thm: gbs top free.2} can hold.  We recall the notion of a {\em unimodular} graph of groups from \cite{BK}.  In the context of a locally finite GBS graph of groups, this means that $q(\gamma) = \pm 1$ for every $\GG$-loop $\gamma$.  In particular, if $\Gamma$ is finite then $\GG$ is unimodular if and only if \eqref{thm: gbs top free.2} does not hold.  We obtain the following corollary of Theorem \ref{thm: gbs top free}.
\end{rmk}

\begin{cor}\label{cor: t.f. and unimodular}
	
	Let $\GG$ be a finite nonsingular GBS graph of groups.  Then the action of $\pi_1(\GG,v)$ on $v\partial X_\GG$ is topologically free if and only if $\GG$ is not unimodular.
	
\end{cor}

\begin{proof}
	The finiteness of $\Gamma$, together with the nonsingularity of $\GG$, imply that $\GG$ has no constant trees.  The corollary now follows from Theorem \ref{thm: gbs top free} and Remark~\ref{rmk: unimodular}.
\end{proof}

In fact, the proof of Theorem \ref{thm: gbs top free} can be used to get a bit more.

\begin{prop} \label{prop: ineffective gbs}
	Let $\GG$ be a locally finite GBS graph of groups.  Suppose that condition \eqref{thm: gbs top free.1} of Theorem \ref{thm: gbs top free} holds, but that condition \eqref{thm: gbs top free.2} of Theorem \ref{thm: gbs top free} does not hold.  Then $\pi_1(\GG,v)$ does not act effectively on $v\partial X_\GG$ (and hence the action is not topologically free).
\end{prop}

\begin{proof}
	The last part of the proof of Theorem \ref{thm: gbs top free} shows that for every $\xi \in \partial W_\GG$, there is $C_\xi \in G_{r_\xi} \setminus \{0\} \cong \Z^*$ such that $C_\xi \xi = \xi$. Fix $v \in \Gamma^0$.  For $n \in \Z_+^*$ let $$S_n = \{ \xi \in v \partial W_\GG : n \text{ generates the isotropy at } \xi \}.$$  Then $v \partial W_\GG = \bigsqcup_n S_n$.  Since each $S_n$ is an open set, compactness of $v \partial W_\GG$ implies that only finitely many $S_n$ are nonempty.  Let $C$ be the least common multiple of $\{n : S_n \not= \emptyset\}$.  Then $C \xi = \xi$ for all $\xi \in v \partial W_\GG$.  Then $C \Z \subseteq G_v \subseteq \pi_1(\GG,v)$ acts trivially on $v\partial X_\GG$.  Thus $\pi_1(\GG,v)$ does not act effectively on $v\partial X_\GG$. 
\end{proof}

\begin{rmk}\label{rem:CondOf7.5Necessary}
The assumption that Theorem \ref{thm: gbs top free}\eqref{thm: gbs top free.1} holds is necessary for the proof of Proposition \ref{prop: ineffective gbs}. Let $\Gamma$ be a tree with $\Gamma^0 = \N \times \N$.  The edges $\Gamma^1$ are $\{ e_n : n \geq 0\} \cup \{ f_{ij} : i,j \ge 0\}$.  Let $r(e_n) = (n,0)$ and $s(e_n) = (n+1,0)$ for $n \ge 0$, and let $r(f_{ij}) = (i,j)$ and $s(f_{ij}) = (i,j+1)$ for $i$, $j \ge 0$.  Let $\omega_{e_n} = \omega_{\overline{e_n}} = \omega_{\overline{f_{ij}}} = 1$ for all $n$, $i$, $j$.  Let $\omega_{f_{ij}} = 1$ if $j > 0$, and $\omega_{f_{i0}} = i + 2$ for $i \ge 0$.  We show that the fundamental group acts effectively on the boundary. Let $v = (0,0)$.  Note that there are no reduced $\GG$-loops at $v$ of length greater than zero; i.e. $\pi_1(\GG,v) = G_v \cong \Z$.  For $n \ge 0$ and $m \in \Sigma_{f_{n0}} \cong \Z / (n+2)\Z$, let $\xi_{n,m} \in v\partial X_\GG$ be defined by $\xi_{n,m} = 0 e_0 \cdots 0 e_{n-1} m f_{n0} 0 f_{n1} 0 f_{n2} \cdots$.  Then it is easily seen that $1 \in \pi_1(\GG,v)$ acts on $\xi_{n,m}$ as $1 \cdot \xi_{n,m} = \xi_{n,m+1}$.  Thus no nonzero integer can fix all $\xi_{n,m}$.  Therefore $\pi_1(\GG,v)$ acts effectively on $v\partial X_\GG$, but Theorem \ref{thm: gbs top free}\eqref{thm: gbs top free.2} fails.
\end{rmk}

\subsection{$C^*$-algebras associated to GBS graphs of groups}\label{subsec: gbs C*s}

Through our work in Sections~\ref{sec: structural props}~and~\ref{sec: GBS} we have a characterisation of when the action of the fundamental group of a GBS graph of groups on its boundary is minimal and topologically free. Since we know from Theorem~\ref{thm: nuclearity of C*(GG)} that the action is amenable, this is equivalent to a characterisation of when a GBS graph of groups $C^*$-algebra is simple and nuclear. It further follows from our results that we have the following dichotomy for simple GBS graph of groups $C^*$-algebras: 

\begin{thm}\label{thm: gbs C* dichotomy}
A simple GBS graph of groups $C^*$-algebra is either a stable Bunce--Deddens algebra, or a Kirchberg algebra.
\end{thm}

\begin{proof}
Let $\GG$ be the GBS graph of groups. We know from Theorem~\ref{thm: nuclearity of C*(GG)} that $C^*(\GG)$ is nuclear. Since the action is minimal, we know that we either have (1), (2) or (3) from Proposition~\ref{prop: min and lc}. Since the action is topologically free, we know from Theorem~\ref{thm: gbs top free} that $\GG$ is not the finite ray in (3). In case (1) we have the crossed product purely infinite. Since $v\partial X_\GG$ is second countable and $\pi_1(\GG,v)$ is countable, we know that the crossed product is separable (see the discussion in Section~\ref{subsubsec:CrossedProducts}). Hence in case (1) we have $C^*(\GG)$ a Kirchberg algebra. We see from Example~\ref{example: subodometers} and Remark~\ref{rmk: normal and odometer} that the infinite ray in case (2) is an example of an odometer of groups. We see from Remark~\ref{rmk: normal and odometer} that topological freeness forces $\cap_{i=1}^\infty G_{r(e_i)}=\{0\}$ which forces $[G_{r(e_i)}:G_{r(e_{i-1})}]>1$ for infinitely many $i$. So we can apply Proposition~\ref{prop: classic odometer} to see that $C^*(\GG)$ in case (2) is a stable Bunce--Deddens algebra.
\end{proof}
  
\begin{cor}\label{cor: bs C*s}
Suppose $\GG$ is a GBS loop of groups, whose two monomorphisms are given by multiplication by integers $m$ and $n$. Then $C^*(\GG)$ is a Kirchberg algebra if and only if $|m|\not =|n|$ and $|m|,|n|\ge 2$. 
\end{cor}

\begin{proof}
Theorem~\ref{thm: gbs C* dichotomy} says that it suffices to characterise when the action is minimal and topologically free, for then, since $\GG$ is obviously not an odometer of groups, we must have $C^*(\GG)$ a Kirchberg algebra. We know from Theorem~\ref{thm: nonminimality second take} that the action will be minimal if and only if neither of the monomorphisms are surjective, which is exactly when $|m|,|n|\ge 2$. Recall from Remark~\ref{rmk: unimodular} that $\GG$ is unimodular if $q(\gamma) = \pm 1$ for every $\GG$-loop $\gamma$. Since this happens precisely when $|m|=|n|$, we see from Corollary~\ref{cor: t.f. and unimodular} that the action is topologically free if and only if $|m|\not=|n|$.
\end{proof}


\begin{thebibliography}{20}
	
	\bibitem{Adams} S. Adams, \textit{Boundary amenability for word hyperbolic groups and an application to smooth dynamics of simple groups}, Topology \textbf{33} (1994), 765--783.
	
	\bibitem{A} C. Anantharaman-Delaroche, \textit{Syst\`emes dynamiques non commutatifs et moyenabilit\'e}, Math. Ann. \textbf{279} (1987), 297--315.
	
	\bibitem{A2} C. Anantharaman-Delaroche, {\em Purely infinite $C^*$-algebras arising from dynamical systems}, Bull. Soc. Math. France \textbf{125} (1997), 199--225.
	
	\bibitem{AR} C. Anantharaman-Delaroche and J. Renault, Amenable Groupoids, L'Enseignement Math\'ematique, Gen\`eve 2000.
	
	\bibitem{AMS} Y. Antol\'{i}n, A. Martino and I. Schwabrow, {\em Kurosh rank of intersections of subgroups of free products of right-orderable groups}, Math. Res. Lett. \textbf{21} (2014), 649--661.
	
	\bibitem{AS} R. Archbold and J. Spielberg, {\em Topologically free actions and ideals in discrete $C^*$-dynamical systems}, Proc. Edinburgh Math. Soc. (2) \textbf{37} (1994), 119--124.
	
	\bibitem{Bass} H. Bass, {\em Covering theory for graphs of groups}, J. Pure Appl. Algebra, {\bf 89} (1991), 3--47.
	
	
	\bibitem{BK} H. Bass and  R. Kulkarni, {\em Uniform Tree Lattices}, J. Amer. Math. Soc., {\bf 3} (1990), 843--902.
	
	\bibitem{BL}  H. Bass and A. Lubotzky, \emph{Tree Lattices}. With appendices by Bass, L. Carbone, Lubotzky, G. Rosenberg and J. Tits. Progress in Mathematics, 176. Birkh\"auser Boston, Inc., Boston, MA, 2001.

	\bibitem{Bla} B. Blackadar, {\em Shape theory for $C^*$-algebras}, Math. Scand.  {\bf 56} (1985), 249--275.
	
	\bibitem{Bla2} B. Blackadar, Operator Algebras, Encyclopaedia of Mathematics \textbf{122}, Springer-Verlag, Berlin-Heidelberg 2006.
	
	\bibitem{BH} M.R. Bridson and A. Haefliger, Metric Spaces of Non-Positive Curvature, Springer-Verlag, Berlin-New York, 1999.
	
			
	\bibitem{BAP} A. Broise-Alamichel and F. Paulin, \emph{Sur le codage du flot g\'eod\'esique dans un arbre},  Ann. Fac. Sci. Toulouse Math. (6) \textbf{16} (2007), 477--527.
			
			
	\bibitem{BO} N. Brown and N. Ozawa, $C^*$-algebras and Finite Dimensional Approximations, Graduate Studies in Mathematics \textbf{88}, American Math. Soc., Providence, 2008.
	

	
	
	\bibitem{Choi} M.D. Choi, \textit{A simple $C^*$-algebra generated by two finite-order unitaries}, Canad. J. Math \textbf{31} (1979), 867--880.
	
	\bibitem{CF} M. Clay and M. Forester, {\em On the isomorphism problem for generalized Baumslag--Solitar groups}, Algebr. Geom. Topol. {\bf 8} (2008), 2289--2322.
	
	\bibitem{CLM} G. Cornelissen, O. Lorscheid and M. Marcolli, {\em On the $K$-theory of graph $C^*$-algebras}, Acta. Appl. Math., {\bf 102} (2008), 57--69.
	
	\bibitem{CP} M. Cortez and S. Petite, {\em $G$-odometers and their almost one-to-one extensions}, J. London Math. Soc., {\bf 78} (2008), 1--20.
	
	\bibitem{CK} J. Cuntz and W. Krieger,\ {\em A class of $C^*$-algebras and topological Markov chains}, Invent.\ Math.,\ {\bf 56} (1980), 251--268.
	
	\bibitem{D} K. Davidson, $C^*$-algebras by Example, Amer. Math. Soc., 1996, Providence.
	
	\bibitem{dlHP} P. de la Harpe and J.-P. Pr\'eaux, \emph{$C^*$-simple groups: amalgamated free products, HNN extensions, and fundamental groups of 3-manifolds},  J. Topol. Anal. \textbf{3} (2011), 451--489.
	
	
	\bibitem{F} M. Forester, {\em Splittings of generalized Baumslag--Solitar groups}, 
Geom. Dedicata {\bf 121} (2006), 43--59.

\bibitem{hatcher} A. Hatcher, \emph{Algebraic Topology}, Cambridge University Press, Cambridge, 2002.

	
	\bibitem{H} P.J. Higgins, \textit{The fundamental groupoid of a graph of groups}, J. London Math. Soc. \textbf{13} (1976), 145--149.
	
	\bibitem{II} N. Ivankov and N. Iyudu, \textit{Complete characterization of $K$-theory for $C^*$-algebras associated to locally finite unoriented graphs}, arxiv OA 1512.08604v1.
	
	
	\bibitem{I} N. Iyudu, {\em $K$-theory of locally finite graph $C^*$-algebras}, Journal of Geometry and Physics, \textbf{71} (2013), 22--29.
	
	\bibitem{KB} I. Kapovich and N. Benakli, Boundaries of hyperbolic groups, Combinatorial and geometric group theory, (New York, 2000/Hoboken NJ 2001), 39--94 Contemp. Math \textbf{296}, 2002.
	
	\bibitem{KT} S. Kawamura and J. Tomiyama, {\em Properties of topological dynamical systems and corresponding $C^â$-algebras}, 
	Tokyo J. Math. \textbf{13} (1990), 251--257. 
	
	\bibitem{K} P.H. Kropholler, {\em Baumslag--Solitar groups and some other groups of cohomological dimension two}, Comment. Math. Helv. {\bf 65} (1990), 547--558. 
	
	\bibitem{KMRW} A. Kumjian, P. Muhly, J. Renault and D. Williams, \textit{The Brauer group of a locally compact groupoid}, Amer. J. Math. \textbf{120} (1998), 901--954.
	
	\bibitem{KPR} A. Kumjian, D. Pask and I. Raeburn, {\em Cuntz--Krieger algebras of directed graphs}, Pacific J. Math. {\bf 184} (1998), 161--174.
	
	\bibitem{LS} M. Laca and J. Spielberg, \textit{Purely infinite $C^*$-algebras from boundary actions of discrete groups}, J. Reine Angew. Math. \textbf{480} (1996), 125--139.
	
	\bibitem{MRW} P. Muhly, J. Renault and D. Williams, \textit{Equivalence and isomorphism for groupoid $C^*$-algebras}, J. Operator Theory \textbf{17} (1987), 3--22.
	
	\bibitem{Mur} G. J. Murphy, \textit{$C^*$-algebras and Operator Theory}, Academic Press, San Diego, 1990.
	
	\bibitem{O} R. Okayasu, \textit{$C^*$-algebras associated with the fundamental groups of graphs of groups}, Math. Scand., \textbf{97} (2005), 49--72.
	
	\bibitem {Orfanos} S. Orfanos, \textit{Generalized Bunce--Deddens algebras}, Proc. Amer. Math. Soc. {\bf 138} (2010), 299--308.
	
	\bibitem{CBMS} I. Raeburn, Graph algebras, Published for the Conference Board of the Mathematical Sciences, Washington, DC, 2005, vi+113.
	
	\bibitem{Ren} J. Renault, A Groupoid Approach to $C^*$-algebras, Lecture Notes in Math., Springer-Verlag, Berlin-Heidelberg-New York, {\bf 793}, 1980.
	
	\bibitem{Rob} G. Robertson, \textit{Boundary operator algebras for free uniform tree lattices}, Houston J. Math., \textbf{31} (2005), 913--935.
	
	\bibitem{Rob2} G. Robertson, \textit{Boundary $C^*$-algebras for acylindrical groups}, Proc. Amer. Math. Soc. {\bf 136} (2008), 3851--3860.
	
	\bibitem{Ror} M. R\o rdam, Classification of Nuclear $C^*$-algebras, Encyclopaedia of Mathematics \textbf{126}, Springer-Verlag, Berlin-Heidelberg 2002.
	
	\bibitem{Se} J-P. Serre, Trees, Springer-Verlag, Berlin-Heidelberg-New York, 1980.
	
	\bibitem{Sp} J. Spielberg, \textit{Free product groups, Cuntz--Krieger algebras, and covariant maps}, Internat. J. Math. \textbf{2} (1991), 457--476.
	
	\bibitem{Sz} W. Szymanski, \textit{The range of K-invariants for $C^*$-algebras of infinite graphs}, Indiana Univ. Math. J., vol. 51, (2002), 239--249.
	
	\bibitem{T} J-L. Tu, {\em La conjecture de Baum-Connes pour les feuilletages moyennables}, K-Theory {\bf 17} (1999), 215--264.
	
	\bibitem{Z} R. Zimmer, Ergodic Theory and Semisimple Groups, Birkh\"auser, Boston, 1984.
	
\end{thebibliography}
\end{document}